\documentclass{amsart}
\usepackage{amssymb,stmaryrd}

\theoremstyle{plain}
\newtheorem{thm}{Theorem}[section]
\newtheorem{claim}{Claim}[thm]
\newtheorem{lemma}[thm]{Lemma}
\newtheorem{fact}[thm]{Fact}
\newtheorem{obs}[thm]{Observation}
\newtheorem{prop}[thm]{Proposition}
\newtheorem{cor}[thm]{Corollary}
\newtheorem*{thma}{Theorem~A}
\newtheorem*{thmat}{Theorem~A$'$}
\newtheorem*{thmb}{Theorem~B}
\newtheorem*{thmc}{Theorem~C}

\theoremstyle{definition}
\newtheorem*{definition}{Definition}
\newtheorem{defn}[thm]{Definition}

\newtheorem*{q}{Question} 
 
\newtheorem{setup}[thm]{Setup} 

\theoremstyle{remark}
\newtheorem{remark}[thm]{Remark}

\DeclareMathOperator{\ssup}{ssup}

\DeclareMathOperator{\otp}{otp}
\DeclareMathOperator{\cf}{cf}
\DeclareMathOperator{\lex}{lex}
\DeclareMathOperator{\dom}{dom}
\DeclareMathOperator{\im}{Im}
\DeclareMathOperator{\Tr}{Tr}
\DeclareMathOperator{\tr}{tr}
\DeclareMathOperator{\osc}{osc}
\DeclareMathOperator{\Osc}{Osc}
\DeclareMathOperator{\fs}{FS}
\DeclareMathOperator{\acc}{acc}
\DeclareMathOperator{\card}{CARD}

\DeclareMathOperator{\ch}{CH}
\DeclareMathOperator{\U}{U}
\DeclareMathOperator{\pr}{Pr}
\DeclareMathOperator{\ext}{Extract}

\newcommand\last[2]{\eth_{#1,#2}}
\newcommand\nsrightarrow{\mathrel{\overset{\sup}{\longarrownot\longrightarrow}}}  
\newcommand\srightarrow{\mathrel{\overset{\sup}{\longrightarrow}}}  
\newcommand{\branches}{\leadsto}
\newcommand\sq{\sqsubseteq}
\newcommand\br{\blacktriangleright}
\newcommand\s{\subseteq}
\newcommand\symdiff{\mathrel{\triangle}}
\renewcommand\mid{\mathrel{|}\allowbreak}
\newcommand*\axiomfont[1]{\textsf{\textup{#1}}}
\newcommand\zfc{\axiomfont{ZFC}}

\title{Sums of triples in Abelian groups}

\author{Ido Feldman}
\address{Department of Mathematics, Bar-Ilan University, Ramat-Gan 5290002, Israel.}

\author{Assaf Rinot}
\address{Department of Mathematics, Bar-Ilan University, Ramat-Gan 5290002, Israel.}
\urladdr{http://www.assafrinot.com}


\begin{document}
\begin{abstract} Motivated by a problem in additive Ramsey theory,
we extend Todor\v{c}evi\'{c}'s partitions of three-dimensional combinatorial cubes
to handle additional three-dimensional objects.
As a corollary, we get that if the continuum hypothesis fails,
then for every Abelian group $G$ of size $\aleph_2$, there exists a coloring $c:G\rightarrow\mathbb Z$
such that for every uncountable $X\s G$ and every integer $k$, there are three distinct elements $x,y,z$ of $X$ such that $c(x+y+z)=k$.
\end{abstract}
\date{Preprint as of January 4, 2023. For the latest version, visit \textsf{http://p.assafrinot.com/57}.}
\maketitle
\section{Introduction}

By Hindman's celebrated theorem (see~\cite[Corollary 5.9]{MR2893605}),
for every partition of an infinite commutative cancellative semigroup $(G,+)$ 
into two cells $A$ and $B$, there exists an infinite subset 
$X\s G$ such that the set of its finite sums
$$\fs(X):=\{x_1+\cdots+x_n\mid x_1,\ldots,x_n\text{ are distinct elements of }X\ \&\ n\in\mathbb N\setminus2\}$$
is completely contained in $A$ or completely contained in $B$. Equivalently, for every coloring $c:G\rightarrow 2$,
there exists an infinite $X\s G$ such that $c\restriction \fs(X)$ is constant.

Hindman's theorem does not generalize to the uncountable,
as it follows from a theorem of Milliken (see \cite[Theorem~9]{MR0505558}) that the following assertion is consistent with the usual axioms of set theory: 
for every (not necessarily Abelian) group $(G,*)$ whose size is a regular uncountable cardinal,
there is a coloring $c:G\rightarrow G$ such that $c\restriction \fs_2(X)$ is \emph{onto} $G$ for every $X\s G$ of size $|G|$, where this time
$$\fs_n(X):=\{x_1*\cdots*x_n\mid x_1,\ldots,x_n\text{ are distinct elements of }X\}.$$

A few years ago, starting with a paper by Hindman, Leader and Strauss \cite{MR3696151}, the study of higher analogs of Hindman's theorem regained interest.
We mention only a few results that are relevant to this paper:
\begin{enumerate}
\item Improving upon a theorem from \cite{MR3696151},
Komj\'ath~\cite{MR3511943}, and independently Soukup and Weiss~\cite{dani-bill},
proved that there exists a coloring $c:\mathbb R\rightarrow2$ 
such that for every uncountable $X\s\mathbb R$ and every $i\in\{0,1\}$,
there are $x\neq y$ in $X$ such that $c(x+y)=i$.
\item Solving a problem of Weiss, 
Komj\'ath \cite{MR4143159} proved that there exists a coloring $c:\mathbb R\rightarrow2$ 
such that for every uncountable $X\s\mathbb R$ and every $i\in\{0,1\}$,
there are $x\neq y$ in $X$ such that $c(|x-y|)=i$.
As for dimension $d>1$, 
assuming the continuum hypothesis, 
there exists a coloring $c:\mathbb R\rightarrow2$ 
such that for every uncountable $X\s\mathbb R^d$ and every $i\in\{0,1\}$,
there are $x\neq y$ in $X$ such that $c(\lVert x-y\rVert)=i$.
\item In \cite{paper27}, Fern\'andez-Bret\'on and Rinot  
proved that there exists a coloring $c:\mathbb R\rightarrow\mathbb N$ 
such that for every $X\s\mathbb R$ of size $|\mathbb R|$ and every $i\in\mathbb N$,
there are $x\neq y$ in $X$ such that $c(x+y)=i$.
\item By \cite{paper27}, for class many cardinals $\kappa$ (including $\kappa=\aleph_n$ for every positive integer $n$),
for every commutative cancellative semigroup $G$ of size $\kappa$, there exists a coloring $c:G\rightarrow G$
such that for all $X,Y\s G$ of size $\kappa$ and every $g\in G$,
there are $x\in X$ and $y\in Y$ such that $c(x+y)=g$.\footnote{More is true, see \cite[Corollary~4.5]{paper27}.}
\item By \cite{paper27}, for every commutative
cancellative semigroup $G$, there exists a coloring $c:G\rightarrow\mathbb N$
such that $c\restriction \fs(X)$ is onto $\mathbb N$ for every uncountable $X\s G$.
It is also consistent that the same holds after replacing $\mathbb N$ by $\mathbb R$.
\end{enumerate}

Note that in the results listed in (1), (2) and (5),
the triggering set $X$ may have cardinality smaller than that of $G$,
whereas in (3) and (4), $|X|$ coincides with $|G|$. 
Another important difference is that unlike the results of (1)--(4), in (5),
no bound is asserted on the length of the sums needed to generate all the infinite colors.
This raises a natural question whose simplest instance reads as follows.
\begin{q} Suppose that $(G,+)$ is an Abelian group of size $\aleph_2$.

Must there exist a positive integer $n$ and a coloring $c:G\rightarrow\mathbb N$
such that $c\restriction \fs_n(X)$ is onto $\mathbb N$ for every uncountable $X\s G$?
\end{q}

A moment's reflection makes it clear that an affirmative answer (even for one particular group $G$) immediately implies the relation $\aleph_2\nrightarrow[\aleph_1]^n_{\aleph_0}$
from the classical study of partition relations for cardinal numbers \cite{MR202613}.
By a theorem of Erd\H{o}s and Rado, the above relation may consistently fail for $n=2$,
and it is a remarkable theorem of Todor\v{c}evi\'{c} \cite{MR1297180} that it does hold for $n=3$. 
The first main result of this paper gives a consistent extension of Todor\v{c}evi\'{c}'s theorem.

\begin{thma}\label{thma} If the continuum hypothesis fails,
then for every Abelian group $(G,+)$ of size $\aleph_2$, there exists a coloring $c:G\rightarrow\mathbb N$
such that for every uncountable $X\s G$ and every $i\in\mathbb N$, there are three distinct elements $x,y,z$ of $X$ such that $c(x+y+z)=i$.
\end{thma}

Theorem~A is not limited to Abelian groups. In fact, it works for all so-called \emph{well-behaved magmas}, as follows.

\begin{definition} A \emph{magma} is a structure $(G,*)$, where $*$ is a binary operation.
We say that it is \emph{well-behaved} iff  there exists a map $\varphi:G\rightarrow[G]^{<\omega}$ such that:\footnote{Here, $[G]^{<\omega}$ denotes the collection of all finite subsets of $G$.}
\begin{itemize}
\item $G$ is countable-to-one;
\item for all $x\neq y$ in $G$, $\varphi(x)\symdiff\varphi(y)\s \varphi(x*y)\s\varphi(x)\cup\varphi(y)$.
\end{itemize}
\end{definition}

Every infinite commutative cancellative semigroup $(G,+)$ is well-behaved (see, e.g., \cite[Lemma~2.2]{paper27}).
Also, every free group $(G,*)$ is well-behaved, as witnessed by the map that sends a word to the set of its letters.
As a third example, consider the magma appearing in result~(2) above, namely, $(\mathbb R,d)$ where $d(x,y):=|x-y|$. 
Indeed, viewing $\mathbb R$ as a $\mathbb Q$-vector space over some Hamel basis $B$, 
any $x\in\mathbb R\setminus\{0\}$
is the unique linear combination $\sum_{i\le n}q_iv_i$ of nonzero rational numbers $q_0,\ldots,q_n$, and an injective sequence $\langle v_i\mid i\le n\rangle$ of elements of $B$.
So $\varphi:\mathbb R\rightarrow[\mathbb R]^{<\omega}$ sending $x$ to the unique $\{ v_i\mid i\le n\}$ (and sending $0$ to the emptyset) is countable-to-one,
and for all $x\neq y$, $\varphi(x)\symdiff\varphi(y)\s \varphi(|x-y|)\s\varphi(x)\cup\varphi(y)$.
The full statement of Theorem~A reads as follows.
\begin{thmat} For every infinite cardinal $\mu$ such that $\mu^{<\mu}<\mu^+<2^\mu$,
for every well-behaved magma $(G,*)$ of size $\mu^{++}$,
there is a coloring $c:G\rightarrow\mathbb N$
such that for every $X\s G$ of size $\mu^+$ and every $i\in\mathbb N$, there are three distinct elements $x,y,z$ of $X$ such that $c(x*y*z)=i$.\footnote{As $*$ is not assumed to be associative, 
the claim is that we get $c(x*y*z)=i$ for both implementations of $x*y*z$.}
\end{thmat}

While not so explicit, the approach of going through well-behaved magmas is already present in \cite{paper27}.
In particular, the coloring of result~(4) attains all possible colors not only over evaluations of the form $x+y$, but also over any nontrivial $\mathbb Q$-combination of $x$ and $y$,
such as $|x-y|$. This suggests that it is possible to obtain a coloring simultaneously witnessing result (1) together with the first half of (2).
Indeed, Komj\'ath's theorems follow from the following finding (using $\theta:=\aleph_0$):

\begin{thmb} For every infinite cardinal $\theta$ such that $2^{<\theta}=\theta$,
for every set $G$ with $\theta<|G|\le 2^\theta$,
and every map $\varphi:G\rightarrow[G]^{<\omega}$,
there exists a corresponding coloring $c:G\rightarrow 2$ satisfying the following.

For every binary operation $*$ on $G$, if $\varphi$ witnesses that $(G,*)$ is well-behaved,
then for every $X\s G$ of size $\theta^+$ and every $i\in\{0,1\}$, 
there are $x\neq y$ in $X$ such that $c(x*y)=i$.
\end{thmb}

The proofs of Theorems A$'$ and B are obtained in a few steps.
As a first step, we consider a coloring principle $S_n(\kappa,\lambda,\theta)$ that is sufficient to imply that any well-behaved magma $(G,*)$ of size $\kappa$
admits a coloring $c:G\rightarrow\theta$ that takes on every possible color on $\fs_n(X)$ for every set $X\s G$ of size $\lambda$.
The next step is the introduction of an extraction principle $\ext_n(\kappa,\lambda,\ldots)$
that is sufficient for the reduction of $S_n(\kappa,\lambda,\theta)$ into a rectangular-type strengthening $\kappa\nsrightarrow[\lambda,\lambda]^n_\theta$ of the classical partition relation $\kappa\nrightarrow[\lambda]^n_\theta$.
This leaves us with two independent tasks: proving instances of $\ext_n(\kappa,\lambda,\ldots)$, and proving instances of $\kappa\nsrightarrow[\lambda,\lambda]^n_\theta$.
The harder task is the latter, and the second main result of this paper is an extension of Todor\v{c}evi\'{c}'s theorem \cite{MR1297180}
that Chang's conjecture fails iff $\omega_2\nrightarrow[\omega_1]^3_{\omega_1}$ holds.
Here $\omega_2\nrightarrow[\omega_1]^3_{\omega_1}$ is improved to $\omega_2\nsrightarrow[\omega_1,\omega_1]^3_{\omega_1}$. Specifically:
\begin{thmc} The following are equivalent:
\begin{enumerate}
\item $(\aleph_2,\aleph_1)\twoheadrightarrow(\aleph_1,\aleph_0)$ fails;
\item There exists a coloring $c:[\omega_2]^3\rightarrow\omega_1$
with the property that for all disjoint $A,B\s\omega_2$ of order-type $\omega_1$ such that $\sup(A)=\sup(B)$,
for every color $\tau<\omega_1$, there is $(\alpha,\beta,\gamma)\in[A\cup B]^3\setminus([A]^3\cup[B]^3)$ such that $c(\alpha,\beta,\gamma)=\tau$.
\end{enumerate}
\end{thmc}

\subsection{Organization of this paper}
In Section~\ref{pre}, we provide some necessary preliminaries.

In Section~\ref{trees}, we recall the definition of a weak Kurepa tree and
study related objects such as the branch spectrum $T(\mu,\theta)$. This will play a role in both getting instances of $\ext_n(\kappa,\lambda,\ldots)$
and of $\kappa\nsrightarrow[\lambda,\lambda]^n_\theta$.

In Section~\ref{sums2}, we prove that $S_n(\kappa,\lambda,\theta)$  implies that any well-behaved magma $(G,*)$ of size $\kappa$
admits a coloring with the strong properties mentioned earlier.
It is proved that in the special case of $\lambda=\kappa$, $S_2(\kappa,\lambda,\theta)$ already follows from $\kappa\nrightarrow[\lambda;\lambda]^2_\theta$,
and that, in general, $S_n(\kappa,\lambda,\theta)$ follows from
$\kappa\nsrightarrow[\lambda,\lambda]^n_\theta$ together with $\ext_n(\kappa,\lambda,\omega,\omega)$.
We then use tree combinatorics to obtain sufficient conditions for $\ext_n(\kappa,\lambda,\ldots)$ to hold.
The definitions of $\ext_n(\kappa,\lambda,\theta,\chi)$ and $\kappa\nsrightarrow[\lambda,\lambda]^n_\theta$ will be found in this section as Definitions \ref{extn} and \ref{narrowsup}.

In Section~\ref{changsection}, we prove the general case of Theorem~C in which $\aleph_2$ is substituted by the double successor of a cardinal $\mu$ satisfying $\mu^{<\mu}=\mu$. The proof is a bit long,
since the analysis goes through a division into a total of six cases and subcases.

In Section~\ref{section5}, we verify that Todor\v{c}evi\'{c}'s theorems on the correspondence between unstable sets and oscillation remains valid in the rectangular context.
We then combine it with the results of Section~\ref{changsection} and get that
${\lambda^+\nsrightarrow[\lambda,\lambda]^3_\omega}$ holds for every successor $\lambda=\mu^+$ of an infinite cardinal $\mu=\mu^{<\mu}$.

In Section~\ref{sectionTHMB}, we obtain the intended applications in additive Ramsey theory. 
Theorem A$'$ is gotten as a corollary of the results of Sections~\ref{sums2} and \ref{section5},
and Theorem~B is gotten as a corollary of a theorem asserting that $S_2(\kappa,\mu^+,2)$ holds
whenever there exists a weak $\mu$-Kurepa tree with $\kappa$-many branches.

\section{Preliminaries}\label{pre}
In this section, $\kappa,\lambda,\mu,\theta,\chi$ stand for nonzero cardinals, and $n$ stand for a positive integer.
We let $H_\kappa$ denote the collection of all sets of hereditary cardinality less than $\kappa$.
We write $[\kappa]^\lambda:=\{ A\s\kappa\mid |A|=\lambda\}$ and $[\kappa]^{<\lambda}:=\{ A\s\kappa\mid |A|<\lambda\}$.
Let $E^\kappa_\chi:=\{\alpha < \kappa \mid \cf(\alpha) = \chi\}$,
and define $E^\kappa_{\le \chi}$, $E^\kappa_{<\chi}$, $E^\kappa_{\ge \chi}$, $E^\kappa_{>\chi}$,  $E^\kappa_{\neq\chi}$ analogously.
For two distinct functions $f,g\in{}^\theta\mu$, write 
$f<_{\lex} g$ to mean that $f(\delta)<g(\delta)$ for the least $\delta<\lambda$ such that $f(\delta)\neq g(\delta)$.
For functions $f,g\in{}^{\le\theta}\mu$, we write $f\sq g$ to mean that $\dom(f)\le\dom(g)$ and $g\restriction\dom(f)=f$.

For sets of ordinals $A_1,\ldots,A_n$, we define $$A_1\circledast\cdots\circledast A_n:=\{ (\alpha_1,\ldots,\alpha_n)\in A_1\times\cdots\times A_n\mid \alpha_1<\cdots<\alpha_n\}.$$
By convention, whenever we write $(\alpha_1,\ldots,\alpha_n)\in[A]^n$ (as opposed to $\{\alpha_1,\ldots,\alpha_n\}\in[A]^n$), 
we mean that $(\alpha_1,\ldots,\alpha_n)\in A\circledast\cdots\circledast A$.

For a set of ordinals $A$, we write $\ssup(A):=\sup\{\alpha+1\mid \alpha\in A\}$,
$\acc^+(A) := \{\alpha < \ssup(A) \mid \sup(a \cap \alpha) = \alpha > 0\}$,
and $\acc(A) := A \cap \acc^+(a)$.
For two sets of ordinals $A$ and $B$, we write $A<B$ to mean that $A\times B$ coincides with $A\circledast B$.

\begin{defn}[Positive round-bracket relations, {\cite[\S3]{MR202613}}] $\kappa\rightarrow(\lambda)^n_\theta$ asserts that for every coloring $c:[\kappa]^n\rightarrow\theta$,
there exists $A\s\kappa$ of order-type $\lambda$ such that $c$ is constant over $[A]^n$.
\end{defn}

\begin{defn}[Negative square-bracket relations, {\cite[\S18]{MR202613}}] \label{hungarian} A coloring $c:[\kappa]^n\rightarrow\theta$ is said to witness:
\begin{itemize}
\item $\kappa\nrightarrow[\lambda]^n_\theta$ iff $c[[A]^n]=\theta$ for every $A\in[\kappa]^\lambda$;
\item $\kappa\nrightarrow[\lambda_1,\ldots,\lambda_n]^n_\theta$ iff $c[A_1\times\cdots\times A_n]=\theta$ for every $\langle A_i\mid 1\le i\le n\rangle\in\prod_{i=1}^n[\kappa]^{\lambda_i}$;
\item $\kappa\nrightarrow[\lambda_1;\ldots;\lambda_n]^n_\theta$ iff $c[A_1\circledast\cdots\circledast A_n]=\theta$ for every $\langle A_i\mid 1\le i\le n\rangle\in\prod_{i=1}^n[\kappa]^{\lambda_i}$.
\end{itemize}
\end{defn}

Note that $(\kappa\nrightarrow[\lambda;\ldots;\lambda]^n_\theta)\implies(\kappa\nrightarrow[\lambda,\ldots,\lambda]^n_\theta)\implies(\kappa\nrightarrow[\lambda]^n_\theta)$.

\begin{defn}[Fiber maps] Given a coloring of pairs $c:[\kappa]^2\rightarrow\theta$ and some $\beta<\kappa$, we sometimes write $c_\beta$ for the $\beta^{\text{th}}$-fiber map of $c$, that is, 
for the unique map $c_\beta:\beta\rightarrow\theta$ to satisfy $c_\beta(\alpha)=c(\alpha,\beta)$ for every $\alpha<\beta$.

We say that $c$ has \emph{injective fibers} iff $c_\beta$ is injective of every $\beta<\kappa$.
\end{defn}

\begin{defn}[\cite{paper34}]\label{uprinciple} $\U(\kappa,\mu,\theta,\chi)$ asserts the existence of a coloring $c:[\kappa]^2\rightarrow\theta$ such that  
for every $\sigma<\chi$, every pairwise disjoint subfamily $\mathcal A\s[\kappa]^{\sigma}$ of size $\kappa$,
for every $\tau<\theta$, there exists $\mathcal B\s \mathcal A$ of size $\mu$ such that $\min(c[a\times b])>\tau$ for all $a\neq b$ from $\mathcal B$.
\end{defn}
\begin{remark}
Of special interest are witnesses $c:[\kappa]^2\rightarrow\theta$ to $\U(\kappa,\mu,\theta,\chi)$ that are moreover \emph{subadditive},
i.e., satisfying that for all $\alpha<\beta<\gamma<\kappa$, the following hold:
\begin{itemize}
\item $c(\alpha,\gamma)\leq\max\{c(\alpha,\beta),c(\beta,\gamma)\}$;
\item $c(\alpha,\beta)\leq\max\{c(\alpha,\gamma),c(\beta,\gamma)\}$.
\end{itemize}	
These colorings are studied in \cite{paper36}, and they will show up here in Section~\ref{changsection}.
\end{remark}

Given a coloring $c:[\kappa]^2\rightarrow\theta$ and a subset $X\s\kappa$ of order-type $\lambda$,
we say that ``$c\restriction[X]^2$ witnesses $\U(\lambda,\mu,\theta,\chi)$'' if for the order-preserving bijection $\pi:\lambda\leftrightarrow X$,
the coloring $d:[\lambda]^2\rightarrow\theta$ defined via $d(\alpha,\beta):=c(\pi(\alpha),\pi(\beta))$
is a witness for $\U(\lambda,\mu,\theta,\chi)$.
To be able to express that this happens globally,
we introduce the following $5$-cardinal extension of the principle of Definition~\ref{uprinciple}.
\begin{defn}
$\U(\kappa, \lambda, \mu , \theta , \chi ) $ asserts the existence of a coloring $c:[\kappa ] ^2 \rightarrow \theta $ such that for every $\sigma < \chi $, 
every pairwise disjoint subfamily $\mathcal A \s [\kappa ]^ {\sigma}$ of size $\lambda$,
for every $ \tau< \theta$,
there exists $\mathcal B \s\mathcal  A $ of size $\mu $ such that $\min (c[a \times b] ) > \tau$
for all $a \neq b$ from $\mathcal B$.
\end{defn}

\begin{fact}[{\cite[Lemma~9.2.3]{MR2355670}}]\label{subadditiveupgrade} For every regular uncountable cardinal $\lambda$,
if $\U(\lambda^+,\lambda,2,\lambda,2) $ holds,
then there exists a subadditive witness to $\U(\lambda^+,\lambda,\lambda,\lambda,\omega)$.
\end{fact}

Finally, we arrive at the notion motivating this paper.
\begin{defn}\label{def27} For a magma $(G,*)$, we write $G\nrightarrow[\lambda]_{\theta}^{\fs_n}$ to assert that there exists a coloring $c:G\rightarrow\theta$
with the property that for every subset $A\s G$ of size $\lambda$ and every prescribed color $\tau<\theta$,
there is an injective sequence $\langle a_i\mid 1\le i\le n\rangle$ of elements of $A$ such that 
$c(a_1*\cdots* a_n)=\tau$ for \emph{all} implementations of $a_1*\cdots *a_n$.\footnote{The issue of implementation arises from the fact that we do not assume $*$ to be associative, e.g.,  it is possible that $(a_1*a_2)*a_3\neq a_1*(a_2*a_3)$.}
\end{defn}

In the special case of $n=2$, \cite[Corollary~4.5]{paper27} and \cite[Corollary~2.20]{paper44} provide sufficient conditions for $G\nrightarrow[\lambda]_{\theta}^{\fs_n}$ to follow from $|G|\nrightarrow[\lambda]_{\theta}^n$ for all values of $\theta$.
Higher dimensional reductions are out of reach at present.

\subsection{Walks on ordinals} In this subsection, we provide a minimal background on walks on ordinals.
This background is only necessary for Section~\ref{section5},
hence the exposition here is quite succinct. A thorough treatment may be found in \cite{MR2355670}.

\medskip

For the rest of this subsection, $\kappa$ denotes a regular uncountable cardinal, and we fix some $C$-sequence over $\kappa$,
that is, a sequence $\vec C=\langle C_\beta\mid\beta<\kappa\rangle$ such that, for every $\beta<\kappa$, $C_\beta$ is closed subset of $\beta$ with $\sup(C_\beta)=\sup(\beta)$.

\begin{defn}[Todor{\v{c}}evi{\'c}]\label{defn21} From $\vec C$, derive maps $\Tr:[\kappa]^2\rightarrow{}^\omega\kappa$,
$\rho_2:[\kappa]^2\rightarrow	\omega$, and
$\tr:[\kappa]^2\rightarrow{}^{<\omega}\kappa$, by letting for all $\alpha<\beta<\kappa$:
\begin{itemize}
\item $\Tr(\alpha,\beta):\omega\rightarrow\kappa$ is defined by recursion on $n<\omega$:
$$\Tr(\alpha,\beta)(n):=\begin{cases}
\beta,&n=0\\
\min(C_{\Tr(\alpha,\beta)(n-1)}\setminus\alpha),&n>0\ \&\ \Tr(\alpha,\beta)(n-1)>\alpha\\
\alpha,&\text{otherwise}
\end{cases}
$$
\item $\rho_2(\alpha,\beta):=\min\{l<\omega\mid \Tr(\alpha,\beta)(l)=\alpha\}$;
\item $\tr(\alpha,\beta):=\Tr(\alpha,\beta)\restriction \rho_2(\alpha,\beta)$.
\end{itemize}
\end{defn}

To explain: Given a pair of ordinals $\alpha<\beta$ below $\kappa$,
one would like to \emph{walk} from $\beta$ down to $\alpha$. This is done by recursion, letting $\beta_0:=\beta$,
and $\beta_{n+1}:=\min(C_{\beta_n}\setminus\alpha)$,
thus, obtaining an ordinal $\beta_{n+1}$ such that $\alpha\le\beta_{n+1}\le\beta_n$.
Since the ordinals are well-founded, there must exist some integer $k$ such that $\beta_{k+1}=\alpha$,
so that, the walk is $\beta=\beta_0>\beta_1>\cdots>\beta_{k+1}=\alpha$. 
This walk is recorded by $\Tr(\alpha,\beta)$, since, for every $n\le k$, we have that $\Tr(\alpha,\beta)=\beta_n$,
and for every $n>k$, we have that $\Tr(\alpha,\beta)=\alpha$. 
The length of the walk is recorded by the positive integer $\rho_2(\alpha,\beta)$. 
Now, since $\Tr(\alpha,\beta)$ is eventually constant with value $\alpha$,
its nontrivial part is those ordinals greater than $\alpha$, i.e., $\beta_0>\beta_1>\cdots>\beta_k$; this is recorded by $\tr(\alpha,\beta)$.

\begin{defn}[{\cite[Definition~2.8]{paper18}}] Define a function $\lambda_2:[\kappa]^2\rightarrow\kappa$ via
$$\lambda_2(\alpha,\beta):=\sup(\alpha\cap\{ \sup(C_\eta\cap\alpha)  \mid \eta \in \im(\tr(\alpha, \beta))\}).$$
\end{defn}

Note that $\lambda_2(\alpha,\beta)<\alpha$ whenever $0<\alpha<\beta<\kappa$,
since $\tr(\alpha, \beta)$ is a finite sequence.
\begin{fact}[{\cite[Lemma~4.7]{paper34}}]\label{fact211}
Suppose that $\lambda_2(\alpha,\beta)<\epsilon<\alpha<\beta<\kappa$.

Then $\tr(\epsilon,\beta)$ end-extends $\tr(\alpha,\beta)$, and one of the following cases holds:
\begin{enumerate}
\item $\alpha\in\im(\tr(\epsilon,\beta))$; or
\item $\alpha\in\acc(C_\eth)$ for $\eth:=\min(\im(\tr(\alpha,\beta)))$.
\end{enumerate}
\end{fact}

\begin{defn}[{\cite[Definition~2.10]{paper44}}]\label{last} For every $(\alpha,\beta)\in[\kappa]^2$, we define an ordinal $\last{\alpha}{\beta}\in[\alpha,\beta]$ via:
$$\last{\alpha}{\beta}:=\begin{cases}
\min(\im(\tr(\alpha,\beta))),&\alpha\in\acc(C_{\min(\im(\tr(\alpha,\beta)))});\\
\alpha,&\text{otherwise};\\
\end{cases}$$
\end{defn}

\begin{remark}\label{concatenation}
It is easy to see that $\sup(C_{\last{\alpha}{\beta}})=\sup(\alpha)$ for all $\alpha<\beta<\kappa$, and it follows from Fact~\ref{fact211} that 
$$\tr(\epsilon,\beta)=\tr(\last{\alpha}{\beta},\beta){}^\smallfrown\tr(\epsilon,\last{\alpha}{\beta}),$$
whenever $\lambda_2(\alpha,\beta)<\epsilon<\alpha<\beta<\kappa$.
\end{remark}

\begin{fact}[Todor\v{c}evi\'{c}, {\cite[\S9]{MR2355670}}]\label{all the rho}
If $\kappa=\lambda^+$ for a regular cardinal $\lambda$ and $\otp(C_\beta)\le\lambda$ for all $\beta<\kappa$,
then there exists a subadditive coloring $\rho:[\kappa]^2\rightarrow\lambda$ with the property that $\rho(\alpha,\beta)\ge\otp(C_\eta\cap\alpha)$ for all $\alpha<\beta<\kappa$ and $\eta\in\im(\tr(\alpha,\beta))$.
\end{fact}

\section{Weak Kurepa trees and the branch spectrum}\label{trees}
In this section, $\mu$ denotes a cardinal and $\theta$ denotes an ordinal.

\begin{defn} $\mathcal T(\mu,\theta)$ denotes the collection of all subsets $T\s{}^{<\theta}\mu$
such that the following two hold:
\begin{enumerate}
\item $T$ is downward-closed, i.e, for every $t\in T$, $\{ t\restriction \alpha\mid \alpha<\theta\}\s T$;
\item for every $\alpha<\theta$, the set
$T_\alpha:=T\cap{}^\alpha\mu$ is nonempty and has size $<\mu$.
\end{enumerate}
\end{defn}

We say that $T$ is a \emph{tree of height $\theta$}
if there exists a cardinal $\mu$ such that $T\in\mathcal T(\mu,\theta)$.\footnote{There is no loss of generality here, see \cite[Lemma~2.5(2)]{paper23}.}
Note that $\theta$ is uniquely determined.
For such a tree $T$, we shall refer to $T_\alpha$ as the \emph{$\alpha^{\text{th}}$-level} of $T$,
and the set $\{ b\in{}^\theta\mu\mid \forall\alpha<\theta\,(b\restriction\alpha\in T_\alpha)\}$ of all branches through $T$ is denoted by $\mathcal B(T)$.
Also, for all $f,g\in{}^{\le\mu}\theta$, we let
$$\Delta(f,g):=\begin{cases}
\min\{\delta\in\dom(f)\cap\dom(g)\mid f(\delta)\neq g(\delta)\},&\text{if }f\nsubseteq g\ \&\ g\nsubseteq f;\\
\min\{\dom(f),\dom(g)\},&\text{otherwise}.\end{cases}$$

\begin{defn} $T\in T(\mu,\theta)$ is said to be \emph{normal} iff for all $\alpha<\beta<\theta$
and $t\in T_\alpha$, there exists $t'\in T_\beta$ with $t\sq t'$.
\end{defn}

\begin{defn} Given a tree $T$ and a subset $B\s\mathcal B(T)$, we consider the subtree:
$$T^{\branches B}:=\{ t\in T\mid | \{ b\in B\mid t\sq b\}|=|B|\}.$$
\end{defn}

\begin{lemma}\label{separation lemma2} Suppose that $T\in\mathcal T(\mu,\theta)$, and $\lambda$ is an infinite regular cardinal.
\begin{enumerate}
\item If $\lambda\ge\mu$, then for every $B\in[\mathcal B(T)]^\lambda$, $T^{\branches B}$ is in $\mathcal T(\mu,\theta)$ and is normal;
\item If $\lambda\ge\max\{\mu,|\theta|^+\}$, then for all $A,B\in[\mathcal B(T)]^\lambda$,
there are $s\in T$ and $i\neq i'$ such that $s{}^\smallfrown\langle i\rangle\in T^{\branches A}$
and $s{}^\smallfrown\langle i'\rangle\in T^{\branches B}$.
\end{enumerate}
\end{lemma}
\begin{proof} (1) Suppose that $B\in[\mathcal B(T)]^\lambda$ and $\lambda\ge\mu$.
It is clear that $\emptyset\in T^{\branches B}$. Thus, to prove that $T^{\branches B}$ has height $\theta$ and is normal,
let $\alpha<\beta<\theta$ and $t\in(T^{\branches B})_\alpha$, and we shall show that there exists $t'\in(T^{\branches B})_\beta$ extending $t$.

By the choice of $t$, $B':=\{ b\in B\mid t\sq b\}|$ has size $\lambda$.
Since $T\in\mathcal T(\mu,\theta)$,  it is the case that $0<|T_\beta|<|B'|=\cf(|B'|)$,
and then the pigeonhole principle provides $t'\in T_\beta$ such that $\{ b\in B'\mid t'\sq b\}$ has size $\lambda$.
Evidently, $t'$ is as sought.

(2) Suppose that $A,B\in[\mathcal B(T)]^\lambda$ and $\lambda\ge\max\{\mu,|\theta|^+\}$.
By possibly passing to $\lambda$-sized subsets of $A$ and $B$, we may assume that $A\cap B=\emptyset$.
Let $\langle a_j\mid j<\lambda\rangle$ be some injective enumeration of $A$,
and likewise let $\langle b_j\mid j<\lambda\rangle$ be some injective enumeration of $B$.
For each $j<\lambda$, as $a_j\neq b_j$, we may let $\delta_j:=\Delta(a_j,b_j)+1$.
As $\lambda$ is a regular cardinal greater than $|\theta|$, we may fix some $J\in[\lambda]^\lambda$ on which the map $j\mapsto \delta_j$ is constant with value, say, $\delta$.
As $T_{\delta+1}$ has size $<\mu\le\lambda$, we may moreover assume that the map $j\mapsto((a_j\restriction{\delta+1}),(b_j\restriction{\delta+1}))$ is constant over $J$,
with value, say, $(s{}^\smallfrown\langle i\rangle,s{}^\smallfrown\langle i'\rangle)$.
Then, we are done.
\end{proof} 

\begin{cor} \label{usefullemma} Suppose that $T\in\mathcal T(\lambda,\theta)\cap\mathcal P({}^{<\theta}2)$, where $\lambda=\cf(\lambda)>\cf(\theta)\ge\omega$.
Suppose that we are given $i<2$ and $X\in[\mathcal B(T)]^{\lambda}$.
Then, for $\lambda$-many $x\in X$, there are cofinally many $\delta<\theta$ such that the following two hold:
\begin{enumerate}
\item $x(\delta)=i$;
\item $\{ y\in X\mid \Delta(x,y)=\delta\}$ has size $\lambda$.
\end{enumerate}
\end{cor}
\begin{proof} Suppose not. In particular, the set $Y$ of all $x\in X$ for which there are boundedly many $\delta<\theta$ satisfying Clauses (1) and (2) has size $\lambda$.
So, for each $x\in Y$, the following ordinal is smaller than $\theta$:
$$\epsilon_x:=\sup\{\delta<\theta\mid x(\delta)=i\ \&\ |\{ y\in X\mid \Delta(x,y)=\delta\}|=\lambda\}.$$
As $|Y|=\cf(\lambda)>\cf(\theta)$, we may find some $\epsilon<\theta$ such that $Z:=\{ x\in Y\mid \epsilon_x=\epsilon\}$ has size $\lambda$.
As $|T_{\epsilon+1}|<\lambda$, we may also find some $t\in T_{\epsilon+1}$ such that $Z_t:=\{ x\in Z\mid t\sq x\}$ has size $\lambda$.
Now, by appealing to Lemma~\ref{separation lemma2}(2) with $\mu:=\lambda$, $A:=Z_t$ and $B:=Z_t$, we may find $s \in T$ and $j<2 $ such that 
$\hat A:=\{ a\in A\mid s{}^\smallfrown\langle j\rangle\sq a\}$ and $\hat B:=\{ b\in B\mid s{}^\smallfrown\langle1-j\rangle\sq b\}$ are both of size $\lambda$.
As $A=B$ and by possibly switching the roles of $\hat A$ and $\hat B$, we may assume that $j=i$.
Denote $\delta:=\dom(s)$.
For all $a\in \hat A$ and $b\in \hat B$,
since $a,b\in Z_t$, both $s{}^\smallfrown\langle i\rangle$  and $s{}^\smallfrown\langle1-i\rangle$ are compatible with $t$, so that $\delta=\dom(s)\ge\dom(t)>\epsilon$.
Now, for every $x\in \hat A$, it is the case that $x(\delta)=i$ and $\{ y\in X\mid \Delta(x,y)=\delta\}$ covers $\hat B$,
but $|\hat B|=\lambda$, so we got a contradiction to the fact that $\delta>\epsilon=\epsilon_x$.
\end{proof}

\begin{lemma}\label{lemma36a} Suppose that $T\in\mathcal T(\mu,\mu)$,  and $\mu$ is a regular uncountable cardinal.
Suppose also that $\langle b_\xi\mid \xi < \mu \rangle$ is an injective enumeration of some $B\in[\mathcal B(T)]^\mu$.
For every $\langle t_\alpha\mid\alpha<\mu\rangle\in\prod_{\alpha<\mu}(T^{\branches B}\cap{}^\alpha\mu)$,
for club many $\alpha<\mu$, 
$$\sup(\{\Delta(b_\beta,t_\alpha)\mid \beta<\alpha\}\cap\alpha)=\alpha.$$
In particular, for club many $\alpha<\mu$,
$$\sup\{ \gamma<\mu\mid \alpha\in\acc^+(\{\Delta(b_\beta,b_\gamma)\mid \beta<\alpha\})\}=\mu.$$
\end{lemma}
\begin{proof} The `In particular' part follows the main claim
together with Lemma~\ref{separation lemma2}(1), using $\lambda:=\mu$.
Next, to prove the main claim, let $\langle t_\alpha\mid\alpha<\mu\rangle\in\prod_{\alpha<\mu}(T^{\branches B}\cap{}^\alpha\mu)$.
Denote $\Gamma_\alpha:=\{\gamma<\mu\mid t_\alpha\sq b_\gamma\}$.
Consider the club
$$C:=\{\alpha\in\acc(\mu)\mid\forall \bar\alpha<\alpha\,[\min(\Gamma_{\bar\alpha}\setminus\bar\alpha)<\alpha]\}.$$

Note that for every $\alpha<\mu$, $D^\alpha:=\{ \Delta(b_\beta,t_\alpha)\mid \beta\in\alpha\setminus \Gamma_\alpha\}$ is a subset of $\alpha$.
\begin{claim} The following set covers a club in $\mu$:
$$A:=\{ \alpha<\mu\mid \sup(D^\alpha)=\alpha\}$$
\end{claim}
\begin{proof}
Suppose not. Fix an ordinal $\epsilon<\mu$ 
for which the following set is stationary:
$$S:=\{ \alpha\in C\mid \sup(D^\alpha)=\epsilon\}.$$

There are two cases to consider:

$\br$ Suppose that there exists a function $b:\mu\rightarrow\mu$ such that $S_b:=\{\alpha\in S\mid b\restriction\alpha=t_\alpha\}$ is cofinal in $\mu$.
Pick $\bar\alpha\in S_b\setminus(\epsilon+1)$. Since $\Gamma_{\bar\alpha}$ has more than one element,
we may now find $\beta\in \Gamma_{\bar\alpha}$ such $b\neq b_\beta$. Then $\epsilon<\bar\alpha\le\Delta(b_\beta,b)<\mu$.
Pick $\alpha\in S_b$ above $\max\{\Delta(b_\beta,b),\beta\}$. Then $\epsilon<\Delta(b_\beta,t_\alpha)<\alpha$,
contradicting the fact that $\sup(D^\alpha)=\epsilon$.

$\br$ Suppose the first case fails.
First, since $|T_{\epsilon+1}|<\mu$, 
pick a node $t\in T_{\epsilon+1}$ such that $S':=\{ \alpha\in S\mid t_\alpha\restriction(\epsilon+1)=t\}$ is stationary.
Since for every function $b:\mu\rightarrow\mu$, the set $S_b:=\{\alpha\in S\mid b\restriction\alpha=t_\alpha\}$ is bounded in $\mu$,
we may now pick a pair $({\bar\alpha},\alpha)\in S'$ such that $t_{\bar\alpha}\not\sq t_\alpha$,
so that $\epsilon<\Delta(t_{\bar\alpha},t_\alpha)<{\bar\alpha}$.
Let $\beta:=\min(\Gamma_{{\bar\alpha}}\setminus\bar\alpha)$. Then $b_\beta\restriction{\bar\alpha}=t_{{\bar\alpha}}$
and hence $\Delta(b_\beta,t_\alpha)=\Delta(t_{\bar\alpha},t_\alpha)$. That is,
$\epsilon<\Delta(b_\beta,t_\alpha)<{\bar\alpha}\le\beta<\alpha$, contradicting the fact that $\sup(D^{\alpha})=\epsilon$.
\end{proof}

Let $\alpha\in A$.
Recall that $\Gamma_\alpha$ has size $\mu$,
and note that, for every $\gamma\in \Gamma_\alpha$,
$$\{\Delta(b_\beta,b_\gamma)\mid \beta<\alpha\}\cap\alpha=\{\Delta(b_\beta,t_\alpha)\mid \beta<\alpha\}\cap\alpha=D^\alpha,$$
so we are done.
\end{proof}

\begin{lemma}\label{height lemma}\label{<lambda tree obs} Suppose that $T\in\mathcal T(2^\mu,\mu)\cap\mathcal P({}^{<\mu}\mu)$, where $\mu$ is an infinite regular cardinal.
For every $B\in[\mathcal B(T)]^\mu$,
there exist $B'\in[B]^\mu$ and $\theta\le\mu$ such that:
\begin{enumerate} 
\item For every $B''\in[B']^\mu$, $T^{\branches B''}$ is in $\mathcal T(\mu,\theta)$ and is normal;
\item If $\theta<\mu$ or if $T$ contains no $\mu$-Aronszajn subtrees, then $|\mathcal B(T^{\branches B''})|=\mu$ for every $B''\in[B']^\mu$;
\item If $\theta<\mu$, then $|\{ b\in B'\mid \ssup\{\Delta(t,b\restriction\theta)\mid t\in T^{\branches B'}\ \&\ t\not\sqsubseteq b\}<\theta\}|<\mu$.
\end{enumerate}
\end{lemma}
\begin{proof} Let $B\in[\mathcal B(T)]^\mu$.  

\begin{claim} If $T^{\branches B}\in \mathcal T(\mu,\mu)$, then the pair $(B',\theta):=(B,\mu)$ is as sought.
\end{claim}
\begin{proof} Suppose that $T^{\branches B}$ is in $\mathcal T(\mu,\mu)$.
For every $B''\in[B]^\mu$, $T^{\branches B''}=(T^{\branches B})^{\branches B''}$,
and hence Lemma~\ref{separation lemma2}(1) implies that $T^{\branches B''}\in \mathcal T(\mu,\mu)$ and is normal.
In addition, if there exists some $B''\in[B]^\mu$ such that $|\mathcal B(T^{\branches B''})|<\mu$,
then looking at $B''':=B''\setminus \mathcal B(T^{\branches B''})$,
we get that $T^{\branches B'''}$ is a $\mu$-Aronszajn subtree of $T$.
\end{proof}

From now on, suppose that $T^{\branches B}\in\mathcal T(2^\mu,\mu)\setminus \mathcal T(\mu,\mu)$.
Let $\theta<\mu$ be the least such that $(T^{\branches B})_\theta$ has size $\ge\mu$.
Let $\langle t_i\mid i<\mu\rangle$ be an injective sequence of elements of $(T^{\branches B})_\theta$.
For each $i<\mu$, pick $b_i\in B$ such that $t_i\sq b_i$, and set $B':=\{b_i\mid i<\mu\}$.
To see that the pair $(B',\theta)$ is as sought, let $B''\in[B']^\mu$.
\begin{claim} $T^{\branches B''}$ is in $\mathcal T(\mu,\theta)$ and is normal.
\end{claim}
\begin{proof} For every $\alpha\in[\theta,\mu)$, it is the case that for every $t\in(T^{\branches B''})_\alpha$,
there exists a unique $i<\mu$ such that $t_i\sq t$, and hence $\{ b\in B''\mid t\sq b\}\s\{ b_i\}$ is finite.
Therefore, $(T^{\branches B''})_\alpha$ is empty.
In addition, as $B''\s B'\s B$, 
it is the case that $|(T^{\branches B''})_\alpha|\le|(T^{\branches B})_\alpha|<\mu$ for all $\alpha<\theta$.

Clearly, $\emptyset\in T^{\branches B''}$.
Finally, let $\alpha<\beta<\mu$ with $t\in(T^{\branches B''})_\alpha$,
and we shall find $t'\in(T^{\branches B''})_\beta$ extending $t$.
By the choice of $t$, $B^*:=\{ b\in B''\mid t\sq b\}|$ has size $\mu$.
By the minimality of $\theta$, the map $b\mapsto b\restriction\beta$
from $B^*$ to $T_\beta$ cannot have an image of size $\mu$,
and hence there exists $B^{**}\in[B^*]^\mu$ on which the said map is constant, with some value, say $t'$. 
Clearly, $t'$ is as sought.
\end{proof}
\begin{claim} $|\mathcal B(T^{\branches B''})|=\mu$.
\end{claim}
\begin{proof} Suppose not. In particular, $I:=\{ i<\mu\mid b_i\in B''\ \&\ t_i\notin\mathcal B(T^{\branches B''})\}$ has size $\mu$.
It follows that there exists an $\alpha<\theta$ such that $I_\alpha:=\{ i\in I\mid (t_i\restriction\alpha)\notin T^{\branches B''}\}$ has size $\mu$.
However, $\mu$ is a regular cardinal greater than $|T_\alpha|\ge|(T^{\branches B''})_\alpha|$,
and hence there must exist some $s\in (T^{\branches B''})_\alpha$ such that $\{ i\in I_\alpha\mid (t_i\restriction\alpha)=s\}$ has size $\mu$.
This is a contradiction.
\end{proof}
\begin{claim} $|\{ i<\mu\mid \ssup\{\Delta(t,t_i)\mid t\in T^{\branches B'}\ \&\ t\not\sqsubseteq t_i\}<\theta\}|<\mu$.
\end{claim}
\begin{proof} Suppose not, and pick $\epsilon<\theta$ such that the following set has size $\mu$:
$$I:=\{ i<\mu\mid \ssup\{\Delta(t,t_i)\mid t\in T^{\branches B'}\ \&\ t\not\sqsubseteq t_i\}=\epsilon\}.$$
Then pick $s\in T_\epsilon$ such that $\{ i\in I\mid t_i\restriction\epsilon=s\}$ has size $\mu$.
Finally, as in the proof of Lemma~\ref{separation lemma2}(2), 
we may find some $s'\in T^{\branches B'}$ extending $s$ and $j\neq j'$ such that 
$\{ i\in I\mid s'{}^\smallfrown\langle j\rangle\sq t_i\}$ and $\{ i\in I\mid s'{}^\smallfrown\langle j'\rangle\sq t_i\}$ are both of size $\mu$.
In particular, there exist $i\neq i'$ in $I$ such that $t_i\cap t_{i'}=s'$. So $\Delta(s'{}^\smallfrown\langle i'\rangle,t_i)\ge\epsilon$,
contradicting the fact that $i\in I$.
\end{proof}

This completes the proof.
\end{proof}

\begin{defn} Let $\mu$ denote an infinite cardinal.
\begin{enumerate}
\item A \emph{weak $\mu$-Kurepa tree} is a tree $T$ of height $\mu$, of size $\mu$, satisfying $|\mathcal B(T)|>\mu$;
\item  A \emph{$\mu$-Kurepa tree} is a tree $T$ of height $\mu$ for which $\{\alpha<\mu\mid |T_\alpha|>|\alpha|\}$ is nonstationary, and $|\mathcal B(T)|>\mu$.
\end{enumerate}
\end{defn}
\begin{remark} As in Exercise~34 of \cite[\S II]{MR756630}, if there exists a $\mu$-Kurepa tree (resp.~weak $\mu$-Kurepa tree),
then there exists one which is a subset of ${}^{<\mu}2$.
\end{remark}

\begin{defn}[Branch spectrum] $T(\mu,\theta)$ stands for the collection of all cardinals $\kappa$ for which there exists $T\in\mathcal T(\mu,\theta)$ with $\kappa\le|\mathcal B(T)|$.
\end{defn}

\begin{prop} Let $\mu$ and $\theta$ denote infinite cardinals. Then:
\begin{enumerate}
\item $\sup(T(\mu,\theta))\le\mu^\theta$;
\item If $\theta$ is the least cardinal to satisfy $\mu^\theta>\mu$, then $\max(T(\mu^+,\theta))=(\mu^+)^\theta$;
\item If there exists a weak $\mu$-Kurepa tree, then $\mu^+\in T(\mu^+,\mu)$;
\item If there exists a $\mu$-Kurepa tree, then $\mu^+\in T(\mu,\mu)$;
\item If $\mu$ is a strong limit, then $2^\mu\in T(\mu,\cf(\mu))$.
\end{enumerate}
\end{prop}
\begin{proof} Clear.
\end{proof}
\begin{remark} $T(\mu,\theta)$ need not have a maximal element. By \cite{MR4323604}, it is consistent for $T(\omega_1,\omega_1)$ to have $\aleph_{\omega_2}$ has a supremum that is not attained.
Note, however, that $\mathcal T(\mu,\theta)$ is closed under unions of length $<\cf(\mu)$,
and hence $T(\mu,\theta)$ is ${<}\cf(\mu)$-closed.
\end{remark}

\begin{prop}\label{prop213} For every $\kappa\in T(\mu,\theta)$, 
there exists a coloring $c:[\kappa ] ^2 \rightarrow \theta $ witnessing
$\U(\kappa,\lambda,\lambda,\theta,2)$ for every regular cardinal $\lambda\in[\mu,\kappa]$.
\end{prop}
\begin{proof} Given $\kappa\in T(\mu,\theta)$, let us fix $T\in\mathcal T(\mu,\theta)$ 
admitting an injective sequence $\langle b_\xi\mid \xi<\kappa\rangle$ consisting of elements of $\mathcal B(T)$.
Define $c:[\kappa]^2\rightarrow\theta$ via $c(\alpha,\beta):=\Delta(b_\alpha,b_\beta)$.
Now, given $\tau<\theta$ and $A\in[\kappa]^\lambda$ for a regular cardinal $\lambda\in[\mu,\kappa]$, 
since $|T_{\tau+1}|<\mu$, it is possible to find $x\in T_{\tau+1}$ for which $B:=\{ \alpha\in A\mid x\sq b_\alpha\}$ has size $\lambda$.
Evidently, $c(\alpha,\beta)>\tau$ for all $\alpha\neq\beta$ from $B$.
\end{proof}

\section{Coloring well-behaved magmas}\label{sums2}

In this section, we obtain sufficient conditions for $G\nrightarrow[\lambda]_{\theta}^{\fs_n}$ to hold. 
To ease on the reader, we start with the special case of $n=2$.
The upcoming Lemma~\ref{S to groups}  reduces this case to the following simple combinatorial principle.

\begin{defn}\label{s2prin}
$S_2(\kappa,\lambda,\theta )$ asserts the existence of a coloring $d: [\kappa ] ^{<\omega}  \rightarrow \theta $ such that,  
for every $\mathcal X \s [ \kappa ] ^{ < \omega } $ of size $\lambda $ and every prescribed color $\tau<\theta$, 
there exist two distinct $x,y\in\mathcal X$ such that $d(z)=\tau$ whenever $(x\symdiff y)\s z\s (x\cup y)$.
\end{defn}  	

\begin{lemma}[{\cite[Theorem~4.7]{paper27}}]\label{S to groups} 
Suppose that $S_2(\kappa,\lambda,\theta )$ holds,
for given cardinals $\theta\le\lambda\le\kappa$ with $\lambda$ regular and uncountable.

Then $G\nrightarrow[\lambda]_{\theta}^{\fs_2}$ holds for every well-behaved magma $(G,*)$ with $|G|=\kappa$.
\end{lemma}
\begin{proof} Let $d$ be coloring witnessing $S_2(\kappa,\lambda,\theta )$.
Suppose that $(G,*)$ is a well-behaved magma with $|G|=\kappa$.
By identifying $[G]^{<\omega}$ with $[\kappa]^{<\omega}$,
we may thus fix a map $\varphi:G\rightarrow[G]^{<\omega}$ such that:
\begin{itemize}
\item $\varphi$ is ${<}\lambda$-to-one;
\item for all $x\neq y$ in $G$, $\varphi(x)\symdiff\varphi(y)\s \varphi(x*y)\s\varphi(x)\cup\varphi(y)$.
\end{itemize}

Define a coloring $c:G\rightarrow \theta$ by letting $c:=d\circ\varphi$.
To see that $c$ is as sought, let $X\in[G]^\lambda$.
As $\varphi$ is ${<}\lambda$-to-one, $\mathcal X:=\{\varphi(x)\mid x\in X\}$ has size $\lambda$.
Thus, given a prescribed color $\tau<\theta$, 
we may find $x,y\in X$ with $\varphi(x)\neq\varphi(y)$ such that $d(z)=\tau$ whenever $(\varphi(x)\symdiff\varphi(y))\s z\s (\varphi(x)\cup\varphi(y))$.
In particular, $x\neq y$ and $c(x*y)=d(\varphi(x*y))=\tau$.
\end{proof}

The question arises: How do one obtain instances of $S_2(\ldots)$? The proof of \cite[Lemma~3.4]{paper27} makes it clear that the following holds:
\begin{fact}\label{fact33} Suppose that $\lambda$ is a regular uncountable cardinal and that $\theta$ is an infinite cardinal.
Then $\pr_1(\kappa,\lambda,\theta,\omega)$ implies $S_2(\kappa,\lambda,\theta)$.
\end{fact}
\begin{remark}\label{Fleissner}
The principle $\pr_1(\kappa,\lambda,\theta,\omega)$ is a particular strengthening of $\kappa\nrightarrow[\lambda;\lambda]^2_\theta$.
Since it will not play a role in this paper, we omit its definition,
and settle for pointing out the following corollary.
By a theorem of Fleissner \cite[\S5]{MR493930},
for every regular uncountable cardinal $\kappa$, in the forcing extension for adding $\kappa$-many Cohen reals, $\pr_1(\kappa,\omega_1,\omega,\omega)$ holds.
It thus follows that if $\kappa$ is a regular cardinal $\ge\mathfrak c$,
then after adding $\kappa$-many Cohen reals, $S_2(2^{\aleph_0},\aleph_1,\aleph_0)$ holds.
\end{remark}

In case that $\lambda=\kappa$, we can now improve Fact~\ref{fact33}, as follows.
\begin{thm} Suppose that $\kappa$ is a regular uncountable cardinal and that $\theta$ is an infinite cardinal.
Then $\kappa\nrightarrow[\kappa;\kappa]^2_\theta$ implies $S_2(\kappa,\kappa,\theta)$.
\end{thm}
\begin{proof} Suppose that $c:[\kappa]^2\rightarrow\theta$ is a coloring witnessing $\kappa\nrightarrow[\kappa;\kappa]^2_\theta$.
Fix a bijection $\pi:\theta\leftrightarrow\theta\times\omega$, and then find $c_0:[\kappa]^2\rightarrow\theta$ and $c_1:[\kappa]^2\rightarrow\omega$ 
such that $\pi(c(\alpha,\beta))=(c_0(\alpha,\beta),c_1(\alpha,\beta))$ for every $(\alpha,\beta)\in[\kappa]^2$.

Define a coloring $d:[\kappa]^{<\omega}\rightarrow\theta$, as follows.
For $z\in[\kappa]^{<2}$, just let $d(z):=0$. Next, for $z\in[\kappa]^{<\omega}$ of size $\geq2$, 
first let $\langle\alpha_i\mid i<|z|\rangle$ denote the increasing enumeration of $z$,
and then let $d(z):=c_0(\alpha_{j_z},\alpha_{j_z+1})$, for
$$j_z:=\min\{j<|z|-1\mid c_1(\alpha_{j},\alpha_{j+1})=\max\{c_1(\alpha_i,\alpha_{i+1})\mid i<|z|-1\}\}.$$

To see this works, suppose that we are given a $\kappa$-sized family $\mathcal X\s[\kappa]^{<\omega}$, and a prescribed color $\tau<\theta$.
By thinning out, we may assume that $\mathcal X$ forms an head-tail-tail $\Delta$-system with some root $r$.
By further thinning out, we may assume the existence of some $n<\omega$ such that $c_1``[x]^2\s n$ for all $x\in\mathcal X$.
Split $\mathcal X$ into two $\kappa$-sized sets $\mathcal X=\mathcal X_0\cup\mathcal X_1$.
Set $A:=\{ \min(x\setminus r)\mid x\in \mathcal X_0\}$ and $B:=\{ \max(x\setminus r)\mid x\in\mathcal X_1\}$.
As $c$ witnesses $\kappa\nrightarrow[\kappa;\kappa]^2_\theta$, fix $(\alpha,\beta)\in A\circledast B$ such that $c(\alpha,\beta)=\pi^{-1}(\tau,n)$.
Pick the unique $x,y\in\mathcal X$ such that $\alpha=\min(x\setminus r)$ and $\beta=\max(x\setminus r)$.
As $\mathcal X_0\cap\mathcal X_1=\emptyset$, $x\neq y$. Consequently, $x\setminus r < y\setminus r$.
Now fix an arbitrary set $z$ such that $(x\symdiff y)\s z\s (x\cup y)$. Clearly $|z|\ge2$.
Let $\langle\alpha_i\mid i<|z|\rangle$ denote the increasing enumeration of $z$.
For every $i<|z|$, if $\{\alpha_i,\alpha_{i+1}\}\s x$ then $c_1(\alpha_i,\alpha_{i+1})<n$,
and likewise, if $\{\alpha_i,\alpha_{i+1}\}\s y$ then $c_1(\alpha_i,\alpha_{i+1})<n$.
As $x\setminus r < y\setminus r$ and $\{\alpha,\beta\}\s z$, it follows that there exists $j<|z|$ such that $a_j=\alpha$ and $\alpha_{j+1}=\beta$.
For this $j$, we would have $c_1(\alpha_j,\alpha_{j+1})=c_1(\alpha,\beta)=n$. Altogether, $d(z)=c_0(\alpha,\beta)=\tau$, as sought.
\end{proof}

As for the general case (i.e., $\lambda\le\kappa$), we now present an extraction principle that is sufficient to derive $S_2(\kappa,\lambda,\theta )$ from $\kappa\nrightarrow[\lambda;\lambda]^2_\theta$.
Roughly speaking, the upcoming principle asserts the existence of a map $e:[\kappa]^{<\omega}\rightarrow[\kappa]^2$ that, in some scenarios, manages to extract two distinguished points $e(z)$ from any given set $z\in[\kappa]^{<\omega}$.

\begin{defn}\label{ext2} $\ext_2(\kappa,\lambda,\theta,\chi)$ asserts the existence of a map $e:[\kappa]^{<\omega}\rightarrow[\kappa]^2$
satisfying that for every sequence $\langle x_\gamma\mid\gamma<\lambda\rangle$ of subsets of $\kappa$,
every $r\in[\kappa]^{<\theta}$, and every nonzero $\sigma<\chi$ such that:
\begin{enumerate}
\item for every $(\gamma,\gamma')\in[\lambda]^2$, $x_\gamma\cap x_{\gamma'}\s r$;
\item for every $\gamma<\lambda$, $y_\gamma:=x_\gamma\setminus r$ has order-type $\sigma$,
\end{enumerate}
there exist $j<\sigma$ and disjoint cofinal subsets $\Gamma_0,\Gamma_1$ of $\lambda$ satisfying the following:
\begin{enumerate}
\item[(a)] For every $(\gamma,\gamma')\in[\Gamma_0\cup\Gamma_1]^2$, $y_\gamma(j)<y_{\gamma'}(j)$;
\item[(b)] For every $(\gamma,\gamma')\in(\Gamma_0 \circledast \Gamma_1)\cup(\Gamma_1 \circledast \Gamma_0)$,
for every  $z\in[x_{\gamma} \cup  x_{\gamma'}]^{<\omega}$ covering $\{y_\gamma(j),y_{\gamma'}(j)\}$,
we have $$e(z)=(y_\gamma(j),y_{\gamma'}(j)).$$
\end{enumerate}
\end{defn}
\begin{remark} Without loss of generality, we may assume that $e(z)\in[z]^2$ for every $z\in[\kappa]^{<\omega}$ of size $\ge 2$.
Also note that $\ext_2(\kappa,\kappa,\cf(\kappa),2)$ is a theorem of $\zfc$.
\end{remark}

\begin{lemma}\label{lemma49}  Suppose that $\lambda$ is a regular uncountable cardinal and that $\theta$ is an arbitrary cardinal.
If $\kappa\nrightarrow[\lambda;\lambda]^2_\theta$ and $\ext_2(\kappa,\lambda,\omega,\omega)$ both hold,
then so does $S_2(\kappa,\lambda,\theta )$.
\end{lemma}
\begin{proof} Suppose that $c:[\kappa]^2\rightarrow\theta$ is a witness for $\kappa\nrightarrow[\lambda;\lambda]^2_\theta$,
and that $e:[\kappa]^{<\omega}\rightarrow[\kappa]^2$ is a witness for $\ext_2(\kappa,\lambda,\omega,\omega)$.
We claim that $d:=c\circ e$ is a witness for $S_2(\kappa,\lambda,\theta)$.
To this end, suppose that we are given a subfamily $\mathcal X \s [ \kappa ] ^{ < \omega } $ of size $\lambda$,
and a prescribed color $\tau<\theta$.
As $\lambda$ is regular and uncountable,  by the $\Delta$-system lemma,
we may find a sequence $\langle x_\gamma\mid\gamma<\lambda\rangle$ consisting of elements of $\mathcal X$,
some $r\in[\kappa]^{<\omega}$, and a nonzero $\sigma<\chi$ such that:
\begin{enumerate}
\item for every $(\gamma,\gamma')\in[\lambda]^2$, $x_\gamma\cap x_{\gamma'}=r$;
\item for every $\gamma<\lambda$, $y_\gamma:=x_\gamma\setminus r$ has order-type $\sigma$.
\end{enumerate}
It thus follows from the choice of $e$
that we may pick some integer $j<\sigma$ and cofinal subsets $\Gamma_0,\Gamma_1$ of $\lambda$ satisfying the following:
\begin{enumerate}
\item[(a)] For every $(\gamma,\gamma')\in[\Gamma_0\cup\Gamma_1]^2$, $y_\gamma(j)<y_{\gamma'}(j)$;
\item[(b)] For every $(\gamma,\gamma')\in(\Gamma_0 \circledast \Gamma_1)\cup(\Gamma_1 \circledast \Gamma_0)$,
for every  $z\in[x_{\gamma} \cup  x_{\gamma'}]^{<\omega}$ covering $\{y_\gamma(j),y_{\gamma'}(j)\}$,
we have $$e(z)=(y_\gamma(j),y_{\gamma'}(j)).$$
\end{enumerate}

Put $A:=\{ y_\gamma(j)\mid \gamma\in\Gamma_0\}$ and $B:=\{ y_\gamma(j)\mid \gamma\in\Gamma_1\}$.
By the choice of $c$, we may find $(\alpha,\beta)\in A\circledast B$ such that $c(\alpha,\beta)=\tau$.
Pick $\gamma\in\Gamma_0$ such that $y_\gamma(j)=\alpha$,
and pick $\gamma'\in\Gamma_1$ such that $y_{\gamma'}(j)=\beta$. 
As $\alpha<\beta$, Clause~(a) implies that $(\gamma,\gamma')\in(\Gamma_0\circledast\Gamma_1)$.
As $x_\gamma\cap x_{\gamma'}=r$,
we infer that $\{ \alpha,\beta\}\s x_\gamma\symdiff x_{\gamma'}$.
So, for every set $z$ such that $(x_\gamma\symdiff x_{\gamma'})\s z\s (x_\gamma\cup x_{\gamma'})$,
we get that $$d(z)=c(e(z))=c(\alpha,\beta)=\tau,$$ as sought.
\end{proof}

Motivated by the preceding reduction, 
one would like to see how to get $\ext_2(\kappa,\lambda,\allowbreak\omega,\omega)$. The next lemma provides a sufficient condition.

\begin{lemma}\label{extract2} Suppose that $\kappa\in T(\mu,\theta)$.
Then there exists a map $e:[\kappa]^{<\omega}\rightarrow[\kappa]^2$
witnessing $\ext_2(\kappa,\lambda,\cf(\theta),\omega)$ for every regular cardinal $\lambda$ with $\max\{\mu,\theta^+\}\le\lambda\le\kappa$.
\end{lemma}
\begin{proof} As $\kappa\in T(\mu,\theta)$, let us fix $T\in\mathcal T(\mu,\theta)$ 
admitting an injective sequence $\langle b_\xi\mid \xi<\kappa\rangle$ consisting of elements of $\mathcal B(T)$.
For notational simplicity, we shall write $\Delta(\alpha,\beta)$ for $\Delta(b_\alpha,b_\beta)$.
First, for every $z\in[\kappa]^{<\omega}$, let:
\begin{itemize}
\item $M_z := \{ (\alpha , \beta) \in [z]^2 \mid \Delta ( \alpha, \beta ) = \max( \Delta `` [z]^2 ) \} $, and
\item $M_z^* := \{ (\alpha , \beta)\in M_z\mid \alpha=\min\{ \alpha'\mid (\alpha',\beta')\in M_z\}\}$.
\end{itemize}
Then, pick any function $e:[\kappa]^{<\omega}\rightarrow[\kappa]^2$ satisfying that for every $z\in[\kappa]^{<\omega}$:
\begin{itemize}
\item for every $z\in[\kappa]^{<\omega}$ of size $\ge 2$, $e(z)\in[z]^2$;
\item if $M_{z^*}$ is a singleton, then $e(z)$ is its unique element.
\end{itemize} 

To see that $e$ is a sought, suppose that $\lambda$ is a regular cardinal satisfying $\max\{\mu,\theta^+\}\le\lambda\le\kappa$,
and that we are given $\langle x_\gamma\mid\gamma<\lambda\rangle$, $r\in[\kappa]^{<\cf(\theta)}$ and $\sigma<\omega$ as in Definition~\ref{ext2}.
By the pigeonhole principle and the Dushnik-Miller theorem,
we may find a cofinal subset $\Gamma\s\lambda$,
an ordinal $\delta < \theta$, and a sequence $\langle t_j\mid j<\sigma\rangle$ of nodes in $T_{\delta+1}$ 
such that for every $(\gamma,\gamma')\in[\Gamma]^2$:
\begin{enumerate}
\item[(I)] $\sup( \Delta``[r \cup y_{\gamma} ]^2)=\delta$;
\item[(II)] for every $j<\sigma$, $b_{y_{\gamma} (j)} \restriction (\delta +1 )=t_j$;
\item[(III)] for every $j<\sigma$, $y_{\gamma}(j) <  y_{\gamma'}(j)$.
\end{enumerate}

\begin{claim}  There exist $\Gamma_0,\Gamma_1\in[\Gamma]^{\lambda}$ and a sequence $\langle (s_j,i_j,i_j')\mid j<\sigma\rangle$ of triples in $T\times \mu\times\mu$ such that,
for every $j<\sigma$:
\begin{itemize}
\item for every $\gamma\in\Gamma_0$, $s_j{}^\smallfrown\langle i_j\rangle\sq b_{y_\gamma(j)}$, 
\item for every $\gamma\in\Gamma_1$, $s_j{}^\smallfrown\langle i_j'\rangle\sq b_{y_\gamma(j)}$, and
\item $i_j\neq i_j'$.
\end{itemize}
\end{claim}
\begin{proof} We shall define by recursion a sequence of pairs $\langle (A_j,B_j)\mid j\le \sigma\rangle$
such that, for all $j<\sigma$, $A_{j+1}\in [A_j\cap\Gamma]^{\lambda}$ and $B_{j+1}\in [B_j\cap\Gamma]^{\lambda}$.

We commence by letting both $A_0$ and $B_0$ be $\Gamma$.
Now, for every $j<\sigma$ for which the pair $(A_j,B_j)$ has already been defined, we do the following.
Set $A^j:=\{ y_\gamma(j)\mid \gamma\in A_j\}$ and $B^j:=\{ y_\gamma(j)\mid \gamma\in B_j\}$.
By Clause~(III), $A^j$ and $B^j$ have size $\lambda$.
So, by Lemma~\ref{separation lemma2}(2),
we may find $s_j \in T$ and $i_j\neq i_j'$ such that $\{ \alpha\in A^j\mid {s_j{}^\smallfrown\langle i_j\rangle}\sq b_\alpha\}$ and 
$\{ \beta\in B^j\mid {s_j{}^\smallfrown\langle i_j'\rangle }\sq b_\beta\}$ are both of size $\lambda$.
Then, let $A_{j+1}:=\{ \gamma\in A_j\mid s_j{}^\smallfrown\langle i_j\rangle\s b_{y_{\gamma}(j)}\}$
and $B_{j+1}:=\{ \gamma\in B_j\mid s_j{}^\smallfrown\langle 1-i_j\rangle\s b_{y_{\gamma}(j)}\}$.

Clearly, $\Gamma_0:=A_\sigma$ and $\Gamma_1:=B_\sigma$ are as sought.
\end{proof}
Let $\Gamma_0,\Gamma_1$ and $\langle (s_j,i_j,i_j')\mid j<\sigma\rangle$ be given by the preceding claim.
Note that $\Gamma_0$ is disjoint from $\Gamma_1$.
Set $\delta^*:=\max\{ \dom(s_j)\mid j<\sigma\}$ and $j^*:=\min\{ j<\sigma\mid \dom(s_j)=\delta^*\}$.

\begin{claim}\label{claigm441}
Let $ (\gamma , \gamma')\in(\Gamma_0\circledast\Gamma_1)\cup (\Gamma_1\circledast\Gamma_0)$.
Let  $z\in[x_{\gamma} \cup  x_{\gamma'}]^{<\omega}$ be such that $\{y_\gamma(j^*),y_{\gamma'}(j^*)\}\s z$.
Then $e(z)=(y_\gamma(j^*),y_{\gamma'}(j^*))$.
\end{claim}
\begin{proof} It is clear that $2\le\otp(z)<\omega$, so that $M_z$ is nonempty.
Let $( \alpha , \beta )\in [z]^2 $. By the choice of $z$,  we must analyze the following cases:
\begin{enumerate}
\item Suppose that $\alpha \in r$.

As $\beta \in r \cup y_{\gamma} \cup y_{\gamma'} $, it follows from Clause~(I) that $\Delta (\alpha , \beta ) \leq \delta $.
\item Suppose that $\alpha \in y_{\gamma} $.
\begin{enumerate}
\item If $\beta \in r \cup y_{\gamma} $, then it follows from Clause~(I) that  $\Delta (\alpha , \beta ) \leq \delta $;
\item If $\beta \in y_{\gamma'} $, then let $j_\alpha , j_\beta < \sigma $ be such that, $\alpha = y_{\gamma } (j_\alpha )$ and $ \beta = y_{\gamma'} (j_\beta ) $. There are two possible options:
\begin{enumerate}
\item If $j_\alpha =  j_\beta = j $, then by Clause~(II), $f_{\alpha } \restriction (\delta +1 ) = t_j = f_{\beta } \restriction (\delta +1 )  $. So $\Delta (\alpha , \beta ) > \delta $.
\item If  $j_\alpha \neq  j_\beta  $, then by Clauses (I) and (II),
$$\Delta (y_{\gamma'} (j_\alpha ) , \beta )\le\delta<\Delta (\alpha , y_{\gamma'} (j_\alpha ) ),$$
and hence $\Delta (\alpha , \beta )=\Delta (y_{\gamma'} (j_\alpha ) , \beta )\le\delta $.
\end{enumerate}
\end{enumerate}
\item If $\alpha \in y_{\gamma'} $, then the analysis is analogous to that of (2).
\end{enumerate}
Altogether, so far we have shown that
$$\emptyset \subsetneq M_z \subseteq \{(y_{\gamma} (j) , y_{\gamma'} (j) )\mid j < \sigma \}.$$
Recalling that $(\gamma,\gamma')\in(\Gamma_0\circledast\Gamma_1)\cup (\Gamma_1\circledast\Gamma_0)$, we infer from the choice of $\delta^*$ that
$$\emptyset \subsetneq M_z \subseteq \{(y_{\gamma}(j),y_{\gamma'}(j))\mid j < \sigma, \dom(s_j)=\delta^* \}.$$
So, since $\{y_\gamma(j^*),y_{\gamma'}(j^*)\}\s z$,
it is the case that 
$M_z^*=\{(y_{\gamma} (j^*) , y_{\gamma'} (j^*) )\}$.
In particular, $e(z)=(y_{\gamma} (j^*) , y_{\gamma'} (j^*) )$, as sought.
\end{proof}
This completes the proof.
\end{proof}

\begin{cor}\label{cor410} Suppose that $\lambda$ is an infinite regular cardinal, and $\nu<\lambda$.

If there exists a cardinal $\theta<\lambda$ such that $\nu^\theta\ge\lambda$,
then $\ext_2(\nu^\theta,\lambda,\cf(\theta),\omega)$ holds for the least such $\theta$.
\end{cor}
\begin{proof} Let $\theta$ denote the least cardinal such that $\nu^\theta\ge\lambda$.
Then $T:={}^{<\theta}\nu$ belongs to $\mathcal T(\theta,\lambda)$,
so that $\nu^\theta=|\mathcal B(T)|$ is in $T(\theta,\lambda)$.
Now, appeal to Lemma~\ref{extract2} with $(\kappa,\mu):=(\nu^\theta,\lambda)$.
\end{proof}
\begin{cor} For every regular uncountable cardinal $\kappa$ that is not a strong limit, $\ext_2(\kappa,\kappa,\omega,\omega)$ holds.\qed
\end{cor}

\begin{prop}\label{prop412} Suppose that $\lambda$ is a regular uncountable cardinal. 

For every cardinal $\kappa>2^{<\lambda}$, $\ext_2(\kappa,\lambda,2,2)$ fails.
\end{prop}
\begin{proof} Set $\nu:=2^{<\lambda}$,
and note that $\nu^\theta=\nu$ for every $\theta<\lambda$.
Towards a contradiction, suppose that $e:[\kappa]^{<\omega}\rightarrow[\kappa]^2$ is a map witnessing $\ext_2(\kappa,\lambda,2,2)$,
and yet $\kappa>\nu$.
Without loss of generality, we may assume that $e(z)\in[z]^2$ for every $z\in[\kappa]^3$.

\begin{claim} For every $\delta<\kappa$, there exist no subset $A\s\delta$ of order-type $\lambda$ such that $\delta\in e(\{\alpha,\beta,\delta\})$
for every $(\alpha,\beta)\in[A]^2$.
\end{claim}
\begin{proof} Otherwise, fix a counterexample $\delta$ and a witnessing $A\s\delta$.
Let $\langle \alpha_\gamma\mid \gamma<\lambda\rangle$ be the increasing enumeration of $A$.
Now let $r:=\{\delta\}$ and, for every $\gamma<\lambda$, put $x_\gamma:=y_\gamma\uplus r$ where $y_\gamma:=\{\alpha_\gamma\}$.
Then, for every $(\gamma,\gamma')\in[\lambda]^2$, setting $z:=x_\gamma\cup x_{\gamma'}$,
we get that $r\cap e(z)\neq\emptyset$. This is a contradiction.
\end{proof}

Denote $\varkappa:=\nu^+$. 
It follows from the claim that for every $\delta\in E^\varkappa_\lambda$, we may fix some $A_\delta\in[\delta]^{<\lambda}$ with the property that for every ordinal $\beta$
such that $\sup(A_\delta)<\beta<\delta$, there exists $\alpha\in A_\delta$ such that $e(\{\alpha,\beta,\delta\})=(\alpha,\beta)$.
Now, using Fodor's lemma, we may find $\varepsilon<\varkappa$ and $\theta<\lambda$ such that
$\{\delta\in E^\varkappa_\lambda\mid \ssup(A_\delta)=\varepsilon\ \&\ |A_\delta|=\theta\}$ is stationary.
Recalling that $\nu^\theta=\nu<\varkappa$, we may then find some $A\in[\varepsilon]^\theta$ for which $S:=\{\delta\in E^\varkappa_\lambda\mid A_\delta=A\}$ is stationary.
Define a coloring $c:[S]^2\rightarrow A$ by letting for every $(\beta,\delta)\in[S]^2$:
$$c(\beta,\delta):=\min\{\alpha\in A\mid e(\{\alpha,\beta,\delta\})=(\alpha,\beta)\}.$$

By the Erd\H{o}s-Rado theorem, $\varkappa\rightarrow(\lambda)^2_\theta$ holds,
so we may pick $B\s S$ of order-type $\lambda$ that is $c$-homogeneous, with value, say, $\alpha$.
Let $\langle \beta_\gamma\mid \gamma<\lambda\rangle$ be the increasing enumeration of $B$.
Finally, let $r:=\{\alpha\}$ and, for every $\gamma<\lambda$, put $x_\gamma:=y_\gamma\uplus r$ where $y_\gamma:=\{\beta_\gamma\}$.
Then, for every $(\gamma,\gamma')\in[\lambda]^2$, setting $z:=x_\gamma\cup x_{\gamma'}$,
we get that $r\cap e(z)\neq\emptyset$. This is a contradiction.
\end{proof}

\begin{cor} If $\kappa$ is a strong limit cardinal, then $\ext_2(\kappa,\lambda,2,2)$ fails for every infinite cardinal $\lambda<\kappa$.\qed
\end{cor}
\begin{cor} $\ext_2(\aleph_2,\aleph_1,\aleph_0,\aleph_0)$ holds iff $\ch$ fails.
\end{cor}
\begin{proof} By Corollary~\ref{cor410} and Proposition~\ref{prop412}.
\end{proof}
Moving on from the case $n=2$ to the general case, we consider the following two definitions.

\begin{defn}
$S_n (\kappa,\lambda,\theta )$ asserts the existence of a coloring $d:[\kappa ] ^{<\omega}  \rightarrow \theta $ such that,  for every $\mathcal X \s [ \kappa ] ^{ < \omega } $ of size $\lambda $ and every prescribed color $\tau<\theta$, 
there exist $\{a_j\mid j<n\} \in[\mathcal X]^n $ such that $d(z)=\tau$ for every $z$ satisfying 
$$a_0\symdiff(\bigcup\nolimits_{0<j<n} a_j)\s z\s\bigcup\nolimits_{j<n} a_j.$$
\end{defn}  	
\begin{prop}[monotonicity] Suppose that:
\begin{enumerate}
\item $2\le n\le n'<\omega$;
\item $\omega\le\lambda\le\lambda'$;
\item $\theta\le\theta'$.
\end{enumerate}

Then $S_n(\kappa,\lambda,\theta')$ implies $S_{n'}(\kappa,\lambda',\theta)$.
\end{prop}
\begin{proof} This is mostly trivial, so we settle for proving that if $d:[\kappa]^{<\omega}\rightarrow\theta$ witnesses $S_n(\kappa,\lambda,\theta)$ for some integer $n\ge 2$,
then it also witnesses $S_{n+1}(\kappa,\lambda,\theta)$.

To this end, let $\mathcal X \s [ \kappa ] ^{ < \omega } $ be a given family of size $\lambda $.
Pick $x\in\mathcal X$, and note that $\mathcal X':=\{ a\cup x\mid a\in\mathcal X\setminus\{x\}\}$ is a $\lambda$-sized subset of $[\kappa]^{<\omega}$.
Now, given a prescribed color $\tau<\theta$, pick $\{a_j'\mid j<n\} \in[\mathcal X']^n $ such that $d(z)=\tau$ for every $z$ satisfying 
$$a_0'\symdiff(\bigcup\nolimits_{0<j<n} a_j')\s z\s\bigcup\nolimits_{j<n} a_j'.$$
For each $j<n$, pick $a_j\in\mathcal X\setminus\{x\}$ such that $a_j'=a_j\cup x$. As $\langle a_j'\mid j<n\rangle$ is an injective sequence, so is $\langle a_j\mid j<n\rangle$.
Set $a_n:=x$. Altogether, $\{a_j\mid j<n+1\} \in[\mathcal X]^{n+1}$.
It is clear that $\bigcup\nolimits_{j<n} a_j'=\bigcup\nolimits_{j<n+1} a_j$. In addition,
$$\begin{array}{lllll}
a_0'\symdiff(\bigcup\nolimits_{0<j<n} a_j')&=&(a_0'\setminus(\bigcup\nolimits_{0<j<n} a_j'))&\cup&((\bigcup\nolimits_{0<j<n} a_j')\setminus a_0')\\
&\s&(a_0\setminus(\bigcup\nolimits_{0<j<n} a_j))&\cup&((\bigcup\nolimits_{0<j<n} a_j)\setminus a_0)\\
&\s&(a_0\setminus(\bigcup\nolimits_{0<j<n+1} a_j))&\cup&((\bigcup\nolimits_{0<j<n+1} a_j)\setminus a_0)\\
&=&a_0\symdiff(\bigcup\nolimits_{0<j<n+1} a_j).
\end{array}$$
Therefore, $d(z)=\tau$ for every set $z$ with $a_0\symdiff(\bigcup\nolimits_{0<j<n+1} a_j)\s z\s\bigcup\nolimits_{j<n+1} a_j$.
\end{proof}
\begin{defn}\label{extn} $\ext_n(\kappa,\lambda,\theta,\chi)$ asserts the existence of a map $e:[\kappa]^{<\omega}\rightarrow{}^n\kappa$
such that for every sequence $\langle x_\gamma\mid\gamma<\lambda\rangle$ of subsets of $\kappa$,
every $r\in[\kappa]^{<\theta}$, and every nonzero $\sigma<\chi$ such that:
\begin{enumerate}
\item for every $(\gamma,\gamma')\in[\lambda]^2$, $x_\gamma\cap x_{\gamma'}\s r$;
\item for every $\gamma<\lambda$, $y_\gamma:=x_\gamma\setminus r$ has order-type $\sigma$,
\end{enumerate}
there exist $j<\sigma$ and disjoint cofinal subsets $\Gamma_0,\Gamma_1$ of $\lambda$ satisfying the following:
\begin{enumerate}
\item[(a)] For every $(\gamma,\gamma')\in[\Gamma_0\cup\Gamma_1]^2$, $y_\gamma(j)<y_{\gamma'}(j)$;
\item[(b)] For every strictly increasing sequence $\langle \gamma_i\mid i<n\rangle$ of ordinals from $\Gamma_0\cup\Gamma_1$ 
such that $\{\gamma_i\mid i<n\}\notin([\Gamma_0]^n\cup [\Gamma_1]^n)$,
for every $z\in[\bigcup_{i<n}x_{\gamma_i}]^{<\omega}$ that covers $\{y_{\gamma_i}(j)\mid i<n\}$, we have
$$e(z)=\langle y_{\gamma_i}(j)\mid i<n\rangle.$$
\end{enumerate}

\end{defn}
\begin{remark} Without loss of generality, we may assume that for every $z\in[\kappa]^{<\omega}$ of size $\ge n$, $e(z)$ consists of ordinals from $z$.
\end{remark}
The proof of \cite[Theorem~4.7]{paper27} makes it clear that the following holds.
\begin{prop}\label{prop316}
Suppose that $S_n(\kappa,\lambda,\theta )$ holds
for given cardinals $\theta\le\lambda\le\kappa$ with $\lambda$ regular and uncountable.
For every map $\varphi:G\rightarrow[G]^{<\omega}$ that is ${<}\lambda$-to-one,
there exists a corresponding coloring $c:G\rightarrow\theta$ satisfying the following.

For every binary operation $*$ on $G$ such that, for all $x\neq y$ in $G$, $$\varphi(x)\symdiff\varphi(y)\s \varphi(x*y)\s\varphi(x)\cup\varphi(y),$$
for every $X\in[G]^\lambda$ and every $\tau<\theta$,
there is an injective sequence $\langle x_j\mid 1\le j\le n\rangle$ of elements of $X$ such that $c(x_1*\cdots *x_n)=\tau$.\footnote{As always, we mean that this holds true for all implementations of $x_1*\cdots *x_n$.}\qed
\end{prop}

In order to generalize Lemma~\ref{lemma49}, we now introduce the relation $\kappa\nsrightarrow[\lambda,\lambda]^n_\theta$.
It is a strengthening of $\kappa\nrightarrow[\lambda]^n_\theta$,
and a weakening of $\kappa\nrightarrow[\lambda,\ldots,\lambda]^n_\theta$.

\begin{defn}\label{narrowsup} $\kappa\nsrightarrow[\lambda,\lambda]^n_\theta$ asserts the existence of a coloring $c:[\kappa]^n\rightarrow\theta$ such that for all $\tau<\theta$ and disjoint $A,B\in\mathcal P(\kappa)$ satisfying the two:
\begin{enumerate}
\item[(i)] $\otp(A)=\otp(B)=\lambda$, 
\item[(ii)] $\sup(A)=\sup(B)$,
\end{enumerate}
there is $\vec x\in[A\cup B]^n\setminus([A]^n \cup[B]^n)$ with $c(\vec x)=\tau$.
\end{defn}
\begin{remark} In the special case of $\lambda=\kappa$,
$\kappa\nsrightarrow[\lambda,\lambda]^2_\theta$ coincides with the classical relation $\kappa\nrightarrow[\lambda,\lambda]^2_\theta$.
\end{remark}

\begin{lemma}\label{pilemma}  Suppose that:
\begin{itemize}
\item $2\le n<\omega$;
\item $\theta\le\lambda\le\kappa$ are cardinals with $\lambda$ regular and uncountable;
\item $\kappa\nsrightarrow[\lambda,\lambda]^n_\theta$ holds;
\item $\ext_n(\kappa,\lambda,\omega,\omega)$ holds.
\end{itemize}

Then $S_n(\kappa,\lambda,\theta )$ holds.
\end{lemma}
\begin{proof} The proof is similar to that of Lemma~\ref{lemma49}, so we settle for a sketch.
Fix a map $c:[\kappa]^n\rightarrow\theta$ witnessing $\kappa\nsrightarrow[\lambda,\lambda]^n_\theta$ and a map $e:[\kappa]^{<\omega}\rightarrow{}^n\kappa$ witnessing $\ext_n(\kappa,\lambda,\omega,\omega)$.
We claim that $d:=c\circ e$ is a witness for $S_n(\kappa,\lambda,\theta)$.
To this end, suppose that we are given a subfamily $\mathcal X \s [ \kappa ] ^{ < \omega } $ of size $\lambda$,
and a prescribed color $\tau<\theta$.
By the $\Delta$-system lemma,
find a sequence $\langle x_\gamma\mid\gamma<\lambda\rangle$ consisting of elements of $\mathcal X$,
some $r\in[\kappa]^{<\omega}$, and a nonzero $\sigma<\chi$ such that:
\begin{enumerate}
\item for every $(\gamma,\gamma')\in[\lambda]^2$, $x_\gamma\cap x_{\gamma'}=r$;
\item for every $\gamma<\lambda$, $y_\gamma:=x_\gamma\setminus r$ has order-type $\sigma$.
\end{enumerate}

Now, let $j<\sigma$ and disjoint cofinal subsets $\Gamma_0,\Gamma_1$ of $\lambda$ be given, as in Definition~\ref{extn}.
Put $A:=\{ y_\gamma(j)\mid \gamma\in\Gamma_0\}$ and $B:=\{ y_\gamma(j)\mid \gamma\in\Gamma_1\}$.
By the choice of $c$, pick $\vec x\in[A\cup B]^n\setminus([A]^n \cup[B]^n)$ such that $c(\vec x)=\tau$.
Find a sequence $\vec\gamma=\langle \gamma_i\mid i<n\rangle$ of ordinals from $\Gamma_0\cup\Gamma_1$ 
such that $\vec x=\langle y_{\gamma_i}(j)\mid i<n\rangle$.
Clearly, $\vec\gamma$ is strictly increasing,
$\{\gamma_i\mid i<n\}\nsubseteq\Gamma_0$ and $\{\gamma_i\mid i<n\}\nsubseteq\Gamma_1$.
So, for every set $z$ of interest, $$d(z)=c(e(z))=c(\vec x)=\tau,$$ as sought.
\end{proof}

\begin{lemma}\label{ext3}  Suppose that $\kappa\in T(\mu,\theta)$.
Then there exists a map $e:[\kappa]^{<\omega}\rightarrow{}^3\kappa$
witnessing $\ext_3(\kappa,\lambda,\cf(\theta),\omega)$ for every regular cardinal $\lambda$ with $\max\{\mu,\theta^+\}\le\lambda\le\kappa$.
\end{lemma}
\begin{proof} As $\kappa\in T(\mu,\theta)$, let us fix $T\in\mathcal T(\mu,\theta)$ 
admitting an injective sequence $\langle b_\xi\mid \xi<\kappa\rangle$ consisting of elements of $\mathcal B(T)$.
For notational simplicity, we shall write $\Delta(\alpha,\beta)$ for $\Delta(b_\alpha,b_\beta)$.
For any triplet $w\in[\kappa]^3$, let $\Delta_2(w):=\min\{\Delta(\alpha,\beta)\mid (\alpha,\beta)\in[w]^2\}$.
First, given  $z\in[\kappa]^{<\omega}$, let:
\begin{itemize}
\item $M_z := \{ ( \alpha  ,  \beta, \gamma) \in [z]^3 \mid \Delta_2(\{\alpha,\beta,\gamma\}) = \max\{\Delta_2(w)\mid w \in [z]^3\} \} $, and
\item $M_z^* := \{ ( \alpha  ,  \beta, \gamma)\in M_z\mid  \alpha=\min\{  \alpha'\mid ( \alpha ', \beta', \gamma')\in M_z\}\}$.
\end{itemize}
Then, pick any function $e:[\kappa]^{<\omega}\rightarrow{}^3\kappa$ satisfying that for every $z\in[\kappa]^{<\omega}$:
\begin{itemize}
\item for every $z\in[\kappa]^{<\omega}$, $e(z)$ is a strictly increasing sequence of ordinals in $\kappa$. If $|z|\ge 3$, then $e(z)$ consists of ordinals from $z$;
\item if $M_{z^*}$ is a singleton, then $e(z)$ is its unique element.
\end{itemize} 

To see that $e$ is a sought, 
suppose that $\lambda$ is a regular cardinal satisfying $\max\{\mu,\theta^+\}\le\lambda\le\kappa$,
and that we are given $\langle x_\gamma\mid\gamma<\lambda\rangle$, $r\in[\kappa]^{<\cf(\theta)}$ and $\sigma<\omega$ as in Definition~\ref{extn}.
As in the proof of Lemma~\ref{extract2},
we may find a cofinal subset $\Gamma\s\lambda$,
an ordinal $\delta<\theta$, and a sequence $\langle t_j\mid j<\sigma\rangle$ of nodes in $T_{\delta+1}$ 
such that for every $(\gamma,\gamma')\in[\Gamma]^2$:
\begin{enumerate}
\item[(I)] $\sup( \Delta``[r \cup y_{\gamma} ]^2)=\delta$;
\item[(II)] for every $j<\sigma$, $b_{y_{\gamma} (j)} \restriction (\delta +1 )=t_j$;
\item[(III)] for every $j<\sigma$, $y_{\gamma}(j) <  y_{\gamma'}(j)$.
\end{enumerate}
In addition, we may fix $\Gamma_0,\Gamma_1\in[\Gamma]^{\lambda^+}$ and a sequence $\langle (s_j,i_j,i_j')\mid j<m\rangle$ of triples in $T\times\mu\times\mu$ such that,
for every $j<\sigma$:
\begin{itemize}
\item for every $\gamma\in\Gamma_0$, $s_j{}^\smallfrown\langle i_j\rangle\sq b_{y_\gamma(j)}$, 
\item for every $\gamma\in\Gamma_1$, $s_j{}^\smallfrown\langle i_j'\rangle\sq b_{y_\gamma(j)}$, and
\item $i_j\neq i_j'$.
\end{itemize}
By possibly passing to cofinal subsets, we may assume that $\Gamma_0\cap\Gamma_1=\emptyset$.
Let $\delta^*:=\max\{ \dom(s_j)\mid j<\sigma\}$ and $j^*:=\min\{ j<\sigma\mid \dom(s_j)=\delta^*\}$.

\begin{claim}
Let  $(\gamma,\gamma',\gamma'')\in[\Gamma_0\cup\Gamma_1]^3\setminus([\Gamma_0]^3\cup[\Gamma_1]^3)$.
Let  $z\in[x_{\gamma} \cup  x_{\gamma'}\cup x_{\gamma''}]^{<\omega}$ be such that $\{y_\gamma(j^*),y_{\gamma'}(j^*),y_{\gamma''}(j^*)\}\s z$.
Then $e(z)=(y_\gamma(j^*),y_{\gamma'}(j^*),y_{\gamma''}(j^*))$.
\end{claim}
\begin{proof} As made clear by the proof of Claim~\ref{claigm441}, for every $(\xi,\zeta)\in[\{\gamma,\gamma',\gamma''\}]^2$,
for every $\alpha\in r\cup a_\xi$ and $\beta\in r\cup y_\zeta$,
$\Delta(\alpha,\beta)>\delta$ iff there is a $j<\sigma$ such that $\alpha=y_\xi(j)$ and $\beta=y_\zeta(j)$.
In addition,  if $\{\xi,\zeta\}\nsubseteq \Gamma_0$ and $\{\xi,\zeta\}\nsubseteq \Gamma_1$, 
then for every $j<m$, $\Delta(y_\xi(j),y_\zeta(j))=\dom(s_j)$.
So, in this case,
$$\begin{aligned}
&\ \{ (\alpha,\beta)\in [r\cup y_\xi\cup y_{\zeta}]^2\mid \Delta(\alpha,\beta)=\max(\Delta``[r\cup y_\xi\cup y_\zeta]^2)\}\\
=&\ \{ (y_\xi(j),y_\zeta(j))\mid \dom(s_j)=\delta^*\}.
\end{aligned}$$

Now, since $\{\gamma,\gamma',\gamma''\}\nsubseteq\Gamma_0$ and $\{\gamma,\gamma',\gamma''\}\nsubseteq\Gamma_1$,
it follows that $$\Delta_2(y_\gamma(j^*),y_{\gamma'}(j^*),y_{\gamma''}(j^*))=\delta^*.$$ 
Consequently,
$$\emptyset \subsetneq M_z = \{ (y_\gamma(j),y_{\gamma'}(j),y_{\gamma''}(j))\mid \dom(s_j)=\delta^*\}.$$
So, by Clause~(III),
$$e(z) = (y_\gamma(j^*),y_{\gamma'}(j^*),y_{\gamma''}(j^*)),$$
as sought.
\end{proof}
This completes the proof.
\end{proof}

\begin{prop} Suppose that $\lambda\le \kappa$ is a pair of infinite cardinals.

If $\ext_2(\kappa,\lambda,3,3)$ holds, then so does $\kappa\nsrightarrow[\lambda,\lambda]^4_2$.
\end{prop}
\begin{proof}  Suppose that $\kappa\nsrightarrow[\lambda,\lambda]^4_2$ fails,
and we shall prove that $\ext_2(\kappa,\lambda,3,3)$ fails, as well.
To this end, let $e:[\kappa]^{<\omega}\rightarrow[\kappa]^2$ be given.
Define a coloring $c:[\kappa]^4\rightarrow2$ by letting for all $\alpha_0<\alpha_1<\alpha_2<\alpha_3<\kappa$:
$$c(\alpha_0,\alpha_1,\alpha_2,\alpha_3):=1\text{ iff }e(\alpha_0,\alpha_1,\alpha_2,\alpha_3)=(\alpha_2,\alpha_3).$$

Now, since $\kappa\srightarrow[\lambda,\lambda]^4_2$ holds, we may find $\tau<2$ and disjoint $A,B\in\mathcal P(\kappa)$ satisfying all of the following:
\begin{enumerate}
\item[(i)] $\otp(A)=\otp(B)=\lambda$, 
\item[(ii)] $\sup(A)=\sup(B)$,
\item[(iii)] for every $(\alpha_0,\alpha_1,\alpha_2,\alpha_3)\in[A\cup B]^4\setminus([A]^4 \cup[B]^4)$, $$c(\alpha_0,\alpha_1,\alpha_2,\alpha_3)\neq\tau.$$
\end{enumerate}

Using Clauses (i) and (ii), fix a sequence $\langle (\alpha_i,\beta_i)\mid i<\lambda\rangle$ of pairs in $A\times B$ such that, for all $i<j<\lambda$,
$\alpha_i<\beta_i<\alpha_j$.

$\br$ If $\tau=1$, 
then let $r:=\{\alpha_0,\alpha_1\}$, and for every $\gamma<\lambda$, let $x_\gamma:=r\uplus y_\gamma$, where $y_\gamma:=\{\beta_{\gamma+1}\}$.
Now, for every $(\gamma,\gamma')\in[\lambda]^2$, as $z:=x_\gamma\cup x_{\gamma'}$ is in $[A\cup B]^4\setminus([A]^4 \cup[B]^4)$, 
$c(z)=0$, and then $e(z)$ is not disjoint from $r$.

$\br$ If $\tau=0$, 
then let $r:=\emptyset$, and for every $\gamma<\lambda$, let $x_\gamma:=r\uplus y_\gamma$, where $y_\gamma:=\{\alpha_\gamma,\beta_\gamma\}$.
Now, for every $(\gamma,\gamma')\in[\lambda]^2$, as $z:=x_\gamma\cup x_{\gamma'}$ is in $[A\cup B]^4\setminus([A]^4 \cup[B]^4)$, 
$c(z)=1$, and then  $e(z)=y_{\gamma'}$ which is disjoint from $y_\gamma$.
\end{proof}

\begin{cor}
If $\lambda=\aleph_0$ or if $\lambda$ is weakly compact, then $\ext_2(\kappa,\lambda,3,3)$ fails for every cardinal $\kappa\ge\lambda$.\qed
\end{cor}

\section{Maximal number of colors}\label{changsection}

This section is dedicated to the proof of Theorem~C.
The main corollary of this section reads as follows:

\begin{cor}\label{cor51} Suppose that $\lambda=\mu^+$ for an infinite cardinal $\mu=\mu^{<\mu}$.

Then the following are equivalent:
\begin{enumerate}
\item $(\lambda^+,\lambda)\twoheadrightarrow(\mu^+,\mu)$ fails;
\item$\lambda^+\nsrightarrow[\lambda,\lambda]^3_\lambda$ holds. 
\end{enumerate}
\end{cor}
\begin{proof} We focus on the nontrivial (that is, forward) implication. As $\lambda=\mu^+$, 
by \cite[Lemma~9.2.3]{MR2355670}, the failure of $(\lambda^+,\lambda)\twoheadrightarrow(\mu^+,\mu)$ 
is equivalent to the existence of a subadditive coloring $\varrho:[\lambda^+]^2\rightarrow\lambda$
witnessing $\U(\lambda^+,\lambda,\lambda,\lambda,\omega)$. In particular, 
$\varrho\restriction [X]^2$ witnesses $\U(\lambda,\lambda,\lambda,3)$
for every $X\s\lambda^+$ of order-type $\lambda$. 
Now, there are two cases to consider:

$\br$ If $2^\mu>\mu^+$, then since $\mu^{<\mu}=\mu$, $T:={}^{<\mu}\mu$ is a weak $\mu$-Kurepa tree with $\lambda^+$-many branches.
In addition, $\mu^{<\mu}=\mu$ implies that $\mu$ is regular, so that $E^\lambda_\mu$ is a nonreflecting stationary set.
Now the result follows from the upcoming Theorem~\ref{use of stability}, using $\kappa:=\lambda^+$.

$\br$ If $2^\mu=\mu^+$, then $\lambda\nrightarrow[\mu;\lambda]^2_\lambda$ holds by a theorem of Sierpi\'nski (see  \cite[Lemma~8.3]{paper47}).
Now the result follows from Theorem~\ref{thm53} below.
\end{proof}

When reading the hypotheses of the upcoming theorem, it may worth keeping in mind that 
if $\lambda=\mu^+$ for an infinite regular cardinal $\mu$, then $E^\lambda_\mu$ is a nonreflecting stationary set, and 
if $\kappa=\lambda^+$, then Fact~\ref{all the rho} provides a subadditive map $\rho:[\kappa]^2\rightarrow\lambda$.
The conclusion of the theorem is a conditional form of $\kappa\nsrightarrow[\lambda,\lambda]^3_\lambda$ in which a third clause is added to Definition~\ref{narrowsup}.

\begin{thm}\label{use of stability} Suppose that:
\begin{itemize}
\item $\mu<\lambda<\kappa$ are infinite regular cardinals;
\item $E^\lambda_\mu$ admits a nonreflecting stationary set;
\item $\varrho:[\kappa]^2\rightarrow\lambda$ is a subadditive coloring of pairs;
\item there exists a weak $\mu$-Kurepa tree with at least $\kappa$-many branches.
\end{itemize}

Then	there exists a corresponding coloring of triples $c:[\kappa]^3\rightarrow\lambda$ such that, 
for all $\tau<\lambda$ and disjoint $A,B\in\mathcal P(\kappa)$ satisfying the three:
\begin{enumerate}
\item[(i)] $\otp(A)=\otp(B)=\lambda$,
\item[(ii)] $\sup(A)=\sup(B)$,
\item[(iii)] $\varrho\restriction [A\cup B]^2$ witnesses $\U(\lambda,\lambda,\lambda,3)$,
\end{enumerate}
there exists $(\alpha,\beta,\gamma)\in[A\cup B]^3\setminus([A]^3\cup[B]^3)$ such that $c(\alpha,\beta,\gamma)=\tau$.
\end{thm}
\begin{proof} Let $T\s{}^{<\mu}2$ be a weak $\mu$-Kurepa tree with at least $\kappa$-many branches,
and let	$\vec b=\langle b_\xi \mid \xi< \kappa\rangle$ be an injective sequence consisting of elements of $\mathcal B(T)$.
For all $\alpha\neq\beta$ from $\kappa$, we write $\Delta(\alpha,\beta)$ for $\Delta(b_\alpha,b_\beta)$.
For all $B\s\kappa$ and $t\in T$, denote $B_t:=\{\beta\in B\mid t\sq b_\beta\}$.
As $E^\lambda_\mu$ admits a nonreflecting stationary set, by \cite[Lemma~3.31]{paper36}, we may fix a coloring $e:[\lambda]^2\rightarrow\mu$ for which the following set is stationary:
$$\partial(e):=\{ \sigma\in E^\lambda_\mu\mid \forall \epsilon\in\lambda\setminus\sigma\,\forall \delta<\mu\,[\sup\{ \zeta<\sigma\mid e(\zeta,\epsilon)\le\delta\}<\sigma]\}.$$
Let $h:\lambda\rightarrow\lambda$ be a surjection such that
$S_\tau:=\{ \sigma\in \partial(e)\mid h(\sigma)=\tau\}$ is stationary
for every $\tau<\lambda$.

For every $(\alpha,\beta,\gamma)\in \kappa\times\kappa\times\kappa$,
let $$Z_{(\alpha,\beta,\gamma)}:= \{\zeta\in\lambda \setminus \varrho(\{\alpha,\beta\}) \mid \max(\Delta``  \{\alpha,\beta,\gamma\}^2) \geq e(\zeta,\varrho(\{\beta , \gamma\})) \}.$$

Now, derive a coloring $c:[\kappa]^3\rightarrow\lambda$ by letting:
$$c( \{\alpha,\beta,\gamma\}) := \begin{cases}
h(\min(Z_{(\beta,\alpha,\gamma)})),&\text{if }\gamma=\max\{\alpha,\beta,\gamma\}\ \&\ b_\alpha<_{\text{lex}}b_\beta<_{\text{lex}}b_\gamma;\\
h(\min(Z_{(\alpha,\beta,\gamma)})),&\text{if }\gamma=\max\{\alpha,\beta,\gamma\}\ \&\ b_\gamma<_{\text{lex}}b_\alpha<_{\text{lex}}b_\beta;\\
0,&\text{otherwise}.
\end{cases}$$

Suppose that $A,B\in\mathcal P(\kappa)$ are disjoint sets satisfying conditions (i)--(iii) above. 
Recalling Lemma~\ref{separation lemma2}(2), by possibly shrinking $A$ and $B$, we may assume the existence of some $\chi<\mu$ such that $\Delta[A\times B]=\{\chi\}$.  
By possibly switching the roles of $A$ and $B$, we may also assume the following:
\begin{itemize}
\item[(iv)] for every $(\alpha,\beta)\in A\times B$, $b_\alpha<_{\text{lex}}b_\beta$.
\end{itemize}

Note that $|T|=\mu<\lambda<\kappa\le 2^\mu$.
Now, given a prescribed color $\tau<\lambda$,
let ${M}$ be an elementary submodel of $\mathcal{H}_{(2^\mu)^+}$ containing $\{A,B,\varrho,\vec b,T\}$ such that $\sigma:= {M}\cap\lambda$ is in $S_\tau$.
In particular, $|{M}|\ge\cf(\sigma)=\mu$.
Denote $\upsilon:=\sup(A)$ and $\upsilon_{M}:=\sup({M}\cap \upsilon)$.
The proof is now divided into two cases: 

\medskip

\underline{\textbf{Case 1: For every $\beta\in B$, $\sup(\varrho_\beta``A)<\lambda$.}} 
By Condition~(iii), find $A'\in[A]^{\lambda}$ and $B'\in[B]^{\lambda}$ such that $\min(\varrho[A'\circledast B'])>\sigma$.
Fix $\alpha\in A'\setminus\upsilon_{{M}}$ arbitrarily. 
Appeal to Corollary~\ref{usefullemma} with $X:=B'$ and $i:=1$ to pick
$\gamma\in B'\setminus(\alpha+1)$ such that for cofinally many $\delta<\mu$, the two hold:
\begin{enumerate}
\item $b_{\gamma}(\delta)=1$, and
\item $\{ \beta\in B'\mid \Delta(\beta,\gamma)=\delta\}$ has size $\lambda$.
\end{enumerate}
Denote $D:=\{\delta<\mu\mid \text{Clauses (1) and (2) both hold}\}$.

As $(\alpha,\gamma)\in A'\circledast B'$, the ordinal $\epsilon:=\varrho(\alpha,\gamma)$ is bigger than $\sigma$. 
Pick $\delta\in D$ above $\max\{\chi,e(\sigma,\epsilon)\}$. 
Since $\sigma\in\partial(e)$, the following set is bounded below $\sigma$:
$$Z:=\{\zeta<\sigma\mid e(\zeta,\epsilon)\le\delta\}.$$

Set $t:=(b_\gamma\restriction\delta){}^\smallfrown\langle 0\rangle$. 
From $\delta\in D$ we infer that $(B')_t$ has size $\lambda$.
As $t\in T\s{M}$, in particular, $B_t$ is a set of size $\lambda$ lying in ${M}$.
By Condition~(iii) and elementarity of ${M}$ one can find $\beta_0\neq\beta_1$ in $B_t\cap{M}$ such that $\varrho(\beta_0,\beta_1)>\sup(Z)$. 
As $\varrho$ is subadditive, 
we may now find $\beta\in\{\beta_0,\beta_1\}$ such that $\varrho(\beta,\alpha)>\sup(Z)$. 
As $\beta\in B$, $\sup(\varrho_\beta``A)<\lambda$. As $\{\varrho,A\}\in{M}$, it follows that $\sup(\varrho_\beta``A)\in M$. In particular, $\varrho(\beta,\alpha)<\sigma$. 

\begin{claim} All of the following hold:
\begin{enumerate}
\item $(\beta,\alpha,\gamma)\in B\circledast A\circledast B$;
\item $b_\alpha<_{\lex}b_\beta<_{\lex}b_\gamma$;
\item $\max(\Delta``  \{\alpha,\beta,\gamma\}^2) =\delta$;
\item $\min(Z_{(\beta,\alpha,\gamma)})=\sigma$.
\end{enumerate}
\end{claim}
\begin{proof}
\begin{enumerate}
\item $\gamma$ was chosen to be in $B\setminus(\alpha+1)$. In addition, $\alpha\in A\setminus\upsilon_{{M}}$, whereas $\beta\in {M}\cap B$. 
\item As $\alpha\in A$ and $\beta\in B$, Condition~(iv) entails that $b_\alpha<_{\lex}b_\beta$.
In addition, as $b_\beta\restriction(\delta+1) = t= (b_{\gamma}\restriction\delta){}^\smallfrown\langle 0\rangle$ and $b_\gamma(\delta)=1$, we get that $b_\beta<_{\lex}b_\gamma$. 
\item By the previous analysis, $\Delta(\beta,\gamma)=\delta$.
In addition, $\Delta(\alpha,\gamma)=\chi<\delta$. Recalling that $|\Delta``  \{\alpha,\beta,\gamma\}^2|=2$, we are done.
\item 
By Clause~(3) and the fact that $\varrho(\alpha , \gamma )=\epsilon$, we infer that 
$Z_{(\beta,\alpha,\gamma)}=\{\zeta\in\lambda \setminus \varrho(\beta,\alpha) \mid \delta\ge e(\zeta,\epsilon)\}$. In particular, $\sigma\in Z_{(\beta,\alpha,\gamma)}$.
Now, if $\zeta:=\min(Z_{(\beta,\alpha,\gamma)})$ is $<\sigma$, then $\zeta\in Z$, contradicting the fact that $\varrho(\beta,\alpha)>\sup(Z)$.\qedhere
\end{enumerate}
\end{proof}

By the preceding claim and the definition of $c$,
$$c(\{\alpha,\beta,\gamma\})=h(\min(Z_{(\beta,\alpha,\gamma)}))=h(\sigma)=\tau,$$ as sought.

\medskip
\underline{\textbf{Case 2: There is $\beta\in B$ such that $\sup(\varrho_\beta``A)=\lambda$.}} 
As $\{\varrho,A\}\in{M}$, we may pick $\beta\in B\cap M$ such that $\sup(\varrho_\beta``A)=\lambda$.
Clearly, $|\varrho_\beta``A|=\lambda$.
Define a function $f:\varrho_\beta``A\rightarrow A$ via
$$f(\xi):=\min\{\alpha\in A\mid \varrho(\beta,\alpha)=\xi\}.$$
As $\{\beta,A,\varrho\}\in {M}$, we infer that $\varrho_\beta``A$, $f$ and $\im(f)$ are all in ${M}$.
Note that $f$ is injective, so that $|\im(f)|=\lambda$. 
It also follows that $$\Gamma:=\{ \gamma\in \im(f)\setminus(\beta+1)\mid \varrho(\beta,\gamma)\le\sigma\}$$ is bounded in $\im(f)$.

Appeal to Corollary~\ref{usefullemma} with $X:=\im(f)\setminus(\beta+1)$ and $i:=0$  to pick
$\gamma\in X\setminus(\Gamma\cup\upsilon_{{M}})$  such that for cofinally many $\delta<\mu$, the two hold:
\begin{enumerate}
\item $b_{\gamma}(\delta)=0$, and
\item $\{ \alpha\in X\mid \Delta(\alpha,\gamma)=\delta\}$ has size $\lambda$.
\end{enumerate}
Denote $D:=\{\delta<\mu\mid \text{Clauses (1) and (2) both hold}\}$.

As $\gamma\notin\Gamma$, $\epsilon:=\varrho(\beta,\gamma)$ is bigger than $\sigma$. 
Pick $\delta\in D$ above $\max\{\chi,e(\sigma,\epsilon)\}$. 
Since $\sigma\in\partial(e)$, the following set is bounded below $\sigma$:
$$Z:=\{\zeta<\sigma\mid e(\zeta,\epsilon)\le\delta\}.$$

Set $t:=(b_{\gamma}\restriction\delta){}^\smallfrown\langle 1\rangle$. 
From $\delta\in D$ we infer that $X_t$ is a set of size $\lambda$. As $t\in T\s{M}$ and $X\in{M}$, $X_t$ is in ${M}$.
As $\alpha\mapsto \varrho(\beta,\alpha)$ is injective over $X$,
we may find an $\alpha\in X_t\cap{M}$ such that $\varrho(\beta,\alpha)>\sup(Z)$. 
Because of the fact that $\{\beta,\alpha\}\in{M}$, we altogether get that $\sup(Z)<\varrho(\beta,\alpha)<\sigma$.

\begin{claim} All of the following hold:
\begin{enumerate}
\item $(\beta,\alpha,\gamma)\in B\circledast A\circledast A$;
\item $b_\gamma<_{\lex}b_\alpha<_{\lex}b_\beta$;
\item $\max(\Delta``  \{\alpha,\beta,\gamma\}^2)=\delta$;
\item $\min(Z_{(\alpha,\beta,\gamma)})=\sigma$.
\end{enumerate}
\end{claim}
\begin{proof}
\begin{enumerate}
\item $\beta$ was chosen to be in $B\cap M$, $\gamma$ was chosen to be in $A\setminus\upsilon_{{M}}$, whereas $\alpha\in M\cap A$ with $\alpha>\beta$.
\item As $\alpha\in A$ and $\beta\in B$, Condition~(iv) entails that $b_\alpha<_{\lex}b_\beta$.
In addition, $b_\alpha\restriction(\delta+1) = t= (b_{\gamma}\restriction\delta){}^\smallfrown\langle 1\rangle$ and $b_\gamma(\delta)=0$, so $b_\gamma<_{\lex}b_\alpha$.
\item By the previous analysis, $\Delta(\alpha,\gamma)=\delta$.
In addition, $\Delta(\alpha,\beta)=\chi<\delta$. Recalling that $|\Delta``\{\alpha,\beta,\gamma\}^2|=2$, we are done.
\item 
By Clause~(3) and the fact that $\varrho(\beta , \gamma )=\epsilon$, we infer that 
$Z_{(\alpha,\beta,\gamma)}=\{\zeta\in\lambda \setminus \varrho(\beta,\alpha) \mid \delta \ge e(\zeta,\epsilon) \}$. In particular, $\sigma\in Z_{(\alpha,\beta,\gamma)}$.
Now, if $\zeta:=\min(Z_{(\alpha,\beta,\gamma)})$ is $<\sigma$, then $\zeta\in Z$, contradicting the fact that $\varrho(\beta,\alpha)>\sup(Z)$.\qedhere
\end{enumerate}
\end{proof}

By the preceding claim and the definition of $c$,
$$c(\{\alpha,\beta,\gamma\})=h(\min(Z_{(\alpha,\beta,\gamma)}))=h(\sigma)=\tau,$$ as sought.
\end{proof}

\begin{thm}\label{thm53}
Suppose that:
\begin{itemize}
\item $\mu=\mu^{<\mu}$ is an infinite cardinal, $\lambda=\mu^+$ and $\kappa=\lambda^+$;
\item $\varrho:[\kappa]^2\rightarrow\lambda$ is a subadditive coloring of pairs;
\item $\lambda\nrightarrow[\mu;\lambda]^2_\lambda$ holds.
\end{itemize}

Then, there exists a corresponding coloring of triples $c:[\kappa]^3\rightarrow\lambda$ such that, 
for all $\tau<\lambda$ and disjoint $A,B\in\mathcal P(\kappa)$ satisfying the three:
\begin{enumerate}
\item[(i)] $\otp(A)=\otp(B)=\lambda$,
\item[(ii)] $\sup(A)=\sup(B)$,
\item[(iii)] $\varrho\restriction [A\cup B]^2$ witnesses $\U(\lambda,\lambda,\lambda,3)$,
\end{enumerate}
there exists $(\alpha,\beta,\gamma)\in[A\cup B]^3\setminus([A]^3\cup[B]^3)$ such that $c(\alpha,\beta,\gamma)=\tau$.
\end{thm}
\begin{proof} Let $d:[\kappa]^2\rightarrow\lambda$ be a coloring witnessing $\lambda\nrightarrow[\mu;\lambda]^2_\lambda$.
Let $T:={}^{<\lambda}2$.
Let $\vec{b}=\langle b_\xi\mid\xi<\kappa\rangle$ be an injective enumeration of elements of $\mathcal B(T)$. 
For $\alpha\neq\beta$ from $\kappa$, we write $\Delta(\alpha,\beta)$ for $\Delta(b_\alpha,b_\beta)$.
Likewise, for $B\s\kappa$, we write $T^{\branches B}$ for $T^{\branches \{b_\beta\mid \beta\in B\}}$.

For all $B\s\kappa$ and $t\in T$,
denote $B_t:=\{\beta\in B\mid t\sqsubseteq b_\beta\}$. 
Let $e:[\lambda]^2\rightarrow\mu$ be a map with injective fibers.
Let $h:\lambda\rightarrow\lambda$ be a surjection such that $S_\tau:=\{\sigma\in E^\lambda_\mu\mid h(\sigma)=\tau\}$ is stationary
for every $\tau<\lambda$.
For every $(\alpha,\beta,\gamma)\in\kappa\circledast\kappa\circledast\kappa$, define:
$$Z_{(\alpha,\beta,\gamma)}:=\{\zeta\in\lambda\setminus \varrho(\alpha,\gamma) \mid e(\Delta(\alpha,\beta),\varrho(\beta,\gamma))\ge e(\zeta,\varrho(\beta,\gamma))\}.$$

We define a coloring $c:[\kappa]^3\rightarrow\lambda$ by letting for all $\alpha<\beta<\gamma<\kappa$:
$$c(\alpha,\beta,\gamma):=\begin{cases}\label{thecmap}
d(\Delta(\alpha,\gamma),\Delta(\beta,\gamma)),&\text{if }\Delta(\alpha,\beta)<\Delta(\beta,\gamma);\\
d(\Delta(\alpha,\gamma),\varrho(\beta,\gamma)),&\text{if }\Delta(\alpha,\beta)=\Delta(\beta,\gamma);\\
d(\Delta(\alpha,\beta),\varrho(\alpha,\gamma)),&\text{if }\Delta(\alpha,\beta)>\Delta(\beta,\gamma)\ \&\ b_\alpha<_{\lex}b_\beta;\\
h(\min(Z_{(\alpha,\beta,\gamma)})),&\text{if }\Delta(\alpha,\beta)>\Delta(\beta,\gamma)\ \&\ b_\beta<_{\lex}b_\alpha.
\end{cases}$$

Suppose that $A,B\in\mathcal P(\kappa)$ are disjoint sets satisfying conditions (i)--(iii) above and let $\tau<\lambda$ be a prescribed color.   
By possibly passing to a cofinal subset of $A$, we may assume that $A=A'$ in the sense of Lemma~\ref{height lemma}. In particular, we may assume the existence of $\theta_A\le\lambda$ such that 
$T^{\branches A}\in\mathcal T(\lambda,\theta_A)$ and, in addition, if $\theta_A<\lambda$, then $|\mathcal B(T^{\branches A})|=\lambda$.
As $\mu^{<\mu}<\mu^+=\lambda$, this means that if $\theta_A<\lambda$, then $\theta_A\in E^\lambda_\mu$.
Likewise, we may assume the existence of $\theta_B\in E^\lambda_\mu\cup\{\lambda\}$ such that $T^{\branches B}\in\mathcal T(\lambda,\theta_B)$. 
Without loss of generality, we may also assume that $\theta_A\le\theta_B$.
The proof is now divided into two main cases.

\medskip

\underline{\bf Case 1: $\theta_A=\theta_B=\lambda$.} In this case, we shall need the following claim.
\begin{claim}\label{claim531} There exists $t\in T^{\branches B}$ satisfying all of the following:
\begin{itemize}
\item If $t\notin T^{\branches A}$, then $D:=\{\Delta(b_\beta,t)\mid \beta\in B\}$ has size $\mu$,
and there exists $t'\in T^{\branches A}$ incompatible with $t$ such that $\sup(D)>\Delta(t,t')$;
\item If $t\in T^{\branches A}$, then $D:=\{\Delta(b_\alpha,t)\mid \alpha\in A\}$ has size $\mu$.
\end{itemize}
\end{claim}
\begin{proof} There are two cases to consider:

$\br$ Suppose that there exists $\epsilon<\lambda$ such that $T^{\branches A}\cap T^{\branches B}\cap{}^\epsilon2=\emptyset$.
Pick $t'\in T^{\branches A}\cap{}^\epsilon2$.
For each $\alpha<\lambda$, pick $t_\alpha\in T^{\branches B}\cap{}^\alpha2$.
Then, by Lemma~\ref{lemma36a},
there exists $\alpha\in E^\lambda_\mu$  above $\epsilon$
such that $D:=\{\Delta(b_\beta,t_\alpha)\mid \beta\in B\}\cap\alpha$ is cofinal in $\alpha$.
To see that $t:=t_\alpha$ is as sought, notice that since $t\restriction\epsilon\notin T^{\branches A}$,
it must be the case that $\Delta(t,t')<\epsilon$.

$\br$ Otherwise. Thus, for each $\alpha<\lambda$, we may pick $t_\alpha\in T^{\branches A}\cap T^{\branches B}\cap{}^\alpha2$.
As $\langle t_\alpha\mid \alpha<\lambda\rangle\in\prod_{\alpha<\lambda} T^{\branches A}\cap{}^\alpha2$,
Lemma~\ref{lemma36a} provides an $\alpha\in E^\lambda_\mu$ 
such that $D:=\{\Delta(b_\beta,t_\alpha)\mid \beta\in A\}\cap\alpha$ is cofinal in $\alpha$.
So $t:=t_\alpha$ is as sought.
\end{proof}

Let $t\in T^{\branches B}$ and the corresponding $D$ be as in the claim.  There are two subcases to consider:

\medskip

\underline{\bf Subcase 1.1: $t\in T^{\branches A}$.} Pick $\bar A\in[A]^\mu$ such that $\bar D:=\{\Delta(b_\alpha,t)\mid \alpha\in \bar A\}$ is a $\mu$-sized subset of $D\cap \dom(t)$.
Clearly, $B':=B_t\setminus\sup(\bar A)$ is a set of size $\lambda$.
Since $T^{\branches B'}\in\mathcal T(\lambda,\lambda)$, $E:=\Delta[B'\circledast B']$ is of size $\lambda$, as well.
By the choice of $d$, we may now find $(\delta,\epsilon)\in\bar D\circledast E$ such that $d(\delta,\epsilon)=\tau$.
Pick $\alpha\in\bar A$ such that $\Delta(b_\alpha,t)=\delta$.
Finally, find $(\beta,\gamma)\in B'\circledast B'$ such that $\Delta(\beta,\gamma)=\epsilon$. 
Then
$$\Delta(\beta,\gamma)=\epsilon>\delta=\Delta(b_\alpha,t)=\Delta(\alpha,\gamma),$$
and hence $\Delta(\alpha,\beta)=\Delta(\alpha,\gamma)<\Delta(\beta,\gamma)$.
Altogether, $(\alpha,\beta,\gamma)\in A\circledast B\circledast B$, and
$$	 		c(\alpha,\beta,\gamma)=d(\Delta(\alpha,\gamma),\Delta(\beta,\gamma))=d(\delta,\epsilon)=\tau.$$

\smallskip

\underline{\bf Subcase 1.2: $t\notin T^{\branches A}$.} Pick $t'\in T^{\branches A}$ incompatible with $t$ such that $\sup(D)>\Delta(t,t')$.
As $\cf(\theta_B)=\mu$, we may now pick $\bar B\in[B]^\mu$ such that $\bar D:=\{\Delta(b_\beta,t)\mid \beta\in \bar B\}$ is a $\mu$-sized subset of $D$ with $\min(\bar D)>\Delta(t,t')$.

As $t'\in T^{\branches A}$, $A_{t'}$ has size $\lambda$ and so does $A':=A_{t'}\setminus\sup(\bar B)$.
As $t\in T^{\branches B}$, $B_t$ has size $\lambda$,
so since $\varrho\restriction [A\cup B]^2$ witnesses $\U(\lambda,2,\lambda,3)$,
the set $E:=\varrho[A'\circledast B_{t}]$ has size $\lambda$, as well. 
By the choice of $d$, find $(\delta,\epsilon)\in\bar D\circledast E$ such that $d(\delta,\epsilon)=\tau$. 
Find $\alpha\in\bar B$ such that $\Delta(b_\alpha,t)=\delta$.
Find $(\beta,\gamma)\in A'\circledast B_t$ such that $\varrho(\beta,\gamma)=\epsilon$. 
As $(\beta,\gamma)\in A_{t'}\circledast B_t$,
$$\Delta(\beta,\gamma)=\Delta(t',t)<\min(\bar D)\le \delta=\Delta(b_\alpha,t)=\Delta(\alpha,\gamma),$$
and hence $\Delta(\alpha,\beta)=\Delta(\beta,\gamma)$.
Altogether, $(\alpha,\beta,\gamma)\in B\circledast A\circledast B$, and
$$	 		c(\alpha,\beta,\gamma)=d(\Delta(\alpha,\gamma),\varrho(\beta,\gamma))=d(\delta,\epsilon)=\tau.$$

\medskip

\underline{\bf Case 2: $\theta_A<\lambda$.}
Set $\theta:=\theta_A$. We shall need the following claim.
\begin{claim}\label{claim532} There exist $\chi<\theta$,
$A'\in[A]^\lambda$ and $B'\in[B]^\lambda$ such that $\Delta[A'\times B']=\{\chi\}$.
\end{claim}
\begin{proof} Denote $\bar A:=\mathcal B(T^{\branches A})$ and $\bar B:=\mathcal B(T^{\branches B})$.
Recall that by our application of Lemma~\ref{height lemma},
$|\bar A|=\lambda$, and if $\theta_B<\lambda$, then $|\bar B|=\lambda$, as well.
We shall prove the claim by showing that there exist $\chi<\theta$ and a pair $(t,t')\in (T^{\branches A})_{\chi+1}\times (T^{\branches B})_{\chi+1}$
such that $\Delta(t,t')=\chi$. Indeed, once we have such a pair $(t,t')$, the sets $A':=A_t$ and $B':=B_{t'}$ would be as sought.

There are two cases to consider:

$\br$ If $\theta_B=\theta$, then set $\bar T:=T\cap{}^{<\theta}2$. 
In this case, $\bar B$ and $\bar A$
are $\lambda$-sized subsets of $\mathcal B(\bar T)$.
So Lemma~\ref{separation lemma2}(2) yields an $s\in\bar T$ together with $i\neq i'$ such that $s{}^\smallfrown\langle i\rangle\in\bar T^{\branches\bar A}\s T^{\branches A}$
and $s{}^\smallfrown\langle i'\rangle\in\bar T^{\branches \bar B}\s T^{\branches B}$.
Evidently, $\chi:=\dom(s)$, $t:=s{}^\smallfrown\langle i\rangle$ and $t':=s{}^\smallfrown\langle i'\rangle$ 
are as sought.

$\br$ If $\theta_B>\theta$, then pick $r\in(T^{\branches B})_\theta$.
For every $a\in \bar A\setminus\{r\}$,
$\chi_a:=\Delta(a,r)$ is smaller than $\theta$.
As $|\bar A|=\lambda$, we can find $\chi<\theta$ such that $\lambda$-many $a$'s in $\bar A\setminus\{r\}$ satisfy $\chi_a=\chi$.
As the $\chi^{\text{th}}$ level of $T^{\branches A}$ has size $<\lambda$, we may then find $t\in (T^{\branches A})_{\chi+1}$ such that
that $\lambda$-many $a$'s in $\bar A\setminus\{r\}$ satisfy $\chi_a=\chi$ and $a\restriction(\chi+1)=t$.
Clearly, $\chi$, $t$ and $t':=r\restriction(\chi+1)$ are as sought.
\end{proof}

Let $\chi$ be given by the claim. For notational simplicity, we shall assume that $\Delta[A\times B]=\{\chi\}$.
Let ${M}$ be an elementary submodel of $\mathcal H_{(2^{\lambda})^+}$  containing $\{\chi,A,B,\allowbreak T^{\branches A},\allowbreak T^{\branches B},\varrho,d\}$
such that $\sigma:={M}\cap\lambda$ is in $S_\tau$. In particular, $|M|=\mu$.
Denote $\upsilon:=\sup(A)$ and $\upsilon_{{M}}:=\sup({M}\cap\upsilon)$.
Note that since $T^{\branches A}\in\mathcal T(\lambda,\theta)$ and as $\lambda=\mu^+>|\theta|$, it follows that $T^{\branches A}$ has size $\le\mu$. 
So, $T^{\branches A}\s{M}$.

Consider the following sets:
\begin{itemize}
\item $A^0:=\{\alpha\in A\mid|\varrho_\alpha[B]|=\lambda\}$,
\item $A^1:=\{\alpha\in A\mid|\varrho_\alpha[B]|\leq\mu\}$.
\end{itemize}
Observe that $A^0,A^1\in{M}$. 
We examine two subcases.

\smallskip

\underline{\bf Subcase 2.1: $A^0$ has size $\lambda$.} 
Appeal to Corollary~\ref{usefullemma} with $X:=A^0$ and $i:=0$ to pick
$\alpha\in A^0$ such that for cofinally many $\delta<\theta$, the two hold:
\begin{enumerate}
\item $b_\alpha(\delta)=0$, and
\item $\{\beta\in A\mid\Delta(\alpha,\beta)=\delta\}$ has size $\lambda$.
\end{enumerate}
Since $\theta\in E^\lambda_\mu$, $D:=\{\delta<\theta\mid \text{Clauses (1) and (2) both hold}\}$ has size $\mu$.
For each $\delta\in D$, use Clause~(2) to fix $\beta_\delta\in A$ above $\alpha$ such that $\Delta(\alpha,\beta_\delta)=\delta$.

Consider $\varsigma:=\sup\{\beta_\delta\mid\delta<\mu\}$.
As $|B\cap\varsigma|\leq\mu$,
the fact that $\alpha\in A^0$ implies that $E:=\varrho_\alpha[B\setminus\varsigma]$ has size $\lambda$. 
By the choice of $d$, then, we may pick $\delta\in D\setminus(\chi+1)$ and $\epsilon\in E$ above $\delta$ such that $d(\delta,\epsilon)=\tau$. 
Pick $\gamma\in B\setminus\varsigma$ such that $\epsilon=\varrho(\alpha,\gamma)$.
Clearly, $\alpha<\beta_\delta<\gamma$.

Recall that $\Delta(\alpha,\beta_\delta)=\delta>\chi=\Delta(\beta_\delta,\gamma)$.
Since $b_\alpha(\delta)=0$, we conclude that $b_{\beta_\delta}(\delta)=1$
and $b_\alpha<_{\lex}b_{\beta_\delta}$.
Altogether, $(\alpha,\beta_\delta,\gamma)\in A\circledast A\circledast B$, and
$$c(\alpha,\beta_\delta,\gamma)=d(\Delta(\alpha,\beta_\delta),\varrho(\alpha,\gamma))=d(\delta,\epsilon)=\tau,$$
as sought.

\smallskip

\underline{\bf Subcase 2.2: $A^1$ has size $\lambda$.} 
By Condition~(iii), find $A'\in[A^1]^{\lambda}$ and $B'\in[B]^{\lambda}$ such that 
$\min(\varrho[A'\circledast B'])>\sigma$.
Appeal to Corollary~\ref{usefullemma} with $X:=A'$ and $i:=0$ to pick
$\beta\in A'\setminus\upsilon_{{M}}$ such that for cofinally many $\delta<\theta$, the two hold:
\begin{enumerate}
\item $b_\beta(\delta)=0$, and
\item $\{\alpha\in A'\mid\Delta(\alpha,\beta)=\delta\}$ has size $\lambda$.
\end{enumerate}
Since $\theta\in E^\lambda_\mu$, $D:=\{\delta<\theta\mid \text{Clauses (1) and (2) both hold}\}$ has size $\mu$.

Pick $\gamma\in B'\setminus(\beta+1)$ arbitrarily.
As $(\beta,\gamma)\in A'\circledast B'$, the ordinal $\epsilon:=\varrho(\beta,\gamma)$ is bigger than $\sigma$.
Since $e_\epsilon$ is an injection to $\mu=|D|$,
we may pick $\delta\in D$ such that $e(\delta,\epsilon)>\max\{e(\chi,\epsilon),e(\sigma,\epsilon)\}$.
In addition, since $\mu=\cf(\sigma)$,
the following set is bounded below $\sigma$:
$$Z:=\{\zeta<\sigma\mid e(\zeta,\epsilon)\le e(\delta,\epsilon)\}.$$
Set $t:=(b_\beta\restriction\delta){}^\smallfrown\langle 1\rangle$. 
From $\delta\in D$ we infer that $(A')_{t}$ has size $\lambda$.
As $t\in T^{\branches A}\s M$, in particular, $(A^1)_{t}$ is a set of size $\lambda$ lying in ${M}$.
By Condition~(iii) and elementarity of ${M}$ one can find $\alpha_0\neq\alpha_1$ in $(A^1)_{t}\cap{M}$ such that $\varrho(\alpha_0,\alpha_1)>\sup(Z)$. 
As $\varrho$ is subadditive, 
we may now find $\alpha\in\{\alpha_0,\alpha_1\}$ such that $\varrho(\alpha,\gamma)>\sup(Z)$. 
Note that, as $\{\alpha,B,\varrho\}\in{M}$, and as $\alpha\in A^1$, $\varrho_\alpha[B]$ is a set of size no more than $\mu$ lying in ${M}$,
so that $\sup(\varrho_\alpha[B])\in{M}$.
In particular, $\varrho(\alpha,\gamma)<\sigma$. 
\begin{claim}
All of the following hold:
\begin{enumerate}
\item $(\alpha,\beta,\gamma)\in A\circledast A\circledast B$;
\item $\Delta(\alpha,\beta)>\Delta(\beta,\gamma)$;
\item $b_\beta<_{\lex}b_\alpha$;
\item $\min(Z_{(\alpha,\beta,\gamma)})=\sigma.$
\end{enumerate}	
\end{claim}
\begin{proof}\begin{enumerate}
\item $\gamma$ was chosen to be in $B\setminus(\beta+1)$. In addition, $\beta\in A'\setminus\upsilon_{{M}}$, whereas $\alpha\in M\cap A$. 

\item Since $\Delta(\alpha,\beta)=\Delta(t,b_\beta)=\delta>\chi=\Delta(\beta,\gamma)$.

\item By the definition of $t$ and since $\delta\in D$.

\item As $\delta=\Delta(\alpha,\beta)$ and $\varrho(\beta,\gamma)=\epsilon$,
we infer that
$Z_{(\alpha,\beta,\gamma)}=\{\zeta\in\lambda\setminus\varrho(\alpha,\gamma)\mid e(\delta,\epsilon)\ge e(\zeta,\epsilon)\}$.
In particular, $\sigma\in Z_{(\alpha,\beta,\gamma)}$.
Now, if $\zeta:=\min(Z_{(\alpha,\beta,\gamma)})$ is below $\sigma$, 
then $\zeta\in Z$, contradicting the fact $\varrho(\alpha,\gamma)>\sup(Z)$. 	\qedhere
\end{enumerate}
\end{proof} 

By the preceding claim and the definition of $c$,
$$c(\alpha,\beta,\gamma)=h(\min(Z_{(\alpha,\beta,\gamma)})=h(\sigma)=\tau,$$
as sought. 
\end{proof}

\section{Countably many colors}\label{section5}

The main result of this section asserts that
$${\lambda^+\nsrightarrow[\lambda,\lambda]^3_\omega}$$ holds, provided that $\lambda=\mu^+$ for an infinite cardinal $\mu=\mu^{<\mu}$.
The idea of the proof is to build on the colorings $c:[\lambda^+]^3\rightarrow\lambda$ given by Theorems \ref{use of stability} and \ref{thm53}
with respect to the subadditive coloring $\rho:[\lambda^+]^2\rightarrow\lambda$ given by Fact~\ref{all the rho}.
By Clause~(iii) of these theorems, we must address the problematic case in which the two sets $A,B$ of Definition~\ref{narrowsup}
do not satisfy that $\rho\restriction[A\cup B]^2$ witnesses $\U(\lambda,\lambda,\lambda,3)$.
Anyone that is familiar with \cite[\S10]{MR2355670} would probably suggest to use the oscillation of \cite[\S8]{MR2355670} in this problematic case,
and this indeed works. Unfortunately, to verify that this works in the rectangular context, we had to reopen and tweak the proofs.
The experts may want to skip directly to Corollary~\ref{corofstable}. The newcomers may benefit from the modular exposition.

\begin{setup}\label{setup51}
For the rest of this section, $\kappa$ stands for a regular uncountable cardinal, $\Upsilon$ is a large enough regular cardinal (e.g., $(2^\kappa)^+$), 
and we fix some $C$-sequence $\vec C=\langle C_\beta\mid\beta<\kappa\rangle$.
We shall also assume that $\otp(C_\beta)=\cf(\beta)$ for all $\beta<\kappa$, though this will only play a role in the proof of Lemma~\ref{stable case_relaxed} below.
\end{setup}

The items of the next definition correspond to Definitions 8.1.4, 6.3.1 and 8.1.1 of \cite{MR2355670}, where the last item is a non-essential strengthening of the latter.
\begin{defn}
A subset $\Gamma\s\kappa$ with $\cf(\otp(\Gamma))>\omega$ is said to be:
\begin{itemize}
\item \emph{$\vec{C}$-stationary} iff 	
$\bigcup_{\beta\in\Gamma}(\acc(C_\beta)\cup\{\beta\})$ is stationary in $\sup(\Gamma)$;
\item \emph{$\vec{C}$-nontrivial} iff 
for every club $D\s\sup(\Gamma)$, there exists $\alpha\in\Gamma$ such that $D\cap\alpha\nsubseteq C_\beta$ for all $\beta\in\Gamma$;
\item \emph{$\vec C$-oscillating}
iff for every club $D\s \sup(\Gamma)$, there exist $\beta\in\Gamma$ and an increasing sequence $\langle\delta_j\mid j<\omega\rangle$ of 
ordinals in $D\setminus C_\beta$ such that $(\delta_j,\delta_{j+1} )\cap C_\beta\neq\emptyset$ for all $j<\omega$.
\end{itemize}
\end{defn}

\begin{lemma}\label{nontrivial and stationary is unbounded}
Suppose that $\Gamma\s\kappa$ is such that $\cf(\otp(\Gamma))>\omega$. 

If $\Gamma$ is $\vec{C}$-nontrivial and $\vec{C}$-stationary, then $\Gamma$ is $\vec{C}$-oscillating.
\end{lemma}
\begin{proof} Suppose that $\Gamma$ is $\vec{C}$-nontrivial and $\vec{C}$-stationary.
By the latter,
$\Delta:=\bigcup_{\beta\in\Gamma}(\acc(C_\beta)\cup\{\beta\})$ is a stationary subset of $\theta$.
For each $\delta\in\Delta$, pick $\beta_\delta\in \Gamma$ such that $\sup(C_{\beta_\delta}\cap\delta)=\delta$.

Next, to verify that $\Gamma$ is $\vec C$-oscillating, let $D\s \sup(\Gamma)$ be a given club.
\begin{claim} There exists $\delta\in \Delta$ such that
$\sup(D\cap\delta\setminus  C_{\beta_\delta})=\delta$.
\end{claim}
\begin{proof} Suppose not. Then, for every $\delta\in\Delta$,
$\epsilon_\delta:=\sup(D\cap\delta\setminus  C_{\beta_\delta})$ is smaller than $\delta$.
Fix $\epsilon<\sup(\Gamma)$ for which $S:=\{ \delta\in \Delta\mid \epsilon_\delta<\epsilon<\delta\}$ is stationary.
Now, consider the club $D':=D\setminus\epsilon$. Then, for every $\alpha\in\Gamma$, letting $\delta:=\min(S\setminus\alpha)$,
it is the case that $D'\cap\alpha\s D\cap[\epsilon,\delta)\s C_{\beta_\delta}$.
This contradicts the fact that $\Gamma$ is $\vec C$-nontrivial.
\end{proof}

Let $\delta$ be given by the claim.
As $\sup(D\cap\delta\setminus  C_{\beta_\delta})=\delta=\sup(C_{\beta_\delta}\cap\delta)$,
it is easy to recursively construct an increasing sequence $\langle\delta_j\mid j<\omega\rangle$ of 
ordinals in $D\cap\delta\setminus C_{\beta_\delta}$ such that $(\delta_j,\delta_{j+1} )\cap C_{{\beta_\delta}}\neq\emptyset$ for all $j<\omega$.
\end{proof}

\begin{defn} For two disjoint sets of ordinals $y$ and $z$,  we say that $P$ is a \emph{$y$-convex subset of $z$}
iff one of the following occurs:
\begin{itemize}
\item $P=\{ \zeta\in z\mid \zeta<\alpha\}$ and $\alpha=\min(y)$;
\item $P=\{ \zeta\in z\mid \beta<\zeta\}$ and $\beta=\max(y)$;
\item $P=\{ \zeta\in z\mid \alpha<\zeta<\beta\}$ and $\alpha<\beta$ are two consecutive elements of $y$.
\end{itemize}
\end{defn}
Note that if $P$ and $Q$ are nonempty $y$-convex subsets of $z$, then either $P<Q$ or $Q<P$.

\begin{defn}[Todor\v{c}evi\'{c}, {\cite[\S8]{MR2355670}}]  \label{2dim-osc}
For an ordinal $\varepsilon<\kappa$, define a function $\Osc_\varepsilon:[\mathcal{P}(\kappa)]^2 \rightarrow\mathcal P(\mathcal P(\kappa))$ via
$$\Osc_\varepsilon(x,y):=\begin{cases}\{ P\mid P\text{ is a nonempty }y\text{-convex subset of }x\setminus\varepsilon\},&\text{if }y\cap x\s\varepsilon;\\
\emptyset,&\text{otherwise}.	\end{cases}$$ 

Then the \emph{oscillation mapping}
$\osc_\varepsilon:[\mathcal{P}(\kappa)]^2 \rightarrow\card(\kappa+1)$ is defined via
$\osc_\varepsilon(x,y):=|\Osc_\varepsilon(x,y)|$. 
\end{defn}
\begin{remark} \begin{enumerate}
\item If we omit the subscript $\varepsilon$, then $\Osc(x,y)$ and $\osc(x,y)$ are understood to be $\Osc_\varepsilon(x,y)$ and $\osc_\varepsilon(x,y)$ for $\varepsilon:=\ssup(x\cap y)$.
\item For all $\varepsilon<\alpha<\beta<\kappa$ such that $C_\alpha\cap C_\beta\s\varepsilon$,
$C_\alpha\setminus\varepsilon$ and $C_\beta$ have no common accumulation points,
and hence $\Osc_\varepsilon(C_\alpha,C_\beta)$ is finite.
In this case, we shall identify $\Osc_\varepsilon(C_\alpha,C_\beta)$ with its increasing enumeration $\langle P_0,\ldots,P_n\rangle$. 
\end{enumerate}
\end{remark}
The next lemma makes explicit some of the features that are present in the proof of \cite[Lemma~8.1.2]{MR2355670}. 

\begin{lemma}[Todor\v{c}evi\'{c}]\label{osclemma} 
Suppose that $\Gamma$ is a cofinal subset of some $\theta\le\kappa$ of uncountable cofinality,
and that $\Gamma$ is $\vec C$-oscillating.
For every cofinal $E\s\theta$,
there exists $\beta\in\Gamma$ such that for every positive integer $n$,
there are $\alpha\in\Gamma\cap\beta$ and $\varepsilon\in E\cap\alpha$ such that all of the following hold:
\begin{itemize}
\item $\osc_\varepsilon(C_\alpha,C_\beta)=n$;
\item for every $j<n$, there is a pair $\epsilon<\epsilon'$ of ordinals in $E\setminus C_\alpha$ for which
$$\Osc_\varepsilon(C_\alpha,C_\beta)(j)=C_\alpha\cap (\epsilon,\epsilon').$$
\end{itemize}
\end{lemma}
\begin{proof} Set $\mu:=\cf(\theta)$, and fix a map $\psi:\mu\rightarrow\theta$ whose image is cofinal in $\theta$.
Let $\mathcal M $ be a continuous $\in$-chain of length $\mu$ consisting of elementary submodels $M\prec  H_\Upsilon$ with $M\cap\mu\in\mu$ and
$\{\psi,\vec C,\Gamma,E\}\in M$.
It follows that $D := \{ \sup(M \cap \theta)\mid M \in \mathcal M\}$ constitutes a club in $\theta$. 
Recalling that $\Gamma$ is $\vec C$-oscillating, pick $\beta\in\Gamma$
and an increasing sequence $\langle\delta_j\mid j<\omega\rangle$ of
ordinals in $D\setminus C_\beta$ such that $(\delta_j,\delta_{j+1} )\cap C_\beta\neq\emptyset$ for all $j<\omega$.	
By possibly replacing $\delta_j$ by $\delta_{j+1}$, we may assume that $C_\beta\cap\delta_0$ is nonempty.
For every $j<\omega$, since $\delta_j\in D\setminus C_\beta$, pick $M_j\in\mathcal M$ such that $\sup(M_j\cap\theta)=\delta_j$,
and note that $\gamma_j:=\sup(C_\beta\cap\delta_j)$ and $\Omega_j:=\min(M_j\cap\theta\setminus\gamma_j)$ are both smaller than $\delta_j$.
So, for every $j<\omega$:
$$0<\sup(C_\beta\cap\delta_j)=\gamma_j\le\Omega_j<\delta_j<\gamma_{j+1}<\beta.$$

For each $k<\omega$, let $\mathcal I_k$ denote the collection of all increasing sequences $\vec I=\langle I_j \mid j\le k\rangle$ of closed intervals in $\theta$.
Now, let $n$ be a positive integer
and we shall find $\alpha\in\Gamma\cap\beta$ and $\varepsilon\in E\cap\alpha$ as in the conclusion of the lemma.

Define a sequence of collections $\langle\mathcal F_{n-i}\mid i\le n\rangle$ by recursion on $i\le n$, as follows:

$\br$ For $i=0$, let $\mathcal F_n$ be the set of all $\langle I_j\mid j\le n\rangle\in \mathcal I_n$
such that the following two hold:
\begin{itemize}
\item[(1)] $I_0 = [0,\Omega_0]$;
\item[(2)] $\alpha:=\max(I_n)$ belongs to $\Gamma$, $C_\alpha\s I_0\cup\cdots\cup I_n$, and $C_\alpha\cap I_j\neq\emptyset$ for every $j\le n$.
\end{itemize}

$\br$ For every $i<n$ such that $\mathcal F_{n-i}$ has already been defined,
let $\mathcal F_{n-i-1}$ be the collection of all
$\vec I\in\mathcal I_{n-i-1}$ with the property that for every $\epsilon < \theta$ there exists a closed interval $I\s(\epsilon,\theta)$ such that $\vec I{}^\smallfrown \langle I\rangle\in\mathcal F_{n-i}$.

\begin{claim} $\langle [0,\Omega_0]\rangle \in\mathcal F_0$.
\end{claim}
\begin{proof} 
For every $j\le n$, define:
$$ I_j:=\begin{cases}
[0,\Omega_0],&\text{if }j=0;\\
[\delta_{j-1},\Omega_j],&\text{if }0<j<n;\\
[\delta_{n-1},\beta],&\text{otherwise}.
\end{cases}$$

We shall prove by induction on $i\le n$
that $\langle  I_j\mid j\le n-i\rangle\in\mathcal F_{n-i}$.
The base case is immediate, since 
$\langle  I_j\mid j\le n\rangle$ satisfies requirements (1) and (2), with $\beta$ playing the role of $\alpha$.

Next, suppose that we are given $i<n$ for which $\langle  I_j\mid j\le n-i\rangle\in\mathcal F_{n-i}$ has been established.
Note:
\begin{itemize}
\item $\langle \mathcal F_k\mid k\le n\rangle\in M_0\s M_{n-i-1}$;
\item $\langle  I_j\mid j\le n-i-1\rangle\in M_{n-i-1}\cap \mathcal F_{n-i}$;
\item $ I_{n-i}\in M_{n-i}\setminus M_{n-i-1}$.
\end{itemize}
So, by elementarity of $M_{n-i-1}$, $\langle  I_j\mid j\le n-i-1\rangle\in\mathcal F_{n-i-1}$.
\end{proof}

It follows that we may recursively construct a sequence $\langle I_j\mid j\le n\rangle$ such that:
\begin{enumerate}
\item[(3)] $I_0=[0,\Omega_0]$, so that $\langle I_0\rangle\in\mathcal F_0\cap M_0$;
\item[(4)] $\langle I_j\mid j\le k+1\rangle\in \mathcal F_{k+1}\cap M_k$ for every $k< n$;
\item[(5)] $I_{j+1}\s (\min(E\setminus\Omega_j+1),\theta)$ for every $j<n$.
\end{enumerate}

For each $j<n$, denote $\epsilon_j:=\min(E\setminus\Omega_j+1)$,
and note that since $I_{j+1}$ and $E$ are in $M_j$, $\epsilon_j':=\min(E\setminus\max(I_{j+1})+1)$ is $<\delta_j$.
Denote $\gamma_j':=\min(C_\beta\setminus\gamma_j+1)$ 
so that $\gamma_j<\gamma_j'$ are two consecutive elements of $C_\beta$.
Since $\sup(C_\beta\cap\delta_j)=\gamma_j$,
altogether, 
$$I_{j+1}\s(\epsilon_j,\epsilon_j')\s (\Omega_j,\delta_j)\s(\gamma_j,\gamma_j')\s (\gamma_j,\gamma_{j+1}).$$

Now, put $\alpha:=\max(I_n)$. Then $\alpha\in \Gamma\cap\delta_{n-1}\s\Gamma\cap\beta$, $C_\alpha\s I_0\cup\cdots\cup I_n$ and $C_\alpha\cap I_j\neq\emptyset$ for every $j\le n$.
So $(C_\alpha\setminus I_0)\s \bigcup_{j<n}(\Omega_j,\delta_j)$.
On the other hand, $(C_\beta\setminus I_0)\cap(\bigcup_{j<n}(\Omega_j,\delta_j))=\emptyset$.
Therefore, for $\varepsilon:=\epsilon_0$, we get that $$C_\alpha\cap C_\beta\s(\Omega_0+1)\s\varepsilon.$$
\begin{claim} $\{\epsilon_j,\epsilon_j'\mid j<n\}\cap C_\alpha=\emptyset$.
\end{claim}
\begin{proof} Suppose not, and fix $j<n$ such that $\{\epsilon_j,\epsilon_j'\}\cap C_\alpha\neq\emptyset$.
As $\max(I_0)=\Omega_0\le\Omega_j<\epsilon_j<\epsilon_j'$,
we may fix some $i<n$ such that 
$\{\epsilon_j,\epsilon_j'\}\cap I_{i+1}\neq\emptyset$.
Recalling that $I_{j+1}\s(\epsilon_j,\epsilon_j')$, it must be the case that $i\neq j$.
Note:

$\br$ If $i<j$, then $\gamma_{i+1}\le\Omega_{i+1}\le\Omega_j<\epsilon_j<\epsilon_j'$. 

$\br$ If $i>j$, then $\epsilon_j<\epsilon_j'<\delta_j<\gamma_{j+1}\le\gamma_i$.

So, both options contradict the fact that $I_{i+1}\s(\gamma_i,\gamma_{i+1})$. 
\end{proof}

By Clause~(2), for every $j\le n$, $C_\alpha\cap I_{j+1}\neq\emptyset$. Altogether, for every $\varsigma\in(\Omega_0,\varepsilon]$:
$$\begin{aligned}\Osc_\varsigma(C_\alpha,C_\beta)&=\langle C_\alpha\cap I_{j+1}\mid j<n\rangle
\\&=\langle C_\alpha\cap (\gamma_j,\gamma_j')\mid j<n\rangle
\\&=\langle C_\alpha\cap (\epsilon_j,\epsilon_j')\mid j<n\rangle.
\end{aligned}$$
In particular, $\osc_\varepsilon(C_\alpha,C_\beta)=n$.
\end{proof}
\begin{remark} In the preceding proof,
in the special case that $\kappa=\theta$ or $\kappa=(\cf(\theta))^+$, 
one can secure that $\Omega_0$ be equal to $\gamma_0$. So, in this case, we would get that $\max(C_\alpha\cap C_\beta)=\Omega_0$,
meaning that the conclusion of the lemma remains valid also after omitting the subscript $\varepsilon$.
\end{remark}

\begin{defn}\label{chimap}
Define $\chi:[\kappa]^3\rightarrow\omega$ by letting for all $\alpha<\beta<\kappa$:
$$\chi(\alpha,\beta,\gamma):=\max\{k<\omega\mid \Tr(\alpha,\gamma)(k)=\Tr(\beta,\gamma)(k)\}.$$
\end{defn}

\begin{defn}[{\cite[Definition~10.3.1]{MR2355670}}]
A subset $A\s\kappa$ is said to be \emph{stable} if  $\chi``[A]^3$ is finite.
Otherwise, we say that $A$ is \emph{unstable}.
\end{defn}

Similar to \cite[Definition~10.3.3]{MR2355670}, we use $\chi$ to derive the following stepping-up of the two-dimensional oscillation.
\begin{defn}
The \emph{three-dimensional oscillation mapping}, $\overline{\osc}:[\kappa]^3\rightarrow\omega$ is defined on the basis of the two-dimensional oscillation defined in Definition~\ref{2dim-osc} via:
$$\overline{\osc}(\alpha,\beta,\gamma):=\osc_\alpha(C_{\Tr(\alpha,\beta)(\chi(\alpha,\beta,\gamma))},C_{\Tr(\alpha,\gamma)(\chi(\alpha,\beta,\gamma))}).$$
\end{defn}

We now verify a rectangular version of	\cite[Lemma~10.3.4]{MR2355670}:
\begin{lemma}[Todor\v{c}evi\'{c}]\label{unstable case}
Suppose that $B$ is a cofinal subset of some $\theta\le\kappa$ of uncountable cofinality,
and that every cofinal subset of $B$ is unstable. 

Then, for every cofinal $A\s\theta$ and every positive integer $n$, there exists
$(\alpha,\beta,\gamma)\in A\circledast B\circledast B$ such that $\overline{\osc}(\alpha,\beta,\gamma)=n$.

\end{lemma}
\begin{proof} 
For each $\delta<\theta$, let $\beta_\delta:=\min(B\setminus(\delta+1))$
and $\Lambda_\delta:=\lambda_2(\delta,\beta_\delta)$.
By Fodor's lemma, fix $\Lambda<\theta$, $k<\omega$ and a stationary $S\s\acc(\theta)$ such that, for all $\delta\in S$:
\begin{enumerate}
\item $\Lambda_\delta\le\Lambda$;
\item $\rho_2(\last{\delta}{\beta_\delta},\beta_\delta)=k$;
\item for every $\bar\delta<\delta$, $\beta_{\bar\delta}<\delta$.
\end{enumerate}

\begin{claim}\label{claim5251} For every $\delta\in S$ and every ordinal $\alpha$ with $\Lambda<\alpha<\last{\delta}{\beta_\delta}$:
\begin{itemize}
\item $\Tr(\alpha,\beta_\delta)\restriction(k+1)=\Tr(\delta,\beta_\delta)\restriction(k+1)$, and
\item $\Tr(\alpha,\beta_\delta)(k)=\last{\delta}{\beta_\delta}$.
\end{itemize}
\end{claim}
\begin{proof} By Remark~\ref{concatenation}.
\end{proof}

Let $\Gamma:=\{\last{\delta}{\beta_\delta}\mid\delta\in S\}$.
For each $\xi\in\Gamma$, pick $\delta(\xi)\in S$ such that $\xi=\last{\delta(\xi)}{\beta_{\delta(\xi)}}$.
Note that $\delta(\xi)\le\xi\le\beta_{\delta(\xi)}$.

\begin{claim}\label{claim6122} $\Gamma$ is $\vec{C}$-oscillating.
\end{claim}
\begin{proof} As $\delta\in\acc(\last{\delta}{\beta_\delta})\cup\{\last{\delta}{\beta_\delta}\}$ for every $\delta\in S$,
we have $S\s\bigcup_{\xi\in\Gamma}\acc(C_\xi)\cup\{\xi\}$. So $\Gamma$ is $\vec C$-stationary.
By Lemma~\ref{nontrivial and stationary is unbounded}, 
it thus suffices to prove that $\Gamma$ is $\vec C$-nontrivial.
Towards a contradiction, suppose this is not so, and fix a club $D\s\theta$ such that, for every $\alpha\in\Gamma$ there exists $\beta\in\Gamma$ with $D\cap\alpha\s C_\beta$. 
As $\sup(\Gamma)=\theta$, we may then recursively construct a sparse enough cofinal subset $X\s\Gamma$ with the property that for every pair $\xi<\xi'$ of ordinals from $X$,
all of the following hold:
\begin{itemize}
\item $\Lambda<\delta(\xi)$;
\item $D\cap(\beta_{\delta(\xi)},\delta(\xi'))\neq\emptyset$;
\item $D\cap\beta_{\delta(\xi)}\s C_{\xi'}$.
\end{itemize}

As $B':=\{ \beta_{\delta(\xi)}\mid \xi\in X\}$ is a cofinal subset of $B$,
it must be unstable. We shall reach a contradiction by showing that $\chi``[B']^3=\{k\}$.
To this end, let $\alpha<\beta<\gamma$ be a triple of ordinals from $B'$.
Fix a triple $\xi<\xi'<\xi''$ of ordinals from $X$ such that $\alpha=\beta_{\delta(\xi)}$, $\beta=\beta_{\delta(\xi')}$, and $\gamma=\beta_{\delta(\xi'')}$.
Then:
\begin{itemize}
\item $\Lambda<\delta(\xi)<\alpha<\delta(\xi')<\beta<\delta(\xi'')\le \xi''\le\gamma$;
\item $D\cap(\alpha,\delta(\xi'))\neq\emptyset$;
\item $D\cap\beta\s C_{\xi''}$.
\end{itemize}

Pick $\iota\in D\cap(\alpha,\delta(\xi'))\neq\emptyset$, so that $\iota\in D\cap(\alpha,\beta)\s C_{\xi''}$.
Appealing to Claim~\ref{claim5251} with $\delta'':=\delta(\xi'')$, we infer that:
\begin{itemize}
\item $\Tr(\alpha,\gamma)\restriction(k+1)=\Tr(\delta'',\gamma)\restriction(k+1)=\Tr(\beta,\gamma)\restriction(k+1)$, and
\item $\Tr(\alpha,\gamma)(k)=\xi''=\Tr(\beta,\gamma)(k)$.
\end{itemize}

Therefore $$\Tr(\alpha,\gamma)(k+1)=\min(C_{\xi''}\setminus\alpha)\le\iota<\beta\le\min(C_{\xi''}\setminus\beta)=\Tr(\beta,\xi)(k+1).$$
Recalling Definition~\ref{chimap}, this indeed means that $\chi(\alpha,\beta,\gamma)=k$.
\end{proof}

Now, given a cofinal $A\s\theta$ and a positive integer $n$, appeal to Lemma~\ref{osclemma} with $E:=\acc^+(A\setminus\Lambda)$ to find a pair $(\xi,\zeta)\in \Gamma\circledast \Gamma$ and an ordinal $\varepsilon\in E\cap\xi$ such that:
\begin{itemize}
\item $\osc_\varepsilon(C_\xi,C_{\zeta})=n+1$, and
\item for every $j<n+1$, there is a pair $\epsilon<\epsilon'$ of ordinals in $E\setminus C_\xi$ for which
$$\Osc_\varepsilon(C_\xi,C_{\zeta})(j)=C_\xi\cap (\epsilon,\epsilon').$$
\end{itemize}

Let $\epsilon<\epsilon'$ be a pair of ordinals witnessing the case $j=0$ of the preceding. Clearly,
$$\Osc_\varepsilon(C_\xi,C_{\zeta})(0)=C_\xi\cap [\epsilon,\epsilon'].$$
Since $\osc(C_\xi,C_{\zeta})>1$ and $\xi<\zeta$, 
we may fix two consecutive elements $\bar\alpha<\bar\beta$ of $C_\zeta$ such that
$$\Osc_\varepsilon(C_\xi,C_{\zeta})(1)=C_\xi\cap (\bar\alpha,\bar\beta).$$
So, $\epsilon<\epsilon'\le\bar\alpha<\xi$.

Since $C_\xi$ is a closed subset of $\xi$, and $\epsilon'\in \acc^+(A\setminus\Lambda)\cap(\xi\setminus C_\xi)$, we may pick a large enough $\alpha\in A\cap(\Lambda,\epsilon')$ such that
$$\Osc_\varepsilon(C_\xi,C_{\zeta})(0)\s (\epsilon,\alpha).$$
In particular, $C_\xi\cap C_{\zeta}\s\varepsilon\s\alpha$, and 
$$\osc_\alpha(C_\xi,C_{\zeta})=\osc_\varepsilon(C_\xi,C_{\zeta})-1=n.$$

Denote $\bar\delta:=\delta(\xi)$ and $\delta:=\delta(\zeta)$.
Then $(\bar\delta,\delta)\in[S]^2$, $\xi=\last{\bar\delta}{\beta_{\bar\delta}}$ and $\zeta=\last{\delta}{\beta_{\delta}}$.
Set $\beta:=\beta_{\bar\delta}$ and $\gamma:=\beta_{\delta}$, so that $(\beta,\gamma)\in[B]^2$.
Note that
$$\Lambda<\alpha<\epsilon'\le\bar\alpha<\xi\le\beta<\delta\le\zeta\le\gamma.$$
By Claim~\ref{claim5251}, then:
\begin{itemize}
\item $\Tr(\alpha,\gamma)\restriction(k+1)=\Tr(\delta,\gamma)\restriction(k+1)=\Tr(\beta,\gamma)\restriction(k+1)$;
\item $\Tr(\alpha,\gamma)(k)=\zeta=\Tr(\beta,\gamma)(k)$;
\item $\Tr(\alpha,\beta)(k)=\xi$.
\end{itemize}

Therefore, $$\Tr(\alpha,\gamma)(k+1)=\min(C_\zeta\setminus\alpha)\le\bar\alpha<\beta\le\min(C_\zeta\setminus\beta)=\Tr(\beta,\gamma)(k+1),$$
and $\chi(\alpha,\beta,\gamma)=k$.

Summing all up,
$(\alpha,\beta,\gamma)\in A\circledast B\circledast B$, and
$$\begin{aligned}\overline{\osc}(\alpha,\beta,\gamma)=&\ \osc_\alpha(C_{\Tr(\alpha,\beta)(\chi(\alpha,\beta,\gamma))},C_{\Tr(\alpha,\gamma)(\chi(\alpha,\beta,\gamma))})\\
=&\ \osc_\alpha(C_{\Tr(\alpha,\beta)(k)},C_{\Tr(\alpha,\gamma)(k)})\\
=&\ \osc_\alpha(C_\xi,C_{\zeta})=n,\end{aligned}$$
as sought.
\end{proof}

The ending of the proof of Claim~\ref{claim6122} makes it clear that the following hold.
\begin{obs}\label{obs on chi} Suppose:
\begin{itemize}
\item $\lambda_2(\delta,\gamma)<\alpha<\beta<\delta<\gamma<\kappa$;
\item $C_{\last{\delta}{\gamma}}\cap[\alpha,\beta)$ is nonempty.
\end{itemize}

Then $\chi(\alpha,\beta,\gamma)= \rho_2(\last{\delta}{\gamma},\gamma)$. \qed
\end{obs}

The next lemma extracts features present in the proof of \cite[Lemma~10.3.2]{MR2355670}.

\begin{lemma} \label{equivlemma}
Suppose that $X$ is a stable cofinal subset of some $\theta\le\kappa$ of uncountable cofinality.
Then there exist a cofinal $X'\s X$, a club $D\s\theta$,
and a positive integer $k$ satisfying all of the following:	
\begin{enumerate}
\item for every $(\delta,\gamma)\in D\circledast X'$, $D\cap\delta\s C_{\last{\delta}{\gamma}}$;
\item for every $(\delta,\alpha,\delta',\beta,\delta'',\gamma)\in D\circledast\theta\circledast D\circledast\theta\circledast D\circledast X'$:
\begin{itemize}
\item $\chi(\alpha,\beta,\gamma)=k$, and
\item $D\cap\delta''\s C_{\Tr(\alpha,\gamma)(k)}$.
\end{itemize}
\end{enumerate}
\end{lemma}
\begin{proof} Set $\mu:=\cf(\theta)$, and fix a map $\psi:\mu\rightarrow\theta$ whose image is cofinal in $\theta$.
Let $\mathcal M$ be a continuous $\in$-chain of length $\mu$ consisting of elementary submodels $M\prec  H_\Upsilon$ with $M\cap\mu\in\mu$ and $\{\psi,\vec C,X\}\in M$.
For each $M\in\mathcal M$, denote $\theta_M:=\sup(M\cap\theta)$,
so that $E:=\{\theta_M\mid M\in\mathcal M\} $ is a club in $\theta$.
\begin{claim}\label{claim1} Let $N\prec\mathcal{H}_\Upsilon $ be such that $\{\psi,\vec C,X,\mathcal M\}\in N$ and $N\cap\mu\in\mu$.
Denote $\theta_N:=\sup(N\cap\theta)$ and let $\gamma\in X\setminus(\theta_N+1)$.\footnote{As $\cf(\theta)=\mu>|N|$, $X\setminus\theta_N$ is co-bounded in $X$.}

If $\sup(E\cap\theta_{N}\setminus C_{\last{\theta_N}{\gamma}})=\theta_N$,
then there exists $\gamma'\in X\setminus(\theta_N+1)$ with $\rho_2(\last{\theta_N}{\gamma'},\gamma')>\rho_2(\last{\theta_N}{\gamma},\gamma)$.
\end{claim}
\begin{proof} 
Denote $\bar\gamma:=\last{\theta_N}{\gamma}$, $n:=\rho_2(\bar\gamma,\gamma)$, and
$$\mathcal M^\gamma:=\{ M\in\mathcal M\cap N\mid  \theta_M\notin C_{\bar{\gamma}}\}.$$
Now, assuming that $\sup(E\cap\theta_{N}\setminus C_{\bar\gamma})=\theta_N$,
we infer that 
$$\sup\{ \theta_M\mid M\in\mathcal M^\gamma\}=\theta_N.$$

Pick $M\in\mathcal M^\gamma$ with $\theta_M>\lambda_2(\theta_N,\gamma)$.
Since $\theta_M\notin C_{\bar\gamma}$, it is the case that $\rho_2(\theta_M,\bar\gamma)>1$.
So, by Remark~\ref{concatenation},
$$\tr(\theta_M,\gamma)=\tr(\bar\gamma,\gamma){}^\smallfrown\tr(\theta_M,\bar\gamma),$$
and
$$\rho_2(\theta_M,\gamma)>\rho_2(\bar\gamma,\gamma)+1.$$

Thus, $\im(\tr(\theta_M,\gamma))$ contains not only $\im(\tr(\bar{\gamma},\gamma))$ but also $\min(C_{\bar{\gamma}}\setminus\theta_M)$ which is strictly above $\theta_M$, 
therefore, $\rho_2(\last{\theta_M}{\gamma},\gamma)>n$.
Pick $\Lambda\in M\cap\theta$ above $\lambda_2(\last{\theta_M}{\gamma},\gamma)$.   
Then, the following set belongs to $M$ and $\gamma$ witnesses that it has $\theta_M$ as an element:
$$S:=\{ \delta<\theta\mid \exists \gamma'\in X\setminus(\theta+1)\,[\rho_2(\last{\delta}{\gamma'},\gamma')>n\ \&\ \lambda_2(\delta,\gamma')\le\Lambda]\},$$
so that $S$ is stationary in $\theta$.
Pick $\delta\in S$ above $\theta_N$, along with a witnessing $\gamma'$. As $\lambda_2(\delta,\gamma')\le\Lambda<\theta_M<\theta_N<\delta$, we get from Remark~\ref{concatenation} that
$$\tr(\theta_N,\gamma')=\tr(\last{\delta}{\gamma'},\gamma'){}^\smallfrown\tr(\theta_N,\last{\delta}{\gamma'}),$$
and hence $\rho_2(\theta_N,\gamma')\ge \rho_2(\last{\delta}{\gamma'},\gamma')+1>n+1$.
Therefore $\rho_2(\last{\theta_N}{\gamma'},\gamma')>n$, as sought.
\end{proof}
\begin{claim} $\Delta:=\{ \delta<\theta\mid \exists\gamma\in X\setminus(\delta+1)\,[E\cap\delta\s^*C_{\last{\delta}{\gamma}}]\}$ covers a club in $\theta$.
\end{claim}
\begin{proof} Let $S$ be a stationary subset of $\theta$, and we shall prove that $S\cap\Delta\neq\emptyset$.
Let $N\prec\mathcal{H}_\Upsilon $ be such that $\{\psi,\vec C,X,\mathcal M\}\in N$, $N\cap\mu\in\mu$,
with $\theta_N:=\sup(N\cap\theta)$ in $S$.
Using the fact that $X$ is stable, fix $m<\omega$ such that $\chi``[X]^3\s m$.
Now, if $\theta_N\notin\Delta$, then by iterating Claim~\ref{claim1} finitely many times,
we may find a $\gamma\in X\setminus(\theta_N+1)$ such that $\rho_2(\last{\theta_N}{\gamma},\gamma)>m$.
Pick $\alpha,\beta\in X$ with $\lambda_2(\theta_N,\gamma)<\alpha<\beta<\theta_N$ such that $C_{\last{\theta_N}{\gamma}}\cap(\alpha,\beta)\neq\emptyset$. 
By Observation~\ref{obs on chi}, $\chi(\alpha,\beta,\gamma)>m$. This is a contradiction. 
\end{proof}

For each $\delta\in\Delta$, fix $\gamma_\delta\in X\setminus(\delta+1)$, $\epsilon_\delta<\delta$ and $k_\delta<\omega$ such that:
\begin{itemize}
\item $E\cap\delta\setminus\epsilon_\delta\s C_{\last{\delta}{\gamma_\delta},\gamma_\delta}$;
\item $\lambda_2(\delta,\gamma_\delta)\le\epsilon_\delta$;
\item $\rho_2(\last{\delta}{\gamma_\delta},\gamma_\delta)=k_\delta$.
\end{itemize}

Find $\epsilon<\theta$ and $k<\omega$ for which the following set is stationary:
$$S:=\{\delta\in\Delta\mid \epsilon_\delta=\epsilon\ \&\ k_\delta=k\}.$$

Consider the club $D:=\{\delta\in\acc(E\setminus\epsilon)\mid \forall \bar\delta\in\Delta\cap\delta\,(\gamma_{\bar\delta}<\delta)\}$,
and the set $X':=\{\gamma_\delta\mid\delta\in S\cap D\}$.
We shall verify that $X'$, $D$ and $k$ satisfy the requirements of the two clauses. 

(1) Let $\delta\in D$ and $\gamma\in X'\setminus(\delta+1)$. Pick $\delta^*\in D$ such that $\gamma=\gamma_{\delta^*}$. 
If $\delta^*=\delta$, then $D\cap\delta^*\s E\cap\delta\setminus\epsilon\s C_{\last{\delta^*}{\gamma}}$.
Otherwise,  $\lambda_2(\delta^*,\gamma)\le\epsilon<\delta<\delta^*<\gamma=\gamma_{\delta^*}$. So, by Remark~\ref{concatenation},
$$\tr(\delta,\gamma)=\tr(\last{\delta^*}{\gamma},\gamma){}^\smallfrown\tr(\delta,\last{\delta^*}{\gamma}),$$
with $E\cap\delta^*\setminus\epsilon\s C_{\last{\delta^*}{\gamma}}$. Since $\delta\in\acc(E)\cap(\epsilon,\delta^*)$,
it is the case that $\delta\in\acc(C_{\last{\delta^*}{\gamma}})$. Altogether, $$\tr(\last{\delta}{\gamma},\gamma)=\tr(\last{\delta^*}{\gamma},\gamma).$$
In particular, $D\cap\delta\s E\cap(\epsilon,\delta^*)\s C_{\last{\delta^*}{\gamma}}$.

(2)  Let $(\delta,\alpha,\delta',\beta,\delta'',\gamma)\in D\circledast\theta\circledast D\circledast\theta\circledast D\circledast X'$.
Fix $\delta^*\in S\cap D$ such that $\gamma=\gamma_{\delta^*}$. 
As $\delta^*\in\Delta$ and $\gamma_{\bar\delta}<\delta''<\gamma$ for every $\bar\delta\in\Delta\cap\delta''$, it must be the case that $\delta^*\ge\delta''$.
So $\lambda_2(\delta^*,\gamma)\le\epsilon<\alpha<\beta<\delta^*<\gamma$ and $\delta'\in D\cap(\alpha,\beta)\s C_{\last{\delta^*}{\gamma}}\cap(\alpha,\beta)$.
Then, Observation~\ref{obs on chi} implies that $\chi(\alpha,\beta,\gamma)=\rho_2(\last{\delta^*}{\gamma},\gamma)=k$.
\end{proof}
\begin{remark} While we will not be needing this fact, we point out that the proofs of Lemmas  \ref{unstable case} and \ref{equivlemma}
together show that for every $\theta\le\kappa$ of uncountable cofinality,
and every cofinal subset $X\s\theta$,
the following are equivalent:
\begin{itemize}
\item Every cofinal subset of $X$ is unstable;
\item For every stationary $S\s\theta$ and every $\langle \beta_\delta\mid\delta\in S\rangle$ in ${\prod_{\delta\in S}(X\setminus(\delta+1))}$, 
the set $\Gamma:=\{ \last{\delta}{\beta_\delta}\mid \delta\in S\}$ is $\vec C$-nontrivial.
\end{itemize}
\end{remark}

The following is an easy strengthening of \cite[Lemma~10.3.2]{MR2355670}: 

\begin{lemma}[Todor\v{c}evi\'{c}]\label{stable case_relaxed}
Suppose that $\kappa=\lambda^+$ for some regular uncountable cardinal $\lambda$,
and let $\rho:[\kappa]^2\rightarrow\lambda$ be the corresponding map given by Fact~\ref{all the rho}.

Suppose also that $X$ is a stable subset of $\kappa$ of order-type $\lambda$. Then there exists a cofinal subset $X'\s X$ such that $\rho\restriction[X']^2$ witnesses $\U(\lambda,\lambda,\lambda,\omega)$. 
\end{lemma}
\begin{proof} Denote $\theta:=\sup(X)$. 
Let $X'\s X$ and $D\s\theta$ be given by Lemma~\ref{equivlemma}.
To see that $\rho\restriction[X']^2$ witnesses $\U(\lambda,\lambda,\lambda,\omega)$,
let $\mathcal A\in[X']^\sigma$ be some $\lambda$-sized pairwise disjoint, with $\sigma<\omega$,
and let $\tau<\lambda$.

As $\mathcal A$ consists of $\lambda$-many pairwise disjoint subsets of $X'$ and $\otp(X')=\lambda$, for every $\delta\in D$, we may pick $b_\delta\in\mathcal A$ with $\min(b_\delta)>\delta$.
Then put $\Lambda_\delta:=\max\{ \lambda_2(\delta,\beta)\mid \beta\in b_\delta\}$.
Recalling that $X'$ and $D$ were given by Lemma~\ref{equivlemma}, for all $\delta\in D$ and $\beta\in b_\delta$, 
$$D\cap\delta\s C_{\last{\delta}{\beta}}.$$ 

Next, fix some $\Lambda<\theta$ and a stationary set $S\s D$ such that for every $\delta\in S$:
\begin{enumerate}
\item $\Lambda_\delta\le\Lambda<\delta$,
\item $\otp(D\cap\delta)>\tau$, and
\item For every $\gamma\in D\cap\delta$, $\sup(b_\gamma)<\delta$.
\end{enumerate}

Evidently, $\mathcal B:=\{ b_\delta\mid \delta\in S\}$ is a $\lambda$-sized subset of $\mathcal A$.
Now, given $a\neq b$ in $\mathcal B$, we may find a pair $\gamma\neq\delta$ of ordinals from $S$ such that $a=b_\gamma$ and $b=b_\delta$.
Without loss of generality, $\gamma<\delta$.
Let $(\alpha,\beta)\in a\times b$. Then 
$$\Lambda_\delta<\gamma<\alpha<\delta<\beta,$$
so by Remark~\ref{concatenation},
$\last{\delta}{\beta}\in\im(\tr(\alpha,\beta))$.
As the map $\rho$ was given by Fact~\ref{all the rho}, 
$$\rho(\alpha,\beta)\ge \otp(C_{\last{\delta}{\beta}}\cap\alpha)\ge \otp(C_{\last{\delta}{\beta}}\cap\gamma)\ge\otp(D\cap\gamma)>\tau,$$
as sought.
\end{proof}

It is clear that every subset of a stable set is stable. The next corollary addresses the question of closure under unions.

\begin{cor}\label{corofstable}
Suppose that $A,B$ are cofinal stable subsets of some $\theta\le \kappa$ of uncountable cofinality.
Then, there exist cofinal subsets $A'\s A$ and $B'\s B$ such that:
\begin{itemize}
\item $A'\cup B'$ is stable, and
\item for every $(\alpha,\beta,\gamma)\in[A'\cup B']^3$, if $(\beta,\gamma)\in[A']^2\cup[B']^2$, then $\overline\osc(\alpha,\beta,\gamma)=0$.
\end{itemize}
\end{cor}
\begin{proof} Appeal to Lemma~\ref{equivlemma} with $A$ to get a cofinal $A'\s A$, a club $D_1\s\theta$ and an integer $k_1$.
Likewise, appeal to Lemma~\ref{equivlemma} with $B$ to get a cofinal $B'\s B$, a club $D_2\s\theta$ and an integer $k_2$.
Consider the club $D:=D_1\cap D_2$.
Pick sparse enough cofinal subsets $A''\s A'$ and $B''\s B'$ such that, letting $X:=A''\cup B''$,
for every $(\alpha,\beta)\in[X]^2$, there are $\iota<\theta$ and $\delta,\delta',\delta''\in D$ with $\delta<\alpha<\delta'<\iota<\delta''<\beta$.

\begin{claim} $X$ is stable. Furthermore, $\max(\chi``[X]^3)=k_2$. 
\end{claim}
\begin{proof} Let $(\alpha,\beta,\gamma)\in[X]^3$.  
Pick $\delta,\delta',\delta''$ such that
$$(\delta,\alpha,\delta',\beta,\delta'',\gamma)\in D\circledast X\circledast D\circledast X\circledast D\circledast X.$$
Then $\chi(\alpha,\beta,\gamma)=k_1$ if $\gamma\in A'$, and $\chi(\alpha,\beta,\gamma)=k_2$ otherwise.
\end{proof}
\begin{claim} Let $Y\in\{A'',B''\}$ and
$(\alpha,\beta,\gamma)\in X\circledast Y\circledast Y$.
Then $\overline\osc(\alpha,\beta,\gamma)=0$.
\end{claim}
\begin{proof} If $Y=A''$, then denote $k:=k_1$. Otherwise, denote $k:=k_2$.
Now, pick $\iota,\delta,\delta',\delta'',\delta'''$ such that
$$(\delta,\alpha,\delta',\iota,\delta'',\beta,\delta''',\gamma)\in D\circledast X\circledast D\circledast\theta\circledast D\circledast Y\circledast D\circledast Y.$$
In particular,
$$\{(\delta,\alpha,\delta',\iota,\delta'',\beta),(\delta,\alpha,\delta'',\beta,\delta''',\gamma)\}\s D\circledast\theta\circledast D\circledast\theta\circledast D\circledast X.$$
Then $D\cap\delta''\s C_{\Tr(\alpha,\beta)(k)}$ and $D\cap\delta'''\s C_{\Tr(\alpha,\gamma)(k)}$.
In addition, $\chi(\alpha,\beta,\gamma)=k$, so that 
$$\delta'\in D\cap(\alpha,\delta'')\s C_{\Tr(\alpha,\beta)(\chi(\alpha,\beta,\gamma))}\cap C_{\Tr(\alpha,\gamma)(\chi(\alpha,\beta,\gamma))}.$$
Recalling Definition~\ref{2dim-osc}, this means that 
$$\Osc_\alpha(C_{\Tr(\alpha,\beta)(\chi(\alpha,\beta,\gamma))},C_{\Tr(\alpha,\gamma)(\chi(\alpha,\beta,\gamma))})=\emptyset,$$
and hence $\overline\osc(\alpha,\beta,\gamma)=0$.
\end{proof}
So $A''$ and $B''$ are as sought.
\end{proof}

\begin{cor}\label{thm712} Suppose that:
\begin{enumerate}
\item $\mu<\lambda<\lambda^+=\kappa$ are infinite regular cardinals;
\item $E^\lambda_\mu$ admits a nonreflecting stationary set;
\item there exists a weak $\mu$-Kurepa tree with at least $\kappa$-many branches.
\end{enumerate}
Then $\kappa\nsrightarrow[\lambda,\lambda]^3_\omega$ holds. 
\end{cor}
\begin{proof} Let $c:[\kappa]^3\rightarrow\lambda$ be the map given by Theorem~\ref{use of stability}
with respect to the subadditive coloring $\rho:[\kappa]^2\rightarrow\lambda$ of Fact~\ref{all the rho}.
Define $c_\omega:[\kappa]^3\rightarrow\omega$ via
$$c_\omega(\alpha,\beta,\gamma):=\begin{cases}c(\alpha,\beta,\gamma),&\text{if }c(\alpha,\beta,\gamma)<\omega;\\
0,&\text{otherwise}.
\end{cases}$$

Let $T\s{}^{<\mu}2$ be a weak $\mu$-Kurepa tree with at least $\kappa$-many branches,
and let $\langle b_\xi \mid \xi< \kappa\rangle$ be an injective sequence consisting of elements of $\mathcal B(T)$.
For notational simplicity, we shall write $\Delta(\alpha,\beta)$ for $\Delta(b_\alpha,b_\beta)$.
Define a coloring $d:[\kappa]^3\rightarrow\omega$ by letting for all $\alpha<\beta<\gamma<\kappa$:
$$d(\alpha,\beta,\gamma):=\begin{cases}
\overline{\osc}(\alpha,\beta,\gamma)-1,&\text{if }\overline\osc(\alpha,\beta,\gamma)>0\text{ and }\Delta(\alpha,\beta)<\Delta(\beta,\gamma);\\
c_\omega(\alpha,\beta,\gamma),&\text{otherwise}.
\end{cases}$$

To see that $d$ witnesses $\kappa\nsrightarrow[\lambda,\lambda]^3_\omega$,
let $A,B$ be disjoint subsets of $\kappa$ of order-type $\lambda$ with $\sup(A)=\sup(B)$,
and let $n<\omega$. 
Using Lemma~\ref{separation lemma2} and by possibly passing to cofinal subsets,
we may assume the existence of $s\in T$ and $i\neq i'$ such that $s{}^\smallfrown\langle i\rangle\sq b_\alpha$ for all $\alpha\in A$,
and $s{}^\smallfrown\langle i'\rangle\sq b_\beta$ for all $\beta\in B$.
\begin{claim} Let $(\alpha,\beta,\gamma)\in [A\cup B]^3\setminus([A]^3\cup[B]^3)$.

Then $(\beta,\gamma)\in([A]^2\cup[B]^2)$ iff $\Delta(\alpha,\beta)<\Delta(\beta,\gamma)$.
\end{claim}
\begin{proof} For every $(\epsilon,\delta)\in (A\circledast B)\cup(B\circledast A)$, $\Delta(\epsilon,\delta)=\Delta(s{}^\smallfrown\langle i\rangle,s{}^\smallfrown\langle i'\rangle)=\dom(s)$.
For every $(\epsilon,\delta)\in (A\circledast A)\cup(B\circledast B)$, $\Delta(\epsilon,\delta)\ge \dom(s+1)>\dom(s)$. 

By the hypothesis on $(\alpha,\beta,\gamma)$, there are three cases to consider:

$\br$ If $(\alpha,\beta,\gamma)\in A\circledast B\circledast B$, then $\Delta(\alpha,\beta)=\dom(s)<\Delta(\beta,\gamma)$.

$\br$ If $(\alpha,\beta,\gamma)\in B\circledast A\circledast A$, then $\Delta(\alpha,\beta)=\dom(s)<\Delta(\beta,\gamma)$.

$\br$ If $(\beta,\gamma)\in (A\circledast B)\cup(B\circledast A)$, then $\Delta(\beta,\gamma)=\dom(s)\le\Delta(\alpha,\beta)$.
\end{proof}

There are three cases to consider:
\begin{itemize}
\item[$\br$] Suppose that there exist $A'\in[A]^\lambda$ and $B'\in[B]^\lambda$ such that $A'$ and $B'$ are stable.
Then, by Corollary~\ref{corofstable}, we may moreover assume that $A'\cup B'$ is stable,
and that for every $(\alpha,\beta,\gamma)\in[A'\cup B']^3$ with $(\beta,\gamma)\in[A']^2\cup[B']^2$, $\overline\osc(\alpha,\beta,\gamma)=0$.
By Lemma~\ref{stable case_relaxed}, then, $\rho\restriction[A'\cup B']^2$ witnesses $\U(\lambda,\lambda,\lambda,3)$. 
As $c:[\kappa]^3\rightarrow\lambda$ was given by Theorem~\ref{use of stability}, 
we may find $(\alpha,\beta,\gamma)\in[A'\cup B']^3\setminus([A']^3\cup[B']^3)$ such that
$c(\alpha,\beta,\gamma)=n$. Now, if $(\beta,\gamma)\in[A']^2\cup[B']^2$, then $\overline\osc(\alpha,\beta,\gamma)=0$,
and if $(\beta,\gamma)\notin [A']^2\cup[B']^2$, then $\Delta(\alpha,\beta)\ge \Delta(\beta,\gamma)$.
It thus follows that $d(\alpha,\beta,\gamma)=c_\omega(\alpha,\beta,\gamma)=c(\alpha,\beta,\gamma)=n$, as sought. 

\item[$\br$]     Suppose that every cofinal subset of $B$ is unstable. 
By appealing to Lemma~\ref{unstable case} with $A$ and $B$, we may find 
$(\alpha,\beta,\gamma)\in A\circledast B\circledast B$ such that $\overline{\osc}(\alpha,\beta,\gamma)=n+1$. 	
As $(\beta,\gamma)\in[B]^2$, it is the case that $\Delta(\alpha,\beta)<\Delta(\beta,\gamma)$.
So $d(\alpha,\beta,\gamma)=\overline{\osc}(\alpha,\beta,\gamma)-1=n$.

\item[$\br$]    Otherwise. So, every cofinal subset of $A$ is unstable, and then the argument is similar to that of the previous case. \qedhere
\end{itemize}
\end{proof}

We are finally in conditions to prove the rectangular extension of \cite[Theorem~10.3.6]{MR2355670}.

\begin{cor}\label{coloring with CH} Suppose that $\lambda=\mu^+$ for an infinite cardinal $\mu=\mu^{<\mu}$. 

Then $\lambda^+\nsrightarrow[\lambda,\lambda]^3_\omega$ holds. 
\end{cor}
\begin{proof} Denote $\kappa:=\lambda^+$. If $2^\mu>\mu^+$, then since $\mu^{<\mu}=\mu$, there is a weak $\mu$-Kurepa tree with $\kappa$-many branches,
and $E^\lambda_\mu$ constitutes a nonreflecting stationary set.
So, by Corollary~\ref{thm712}, we may assume here that $2^\mu=\mu^+$.  In particular, $\lambda\nrightarrow[\mu;\lambda]^2_\lambda$ holds by a theorem of Sierpi\'nski.
Let  $c:[\kappa]^3\rightarrow\lambda$ be the map given by Theorem~\ref{thm53}
with respect to the subadditive coloring $\rho:[\kappa]^2\rightarrow\lambda$ of Fact~\ref{all the rho}.
Derive $c_\omega$ from $c$ as in the proof of Theorem~\ref{thm712}.
Also, denote $T:={}^{<\lambda}2$, 
and let $\langle b_\xi\mid\xi<\kappa\rangle$ be the injective sequence of elements of $\mathcal B(T)$ used in the proof of
Theorem~\ref{thm53} to define the coloring $c$.
For notational simplicity, we shall write $\Delta(\alpha,\beta)$ for $\Delta(b_\alpha,b_\beta)$.
Likewise, for $B\s\kappa$, we write $T^{\branches B}$ for $T^{\branches \{b_\beta\mid \beta\in B\}}$.

Finally, define a coloring $d:[\kappa]^3\rightarrow\omega$ by letting for all $\alpha<\beta<\gamma<\kappa$:\footnote{It may appear that this is the same map from the proof of Corollary~\ref{thm712}. Note, however, that there $\Delta$ was a map from $[\kappa]^2$ to $\mu$, whereas here $\Delta$ is a map from $[\kappa]^2$ to $\lambda$.}
$$d(\alpha,\beta,\gamma):=\begin{cases}
\overline{\osc}(\alpha,\beta,\gamma)-1,&\text{if }\overline\osc(\alpha,\beta,\gamma)>0\text{ and }\Delta(\alpha,\beta)<\Delta(\beta,\gamma);\\
c_\omega(\alpha,\beta,\gamma),&\text{otherwise}.
\end{cases}$$

To see that $d$ witnesses $\kappa\nsrightarrow[\lambda,\lambda]^3_\omega$,
let $A,B$ be disjoint subsets of $\kappa$ of order-type $\lambda$ with $\sup(A)=\sup(B)$,
and let $n<\omega$. 
Using Corollary~\ref{corofstable} and by possibly passing to cofinal subsets,
we may assume that one of the following holds:
\begin{itemize}
\item[(I)] $A\cup B$ is stable,
and for every $(\alpha,\beta,\gamma)\in[A\cup B]^3$ with $(\beta,\gamma)\in[A]^2\cup[B]^2$, $\overline\osc(\alpha,\beta,\gamma)=0$;
\item[(II)] every cofinal subset of $B$ is unstable;
\item[(III)] every cofinal subset of $A$ is unstable.
\end{itemize}

Let us dispose of Case~(I) right away.
In this case, by Lemma~\ref{stable case_relaxed}, $\rho\restriction[A\cup B]^2$ witnesses $\U(\lambda,\lambda,\lambda,3)$. 
So by the choice of $c$,
we may find $(\alpha,\beta,\gamma)\in[A\cup B]^3\setminus([A]^3\cup[B]^3)$ such that
$c(\alpha,\beta,\gamma)=n$. Going over the division into cases in the proof of Theorem~\ref{thm53},
we see that $\Delta(\alpha,\beta)\ge\Delta(\beta,\gamma)$ in all subcases but to Subcase~1.1.
So, in all of these cases, $d(\alpha,\beta,\gamma)=c_\omega(\alpha,\beta,\gamma)=c(\alpha,\beta,\gamma)=n$.
Finally, looking at Subcase~1.1, we see that the provided triple $(\alpha,\beta,\gamma)$ is an element of $A\circledast B\circledast B$ (or an element of $B\circledast A\circledast A$, once lifting the initial ``without loss of generality'' assumption).
So $(\beta,\gamma)\in[A]^2\cup[B]^2$, and hence $\overline\osc(\alpha,\beta,\gamma)=0$.
Therefore, again $d(\alpha,\beta,\gamma)=c_\omega(\alpha,\beta,\gamma)=n$.

\medskip

Moving on to handling Cases (II) and (III), we shall need the following claim.
\begin{claim} There are cofinal subsets $A'\s A$ and $B'\s B$ such that, for every $(\alpha,\beta,\gamma)\in(A\circledast B\circledast B)\cup(B\circledast A\circledast A)$, $\Delta(\alpha,\beta)<\Delta(\beta,\gamma)$.
\end{claim}
\begin{proof} By possibly passing to a cofinal subset of $B$,
we may assume that $B=B'$ in the sense of Lemma~\ref{height lemma}.
So let $\theta_B\le\lambda$ be
such that $T^{\branches B'}$ is a normal tree in $\mathcal T(\lambda,\theta_B)$ for every $B'\in[B]^\lambda$.
Likewise, we may assume that $A=A'$ in the sense of Lemma~\ref{height lemma},
and let $\theta_A\le\lambda$ be such that $T^{\branches A'}$ is a normal tree in $\mathcal T(\lambda,\theta_A)$ for every $A'\in[A]^\lambda$.

If $\min\{\theta_A,\theta_B\}<\lambda$,
then the proof of Claim~\ref{claim532} provides $\chi<\min\{\theta_A,\theta_B\}$ and a pair $(t,t')\in (T^{\branches A})_{\chi+1}\times (T^{\branches B})_{\chi+1}$
such that $\Delta(t,t')=\chi$. Thus letting $A':=A_t$ and $B':=B_{t'}$,
we see that $\Delta(\alpha,\beta)=\chi$ whenever $(\alpha,\beta)\in (A'\circledast B')\cup (B'\circledast A')$,
and $\Delta(\alpha,\beta)>\chi$ whenever $(\alpha,\beta)\in (A'\circledast A')\cup (B'\circledast B')$.
Thus, we may assume that $\theta_A=\theta_B=\lambda$, so that $T^{\branches A},T^{\branches B}\in \mathcal T(\lambda,\lambda)$.
Now, if $(T^{\branches A})\nsubseteq(T^{\branches B})$ and $(T^{\branches A})\nsubseteq(T^{\branches B})$, then 
by normality of the two trees there must exist $\chi<\lambda$,
$t\in (T^{\branches A})_\chi\setminus (T^{\branches B})_\chi$ and $t'\in (T^{\branches B})_\chi\setminus(T^{\branches A})_\chi$.
Clearly, $A':=A_t$ and $B':=B_{t'}$ are as sought.

Thus, the only nontrivial case is in which $\theta_A=\theta_B=\lambda$
and $T^{\branches A}\s T^{\branches B}$ or $T^{\branches A}\s T^{\branches B}$.
Without loss of generality, assume that $T^{\branches A}\s T^{\branches B}$. Now, there are two options:

$\br$ If $\mathcal B(T^{\branches A})$ is nonempty,
then there is $b:\lambda\rightarrow2$ that constitutes a branch through both $T^{\branches A}$ and $T^{\branches B}$.
In this case, it is easy to recursively simultaneously construct cofinal subsets $A'\s A$ and $B'\s B$ such that
for all triple $\alpha<\beta<\gamma$ of ordinals from $A'\cup B'$, it is the case that $\Delta(\alpha,\beta)<\Delta(\beta,\gamma)$.

$\br$ Otherwise. So $T^{\branches A}$ is a $\lambda$-Aronszajn tree.
In particular, we may find $\chi<\lambda$ and $t\neq t'$ in $(T^{\branches A})_\chi$.
Altogether, $t\in (T^{\branches A})_\chi$, $t'\in (T^{\branches B})_\chi$ and $\Delta(t,t')<\chi$.
Then $A':=A_t$ and $B':=B_{t'}$ are as sought.
\end{proof}

At this point, the proof is similar to that of Theorem~\ref{thm712}.
Succinctly, in Case~(II), 
we appeal to Lemma~\ref{unstable case} with $A$ and $B$, to find 
$(\alpha,\beta,\gamma)\in A\circledast B\circledast B$ such that $\overline{\osc}(\alpha,\beta,\gamma)=n+1$. 	
By the preceding claim, it is the case that $\Delta(\alpha,\beta)<\Delta(\beta,\gamma)$.
So $d(\alpha,\beta,\gamma)=\overline{\osc}(\alpha,\beta,\gamma)-1=n$.
The handling of Case~(III) is similar.
\end{proof}

\section{Connecting the dots}\label{sectionTHMB}

The next result implies Theorem~A$'$.
\begin{thm}
For every infinite cardinal $\mu$ satisfying $\mu^{<\mu}<\mu^+<2^\mu$:
\begin{enumerate}
\item $S_3(\mu^{++} , \mu^+, \omega )$ holds;
\item  $G\nrightarrow	[\mu^+]^{\fs_3}_\omega$ holds for every Abelian group $(G,+)$ of size $\mu^{++}$.
\end{enumerate}
\end{thm}
\begin{proof} Suppose that $\mu$ is an infinite cardinal satisfying $\mu^{<\mu}<\mu^+<2^\mu$.
Denote $\lambda:=\mu^+$ and $\kappa:=\lambda^+$.
Then, $T:={}^{<\mu}2$ is a weak $\mu$-Kurepa tree with at least $\kappa$-many branches.
So, by Corollary~\ref{thm712}, $\kappa\nsrightarrow[\lambda,\lambda]^3_\omega$ holds.
As $T$ witnesses that $\kappa\in T(\lambda,\mu)$,
by Lemma~\ref{ext3},  $\ext_3(\kappa,\lambda,\mu,\omega)$ holds, as well.
Together with Lemma~\ref{pilemma}, this yields Clause~(1). Then, Clause~(2) follows from Proposition~\ref{prop316} and the fact (see \cite[Lemma~2.2]{paper27}) that every Abelian group is a well-behaved magma.
\end{proof}

By \cite[Theorem~3.8]{paper27}, for every infinite cardinal $\mu=2^{<\mu}$,
$S_2(2^\mu,2^\mu,\omega)$ holds. By the upcoming corollary, for every infinite cardinal $\mu=2^{<\mu}$, $S_2(2^\mu,\mu^+,2)$ holds.
While it is easy to get $S_2(2^\mu,\mu^+,\mu)$ from $2^\mu=\mu^+$,
Remark~\ref{Fleissner} shows that $S_2(2^\mu,\mu^+,\mu)$ is also compatible with $2^\mu>\mu^+$.
Note, however, that by a theorem of Shelah \cite[Theorem~2.1]{MR955139}, one cannot prove $S_2(2^\mu,\mu^+,3)$ in $\zfc$.
Thus, in view of the number of colors, the following corollary is optimal.

\begin{cor} For every infinite cardinal $\mu$, $S_2(\mu^\theta,\mu^+,2)$ holds, for $\theta:=\log_\mu(\mu^+)$.

In particular, $S_2(2^\mu,\mu^+,\allowbreak2)$ holds for every strong limit cardinal $\mu$.
\end{cor}
\begin{proof} Appeal to the upcoming theorem with $\lambda:=\mu^+$ and $\kappa:=\mu^\theta$.
\end{proof}

\begin{lemma}\label{thm73} Suppose that $\theta<\lambda\le\kappa$ are infinite cardinals, with $\lambda$ being regular.

If $\kappa \in T(\lambda,\theta)$, then:
\begin{enumerate}
\item $\kappa\nsrightarrow[\lambda,\lambda]^2_2$ holds;
\item $S_2(\kappa,\lambda,2)$ holds.
\end{enumerate}
\end{lemma}
\begin{proof} Suppose that $\kappa\in T(\lambda,\theta)$, and fix $T\in\mathcal T(\lambda,\theta)$ 
admitting an injective sequence $\langle b_\xi\mid \xi<\kappa\rangle$ consisting of elements of $\mathcal B(T)$.

(1) Consider the Sierpi\'{n}ski map $c:[\kappa]^2\rightarrow2$ defined by letting, for all $\alpha<\beta<\kappa$:
$$c(\alpha,\beta):=1\text{ iff }b_\alpha<_{\lex}b_\beta.$$
\begin{claim}\label{claim721} $c$ witnesses $\kappa\nsrightarrow[\lambda,\lambda]^2_2$.
\end{claim}
\begin{proof} Suppose that we are given two disjoint subsets $A,B$ of $\kappa$ with $\otp(A)=\otp(B)=\lambda$ and $\sup(A)=\sup(B)$.
By Lemma~\ref{separation lemma2}(2) (using $\mu:=\lambda$), we may find $s\in T$ and $i\neq i'$ such that 
$A':=\{ \alpha\in A\mid s{}^\smallfrown\langle i\rangle\sq b_\alpha\}$ and $B':=\{ \beta\in B\mid s{}^\smallfrown\langle i'\rangle\sq b_\beta\}$ are both of size $\lambda$.
As $\sup(A')=\sup(B')$, we may now fix $(\alpha,\beta,\gamma)\in A'\circledast B'\circledast A'$.
To see that $\{ c(\alpha,\beta),c(\beta,\gamma)\}=2$, consider the following cases:

$\br$ If $i<i'$, then $b_\alpha,b_{\gamma}<_{\lex} b_\beta$ and hence $c(\alpha,\beta)=1>0=c(\beta,\gamma)$;

$\br$ If $i'<i$, then $b_\beta<b_\alpha,b_{\gamma}$ and hence $c(\alpha,\beta)=0<1=c(\beta,\gamma)$.
\end{proof}

(2) By Lemma~\ref{extract2} (again, using $\mu:=\lambda$), in particular, $\ext_2(\kappa,\lambda,\omega,\omega)$ holds.
So, by Lemma~\ref{pilemma}, Clause~(1) implies Clause~(2).
\end{proof}

\begin{cor}\label{Theorem A} If there exists a weak $\mu$-Kurepa tree with $\kappa$-many branches,
then $S_2(\kappa,\mu^+,\allowbreak2)$ holds.
\qed
\end{cor}

Theorem~B now follows (using $\mu:=2$):
\begin{cor} For every infinite set $G$,
for every map $\varphi:G\rightarrow[G]^{<\omega}$,
and for every pair of cardinals $\mu,\theta$ such that $\mu^{<\theta}<|G|\le\mu^\theta$,
there exists a corresponding coloring $c:G\rightarrow 2$ satisfying the following.

For every binary operation $*$ on $G$, if $\varphi$ witnesses that $(G,*)$ is well-behaved,
then for every $X\s G$ of size $(\mu^{<\theta})^+$ and every $i\in\{0,1\}$, 
there are $x\neq y$ in $X$ such that $c(x*y)=i$.
\end{cor}
\begin{proof} Given $G$, $\varphi$, $\mu$ and $\theta$ as above, denote $\kappa:=|G|$ and $\lambda:=(\mu^{<\theta})^+$, so that $\lambda\le\kappa\le\mu^\theta$.
Evidently, $T:={}^{<\theta}\mu$ witnesses that $\kappa\in T(\lambda,\theta)$,
so $S_2(\kappa,\lambda,2)$ holds by Lemma~\ref{thm73}.
Now, appeal to Proposition~\ref{prop316}.
\end{proof}

\section{Acknowledgments}

The results of this paper stemmed from the first author's M.Sc.~thesis written under the supervision of the second author at Bar-Ilan University,
and supported by the Israel Science Foundation (grant agreement 2066/18).

The first author was partially supported by the Israel Science Foundation (grant agreement 203/22).
The second author was partially supported by the Israel Science Foundation (grant agreement 203/22)
and by the European Research Council (grant agreement ERC-2018-StG 802756).


\begin{thebibliography}{EHR65}

\bibitem[BR21]{paper23}
Ari~Meir Brodsky and Assaf Rinot.
\newblock A microscopic approach to {S}ouslin-tree constructions. {P}art {II}.
\newblock {\em Ann. Pure Appl. Logic}, 172(5):Paper No. 102904, 65, 2021.

\bibitem[EHR65]{MR202613}
P.~Erd\H{o}s, A.~Hajnal, and R.~Rado.
\newblock Partition relations for cardinal numbers.
\newblock {\em Acta Math. Acad. Sci. Hungar.}, 16:93--196, 1965.

\bibitem[FBR17]{paper27}
David Fernandez-Breton and Assaf Rinot.
\newblock Strong failures of higher analogs of {H}indman's theorem.
\newblock {\em Trans. Amer. Math. Soc.}, 369(12):8939--8966, 2017.

\bibitem[Fle78]{MR493930}
William~G. Fleissner.
\newblock Some spaces related to topological inequalities proven by the
{E}rd{\H o}s-{R}ado theorem.
\newblock {\em Proc. Amer. Math. Soc.}, 71(2):313--320, 1978.

\bibitem[HLS17]{MR3696151}
Neil Hindman, Imre Leader, and Dona Strauss.
\newblock Pairwise sums in colourings of the reals.
\newblock {\em Abh. Math. Semin. Univ. Hambg.}, 87(2):275--287, 2017.

\bibitem[HS12]{MR2893605}
Neil Hindman and Dona Strauss.
\newblock {\em Algebra in the {S}tone-\v {C}ech compactification}.
\newblock de Gruyter Textbook. Walter de Gruyter \& Co., Berlin, 2012.

\bibitem[IR23]{paper47}
Tanmay Inamdar and Assaf Rinot.
\newblock Was {U}lam right? {I}: {B}asic theory and subnormal ideals.
\newblock {\em Topology Appl.}, 323(C):Paper No. 108287, 53pp, 2023.

\bibitem[Kom16]{MR3511943}
P\'{e}ter Komj\'{a}th.
\newblock A certain 2-coloring of the reals.
\newblock {\em Real Anal. Exchange}, 41(1):227--231, 2016.

\bibitem[Kom20]{MR4143159}
P\'{e}ter Komj\'{a}th.
\newblock Weiss's question.
\newblock {\em Mathematika}, 66(4):954--958, 2020.

\bibitem[Kun80]{MR756630}
Kenneth Kunen.
\newblock {\em Set theory}, volume 102 of {\em Studies in Logic and the
Foundations of Mathematics}.
\newblock North-Holland Publishing Co., Amsterdam, 1980.
\newblock An introduction to independence proofs.

\bibitem[LHR18]{paper34}
Chris Lambie-Hanson and Assaf Rinot.
\newblock Knaster and friends {I}: {C}losed colorings and precalibers.
\newblock {\em Algebra Universalis}, 79(4):Art. 90, 39, 2018.

\bibitem[LHR23]{paper36}
Chris Lambie-Hanson and Assaf Rinot.
\newblock Knaster and friends {III}: {S}ubadditive colorings.
\newblock {\em J. Symbolic Logic}, 51pp, to appear 2023.
\newblock \verb"https://doi.org/10.1017/jsl.2022.50".

\bibitem[Mil78]{MR0505558}
Keith~R. Milliken.
\newblock Hindman's theorem and groups.
\newblock {\em J. Combin. Theory Ser. A}, 25(2):174--180, 1978.

\bibitem[Po{\'{o}}21]{MR4323604}
M{\'{a}}rk Po{\'{o}}r.
\newblock On the spectra of cardinalities of branches of {K}urepa trees.
\newblock {\em Arch. Math. Logic}, 60(7-8):927--966, 2021.

\bibitem[Rin14]{paper18}
Assaf Rinot.
\newblock Chain conditions of products, and weakly compact cardinals.
\newblock {\em Bull. Symb. Log.}, 20(3):293--314, 2014.

\bibitem[RZ21]{paper44}
Assaf Rinot and Jing Zhang.
\newblock Transformations of the transfinite plane.
\newblock {\em Forum Math. Sigma}, 9(e16):1--25, 2021.

\bibitem[She88]{MR955139}
Saharon Shelah.
\newblock Was {S}ierpi\'{n}ski right? {I}.
\newblock {\em Israel J. Math.}, 62(3):355--380, 1988.

\bibitem[SW16]{dani-bill}
Daniel Soukup and William Weiss.
\newblock Pairwise sums in colourings of the reals.
\newblock Unpublished note, 2016.
\newblock \verb"https://danieltsoukup.github.io/academic/finset_colouring.pdf".

\bibitem[Tod94]{MR1297180}
Stevo Todorcevic.
\newblock Some partitions of three-dimensional combinatorial cubes.
\newblock {\em J. Combin. Theory Ser. A}, 68(2):410--437, 1994.

\bibitem[Tod07]{MR2355670}
Stevo Todorcevic.
\newblock {\em Walks on ordinals and their characteristics}, volume 263 of {\em
Progress in Mathematics}.
\newblock Birkh\"{a}user Verlag, Basel, 2007.

\end{thebibliography}
\end{document}